\documentclass[11pt]{amsart}

\usepackage{amsmath,amssymb,latexsym}
\usepackage{dsfont}
\usepackage[T1]{fontenc}
\usepackage{enumerate}
\usepackage{eqnarray}
\usepackage{mathtools}
\usepackage{url}
\usepackage{xcolor}

\usepackage[a4paper,top=3cm, bottom=3cm, inner=3.5cm,outer=2.7cm,nofoot,headsep=1.5cm]{geometry}
\def\Ind#1#2{#1\setbox0=\hbox{$#1x$}\kern\wd0\hbox to 0pt{\hss$#1\mid$\hss}
\lower.9\ht0\hbox to 0pt{\hss$#1\smile$\hss}\kern\wd0}
\def\Notind#1#2{#1\setbox0=\hbox{$#1x$}\kern\wd0\hbox to 0pt{\mathchardef
\nn="3236\hss$#1\nn$\kern1.4\wd0\hss}\hbox to 0pt{\hss$#1\mid$\hss}\lower.9\ht0
\hbox to 0pt{\hss$#1\smile$\hss}\kern\wd0}
\def\indi{\mathop{\mathpalette\Ind{}}}
\def\nindi{\mathop{\mathpalette\Notind{}}}

\newcommand{\F}{\mathbb{F}}

\newcommand{\1}{\mathbf{1}}

\newcommand{\abs}[1]{\left|#1\right|}

\theoremstyle{plain}
\newtheorem{theorem}{Theorem}[section]
\newtheorem*{theorem*}{Theorem}
\newtheorem{prop}[theorem]{Proposition}
\newtheorem{proposition}[theorem]{Proposition}

\newtheorem{fact}[theorem]{Fact}
\newtheorem{lemma}[theorem]{Lemma}
\newtheorem{cor}[theorem]{Corollary}
\newtheorem*{cor*}{Corollary}
\newtheorem{claim}{Claim}

\theoremstyle{definition}
\newtheorem{defn}[theorem]{Definition}
\newtheorem{definition}[theorem]{Definition}
\newtheorem{remark}[theorem]{Remark}

\newtheorem{problem}[theorem]{Problem}
\newtheorem{expl}[theorem]{Example}
\newtheorem{example}[theorem]{Example}
\newtheorem{conjec}[theorem]{Conjecture}

\newtheorem{notation-num}{Notation}
\newtheorem*{notation}{Notation}

\newcommand{\tp}{\operatorname{tp}}
\newcommand{\M}{\mathcal{M}}

\newcommand{\cU}{\mathbb{M}}

\newcommand{\acl}{\operatorname{acl}}

\newcommand{\eq}{\operatorname{eq}}
\newcommand{\RCF}{\operatorname{RCF}}
\newcommand{\bdd}{\operatorname{bdd}}
\newcommand{\heq}{\operatorname{heq}}
\newcommand{\Dp}{\operatorname{dp}}
\newcommand{\gs}{\operatorname{gs}}
\newcommand{\NTP}{\operatorname{NTP}}

\newcommand{\GIP}{\operatorname{GIP}}
\newcommand{\EM}{\operatorname{EM}}

\newcommand{\Aut}{\operatorname{Aut}}

\newcommand{\Av}{\operatorname{Av}}
\newcommand{\ACF}{\operatorname{ACF}}

\newcommand{\alg}{\operatorname{alg}}
\newcommand{\Th}{\operatorname{Th}}
\newcommand{\Sk}{\operatorname{Sk}}

\newcommand{\ind}{\operatorname{\indi}}
\newcommand{\nind}{\operatorname{\nindi}}
\newcommand{\opg}{\operatorname{opg}}

\newcommand{\id}{\operatorname{id}}
\newcommand{\cl}{\operatorname{cl}}

\newcommand{\VC}{\operatorname{VC}}

\newcommand{\Mbb}{\mathbb{M}}

\newcommand{\Ucal}{\mathcal{U}}

\newcommand{\Nbb}{\mathbb{N}}
\newcommand{\End}{\operatorname{End}}

\newcommand{\Ical}{\mathcal{I}}
\newcommand{\Jcal}{\mathcal{J}}
\newcommand{\Ncal}{\mathcal{N}}

\newcommand{\B}{\mathfrak{B}}
\newcommand{\A}{\mathfrak{A}}

\newcommand{\Lcal}{\mathcal{L}}
\newcommand{\xb}{\bar x}
\newcommand{\yb}{\bar y}
\newcommand{\ab}{\bar a}
\newcommand{\bb}{\bar b}
\newcommand{\cb}{\bar c}
\newcommand{\db}{\bar d}

\newcommand{\zb}{\bar z}

\newcommand{\ib}{\bar i}
\newcommand{\jb}{\bar j}
\newcommand{\rb}{\bar r}
\newcommand{\ub}{\bar u}

\newcommand{\cfrak}{\mathfrak{c}}
\newcommand{\dfrak}{\mathfrak{d}}
\newcommand{\Bfrak}{\mathfrak{B}}
\newcommand{\Mfrak}{\mathfrak{M}}
\newcommand{\Afrak}{\mathfrak{A}}
\newcommand{\sm}{\setminus}
\newcommand{\es}{\emptyset}
\newcommand{\sub}{\subseteq}

\title{On $n$-distality, $n$-triviality and hypergraph regularity in NIP theories}
\author{Artem Chernikov and Francis Westhead}

\begin{document}
\begin{abstract}
 We study Keisler measures in strongly $n$-distal NIP theories, generalizing  some results of Simon and Chernikov--Starchenko for distal theories  and addressing some questions of Walker. In particular, we establish a hypergraph version of the distal regularity lemma, compact domination for definable fsg groups, and demonstrate that the strong $n$-distality hierarchy is strict among stable theories using a connection to Poizat's total triviality of forking. We also show that infinite strongly $n$-distal NIP fields have characteristic $0$ using a discrepancy result of Babai--Hayes--Kimmel from multiparty communication complexity.
 \end{abstract}

\maketitle

\section{Introduction}

This paper contributes to the emerging subject of \emph{higher arity   classification theory} in model theory and its connections to hypergraph combinatorics, focusing on higher arity generalizations of \emph{distality}. Typically, various tameness notions in Shelah's classification theory are  given by restrictions on the combinatorial complexity of definable
binary relations, by forbidding certain induced subgraphs (e.g.~$T$
is \emph{stable} if no definable binary relation can contain arbitrary
large finite half-graphs; and \emph{NIP} if sufficiently large random
bipartite graphs are omitted). A typical result then demonstrates
that binary relations are ``approximated'' by the unary ones in
some form, up to a ``small'' error. For example, stationarity of
forking in stable theories says that given $p\left(x\right),q\left(y\right)$
types over a model $M$, there exists a \emph{unique type} $r(x,y)$
over $M$ so that if $\left(a,b\right)\models r$ then $a\models p,b\models q$
and $a\ind_{M}b$ \textemdash{} that is, there is a unique type $r\left(x,y\right)$
extending $p\left(x\right)\cup q\left(y\right)$, up to the forking
formulas $\varphi\left(x,y\right)\in\mathcal{L}\left(M\right)$.  Recently a number of results began to emerge concerning the higher arity generalizations
of these phenomena: under
some restricting assumption on the definable relations of arity $n+1$,
demonstrate an ``approximation'' by relations each involving at
most $n$ out of $n+1$ variables, up to a ``small error''. Mirroring
the passage from graphs to hypergraphs in combinatorics, this leads
to a significant growth in complexity of the occurring phenomena. 

Distal theories were defined by Simon \cite{simon2013distal} in terms of indiscernible sequences, aiming to isolate the subclass of purely unstable NIP theories \cite{simon2015guide}. In particular, these are precisely the theories in which all global generically stable Keisler measures are smooth (\cite[Theorem 1.1]{simon2013distal}; see Section \ref{sec: keisler meas intro} for the definitions). Distality was characterized in terms of the existence of (uniform) strong honest definitions in \cite{chernikov2015externally},  exhibiting a (not noticed at the time) connection to the existing notion of the ``isolation property'' \cite{belegradek1999extended, benedikt2003definable} studied in computer science. In \cite{chernikov2018regularity}, this was recast as a distal cell decomposition, and used to establish equivalence of distality of a theory to a definable version of the strong Erd\H{o}s-Hajnal property (generalizing \cite{alon2005crossing} in the semi-algebraic and \cite{basu2010combinatorial} in the o-minimal case), as well as a so-called distal regularity lemma (generalizing \cite{fox2012overlap, fox2016polynomial} in the semi-algebraic case), for hypergraphs definable in it (see also \cite{simon2016note}). Further work demonstrates that distal structures provide a general abstract setting for tame Erd\H os-style incidence combinatorics \cite{chernikov2020cutting, chernikov2021ramsey,  chernikov2021model, bays2023incidence, chernikov2024model, tong2026zarankiewicz}. Distality is  studied from the point of view of pure model theory \cite{boxall2018definable, boxall2023theories, chernikov2015externally, kaplan2017exact, anderson2023fuzzy}, and we refer to e.g.~\cite{hieronymi2017distal, aschenbrenner2022distality, tong2025distal, okura2026distal} for further examples of distal theories (see also the introduction of \cite{chernikov2018regularity} for a survey).

Two natural higher arity generalizations of distality, \emph{$n$-distality} and \emph{strong $n$-distality} (for $1 \leq n \in \mathbb{N}$), were proposed and studied by Walker \cite{walker2023distality} (see Section \ref{sec: n-dist}; with distality = $1$-distality = strong $1$-distality). $N$-distal theories form a subclass of the better studied $n$-dependent theories (see Proposition \ref{prop: n-dist implies n-dep}), introduced by Shelah \cite{MR3666349, MR3273451} and generalizing NIP/dependent theories in the case $n=1$. In the last decade, $N$-dependent theories were studied in pure model theory \cite{chernikov2014n, hempel2016n, chernikov2019mekler, chernikov2021n, chernikov2024n}; in connection to (arithmetic) hypergraph regularity in combinatorics \cite{chernikov2020hypergraph, terry2021irregular, terry2023improved, terry2024growth, terry2025structure, terry2025quadratic, gishboliner2025regularity, sheats2025linear}  (generalizing from graphs \cite{alon2007efficient, lovasz2010regularity, hrushovski2013generically}, as well as hypergraphs \cite{chernikov2021definable, fox2019erdHos}, of finite VC-dimension; we also refer to the introduction of \cite{chernikov2024perfect} for a survey); and higher arity VC theory (VC$_n$ theory) and PAC learning  (PAC$_n$ learning) in product spaces \cite{Kobayashi, KotaPAC, chernikov2020hypergraph, chernikov2025higher, CorMal2025}. We also mention some higher arity generalizations of stability, forming another subclass of $n$-dependent theories, considered in \cite{takeuchi, terry2021higher, terry2021irregular, CheOber, zbMATH08064119, arXiv:2508.05839, arXiv:2506.19147}. See also \cite{CheOber} for some  connections to higher amalgamation and stationarity in simple theories.

However, already the restriction of (strong) $n$-distality to NIP, or even stable,  theories, is of interest. In this paper, we  generalize some aspects of the rich theory of distal theories to strongly $n$-distal NIP theories, focusing on Keisler measures, as well as isolate some new phenomena --- this should be viewed as a first step in the development of a general theory of (strong) $n$-distality, informed by $n$-dependence. In what follows we summarize the main results of the paper. 

In Section \ref{sec: prelims} we discuss some preliminaries for the rest of the paper. In particular, in Section \ref{sec: meas on Bool alg} we overview some facts about finitely additive probability measures on Boolean algebras, including a criterion for the existence/uniqueness of extensions of a measure to a bigger Boolean algebra (Fact \ref{existence of measure extensions}, Proposition \ref{border and determination}). In Section \ref{sec: keisler meas intro} we review some basic properties and results concerning Keisler measures on definable sets in NIP theories, and in Section \ref{sec: forking} we review properties of forking (and related pre-independence relations) in NIP theories. In Section \ref{sec: n-dist} we review the definitions and basic properties of (strong) $n$-distality, and note that $n$-distality implies $n$-dependence (Proposition \ref{prop: n-dist implies n-dep}). In Section \ref{sec: cyl inters sets} we recall (definable) \emph{cylinder intersection sets} (i.e.~subsets of product spaces in a Boolean algebra generated by sets that only depend on a proper subset of the coordinates) which play a crucial role in the paper.

In Section \ref{sec: n-determinacy for measures and n-distal regularity lemma} we establish our main results about products of Keisler measures in (strongly) $n$-distal NIP theories and obtain some combinatorial applications.  In Section \ref{sec: Indiscernible measures in strongly n-distal NIP theories} we generalize some results in NIP and strongly $n$-distal theories from indiscernible sequences to indiscernible measures. In Section \ref{sec: n-smooth measures in strongly n-distal NIP theories} we generalize Simon's result that in a distal theory all generically stable measures are smooth \cite{simon2013distal} to strongly $n$-distal NIP theories:
\begin{theorem*}[Proposition \ref{2-determinacy with a type}]
    Assume $T$ is NIP and strongly $n$-distal, and let $\mu_0(x_0)$, $\ldots$, $\mu_{n-1}(x_{n-1})$ be global measures generically stable over a small model $\M$. Then $\mu := \mu_0\otimes\dots\otimes\mu_{n-1}$ is $n$-smooth over $\M$ (i.e.~for any global measure $\omega(x_0,\dots,x_{n-1})$ so that $\omega|_{\M} = \mu|_{\M}$ and $\omega|_{(x_i : i \neq j)} = \mu|_{(x_i : i \neq j)}$ for all $0 \leq j < n$, we must have $\omega = \mu$; Definition \ref{def: n-smooth}).
\end{theorem*}

\noindent We note that a different notion of $n$-smoothness was also considered by Walker. In Section \ref{sec: N-distal packing lemma, n-determinacy for measures and hypergraph regularity lemma} we present some applications of this result. First, combining it with a compactness argument,  we obtain an \emph{$n$-distal packing lemma}:
\begin{theorem*}[Proposition \ref{prop: dist packing lemma}]
	Let $T$ be NIP and strongly $n$-distal. For every formula $\varphi(x_0, \ldots,x_{n-1};x_{n})$ and $\varepsilon>0$ there exist some $k \in \mathbb{N}$ and formulas  $\gamma_{i}(x_0, \ldots, x_{n-1}; z_0) \in \Lcal$, $\rho_i(x_{n}; z_0) \in \Lcal$ and $\chi_i^-(x_0,\dots, x_{n-1};z), \chi_i^+(x_0,\dots, x_{n-1};z) \in \Lcal^{n}_{x_0, \ldots, x_{n-1}, z}$ for $i < k$, where $z = z_0 z_1$ and each $\chi_i^-, \chi_i^+$ is given by a Boolean combination of $\gamma_{j}(x_0, \ldots, x_{n-1}; z_0)$, and some formulas of the form $ \delta_{u}((x_t : t \in u ); z_1) \in \mathcal{L}$ for $u \subsetneq \{0, \ldots, n-1\}$, satisfying the following.   Given a tuple of global generically stable measures $\bar{\mu} = (\mu_0(x_0), \ldots, \mu_{n-1}(x_{n-1}))$ there exists some $e_{\bar{\mu}} \in \cU^{z_0}$ so that: 
	$(\rho_i(\cU;e_{\bar{\mu}}) : i < k)$ partition $\cU^{x_n}$ and for every  $b \in \rho_i(\cU^{x_{n}})$ there is some $d_b \in \cU^{z_1}$ with 
     \begin{enumerate}
         \item $\chi_i^-(x_0,\dots, x_{n-1}; e_{\bar{\mu}}, d_b)\rightarrow \varphi(x_0,\dots, x_{n-1},b) \rightarrow \chi_i^+(x_0,\dots, x_{n-1}; e_{\bar{\mu}}, d_b)$,
        \item $\mu_0\otimes \dots \otimes\mu_{n-1}(\chi_i^+(x_0,\dots, x_{n-1}; e_{\bar{\mu}}, d_b)) - \mu_0\otimes \dots \otimes\mu_{n-1}(\chi_i^-(x_0,\dots, x_{n-1}; e_{\bar{\mu}}, d_b))< \varepsilon$.
     \end{enumerate}
\end{theorem*}
\noindent This is a stronger form of the $n$-dependent packing lemma (i.e.~a packing lemma for families of sets of finite VC$_n$ dimension, generalizing the classical Haussler's packing lemma \cite{haussler1995sphere} in the case $n=1$) established in 
\cite{chernikov2020hypergraph} (see Remark \ref{rem: n-dependent packing lemma} for a discussion).

\noindent As an application, we obtain a regularity lemma for hypergraphs definable in strongly $n$-distal NIP theories (generalizing the case $n=1$ established in \cite{chernikov2018regularity})\footnote{We note that a hypergraph regularity lemma for strongly $n$-distal NIP hypergraphs was also obtained independently by Tong \cite{TongThesis}, using a different approach of generalizing strong honest definitions to higher arity. Another higher arity version of strong honest definitions was previously obtained by Walker \cite{WalkerThesis}.}:
\begin{theorem*}[See Corollary \ref{cor: str n-dist reg lemma} for a more precise version]
	Let $\M$ be an NIP and strongly $n$-distal structure. Given a definable $(n+1)$-ary relation $\varphi(x_0, \ldots, x_{n})$ and $\varepsilon > 0$ there exists $k \in \mathbb{N}$ satisfying the following. Given finitely supported measures $\mu_0(x_0), \ldots, \mu_n(x_n)$, for every $0 \leq i \leq n$ there is a  partition $(P^i_{\ell} : \ell < k)$ of $\prod_{j \neq i}\M^{x_j}$ so that $\sum_{(\ell_0, \ldots, \ell_n)} \mu_0 \otimes \ldots \otimes \mu_n \left( \bigwedge_{i =0}^{n} P^i_{\ell_i} \right) < \varepsilon$, where the sum is over all cylinder intersections sets   $\bigwedge_{i =0}^{n} P^i_{\ell_i} \subseteq \prod_{0 \leq j \leq n}\M^{x_j}$ (Definition \ref{def: cylinder inters sets}) that are \emph{not $\varphi$-homogeneous}. 
\end{theorem*}

\noindent Again, this strengthens the $n$-dependent (equivalently, finite $\VC_n$-dimension) hypergraph regularity lemma from \cite{chernikov2020hypergraph} which only guarantees $\varepsilon$-homogeneous rather than actually homogeneous cylinder intersection sets, see Remark \ref{rem: n-dep reg lemma}). 

\noindent This regularity lemma immediately implies a higher arity version of the \emph{strong Erd\H{o}s-Hajnal property} (again generalizing the $n=1$ case from \cite{chernikov2018regularity}):

\begin{theorem*}[See Corollary \ref{cor: n-str EH} for a more precise version]
	Let $\M$ be NIP and strongly $n$-distal. Then every definable relation $\varphi(x_0, \ldots, x_{n}) \in \Lcal$ satisfies the \emph{$n$-strong Erd\H{o}s-Hajnal property}, or \emph{$n$-sEH}: there exists $\alpha > 0$ satisfying the following. Given  any finitely supported measures $\mu_0(x_0), \ldots, \mu_n(x_n)$, there exists a cylinder intersection set $C = \bigwedge_{i =0}^{n} C_{i} \subseteq \prod_{0 \leq j \leq n}\M^{x_j}$ (with $C_i \subseteq \prod_{j \neq i}\M^{x_j}$) so that $\mu_0 \otimes \ldots \otimes \mu_n \left(C \right) \geq \alpha$ and $C$ is $\varphi$-homogeneous.
\end{theorem*}
\noindent Finally, in Section \ref{sec: n-distality vs av meas} we demonstrate that $n$-distal (rather than strongly $n$-distal) NIP theories are characterized by $n$-determinacy for $(n+1)$-tuples of particular generically stable measures obtained by averaging over mutually indiscernible sequences, generalizing the results in  \cite[Proposition 2.21]{simon2013distal} in the case $n=1$.

In a distal theory all global generically stable types are algebraic. This fits as the first level into a hierarchy expressing $n$-triviality of forking for realizations of generically stable types in (strongly) $n$-distal NIP theories. Recall that a stable theory $T$ has \emph{trivial forking} if $a \ind_{A} b, a \ind_{A} c$ and $b \ind_{A} c$ implies $a \ind_{A} bc$. Poizat  \cite{goode1991some} considers its higher arity version, \emph{$n$-trivial forking} (if all proper subsets of a set of size $n+2$ are independent over $A$, then the whole set is independent over $A$). Walker shows, using the $n$-determinacy principle for generically stable types, that in a stable theory $(n+1)$-distality is equivalent to $n$-triviality of forking (Fact \ref{fac: walker $k$-triv vs $k$-dist}). In Section \ref{sec: n-dist and n-triv} we consider two other notions of ``$n$-triviality''. First is also from \cite{goode1991some}: $T$ has \emph{(1-)totally trivial forking} if $a \ind_{A} b $ and $a \ind_{A} c$ implies $a \ind_{A} bc$, without requiring $b$ and $c$ to be independent (and its higher arity version, totally $n$ trivial forking, see Definition \ref{def: tot k-triv forking NIP}). The other is (endless) \emph{indiscernible triviality}, considered in \cite{braunfeld2021characterizations} (we also consider its higher arity version, $n$-indiscernible triviality, Definition \ref{def: k-indisc triv}). We show:
\begin{theorem*}[Proposition \ref{prop: total triv in stable}]
	($T$ stable) The following are equivalent for all $n \geq 1$: 
	\begin{enumerate}
	\item $T$ is strongly $(n+1)$-distal,
	\item  $T$ is (endlessly) indiscernibly $n$-trivial,
	\item  $T$ has totally $n$-trivial forking.
	\end{enumerate}
\end{theorem*}

\noindent More generally, this holds for realizations of generically stable types in NIP theories (see Remark \ref{rem: gen stable tot triv}). Combining this with some results and examples in \cite{goode1991some}, we can thus address a question of Walker \cite{walker2023distality}:

\begin{cor*}[Corollary \ref{cor: sep for strong $n$-dist}]
	 For every $n\geq 1$, there exists a superstable $2$-distal $T$ which is strongly $2^n$-distal but not strongly $(2^{n}-1)$-distal. There exists a $2$-distal superstable theory which is not strongly $n$-distal for any $n < \omega$. For a $1$-based stable theory and any $k \geq 2$, $k$-distality implies strong $2$-distality.
\end{cor*}

\noindent We also consider some natural intermediate notions between $n$-distality and strong $n$-distality for $n \geq 2$ (Definition \ref{def: (k,l)-distality}) and demonstrate that some of them collapse in stable theories (Proposition \ref{prop: 2-dist implies 2+-dist}).

In Section \ref{sec: Invariant generically stable measures in n-distal groups}
we consider translation invariant measures on definable groups. If $T$ is ($1$-)distal then every generically stable measure $\mu \in \mathfrak{M}_x(\cU)$ is smooth. This fails badly in stable $2$-distal theories. However, generalizing the distal case, we have: 
\begin{theorem*}[Theorem \ref{thm: G-inv gen stab mu is smooth}]
 Assume $1 \leq n \in \omega$ and $T$ is NIP and strongly $n$-distal. Assume $G$ is a definable group and $\mu \in \mathfrak{M}_{G}(\cU)$ is generically stable and $G$-invariant. Then $\mu$ is smooth.
\end{theorem*}
\noindent This implies that compact domination holds for definable fsg groups in strongly $n$-distal  NIP theories (see Fact \ref{fac: comp dom iff smooth measure} and the discussion there), and so an  arithmetic version of the distal regularity lemma holds for definable groups exactly as in \cite[Section 6]{conant2021structure}.	We also note that for types the result holds in arbitrary $n$-distal theories, hence no stable $n$-distal theory can type-define an infinite group (Proposition \ref{prop: stable groups are not $n$-distal}).

In Section \ref{sec: Infinite strongly k-distal field have characteristic 0} we prove the following: 
\begin{theorem*}[Theorem \ref{thm: no inf fields n-dist}] 
	No theory satisfying the $k$-strong Erd\H{o}s-Hajnal property (for uniform finitely supported measures) can define an infinite field of positive characteristic. In particular any infinite field definable in a strongly $k$-distal NIP theory has characteristic $0$; and the (stable) theory $\ACF_p$ of algebraically closed fields of characteristic $p>0$ does not admit a strongly $n$-distal NIP expansion for any $n$.
\end{theorem*}
\noindent This generalizes \cite[Corollary 6.3]{chernikov2018regularity} in the case $n=1$, which combined that NIP fields are Artin-Schreier closed \cite{kaplan2011artin} with the count for point-line incidences on the affine planes over large finite fields. Our proof here combines the fact that $n$-dependent fields are still Artin-Schreier closed \cite{hempel2016n} with a classical result of Babai--Hayes--Kimmel from multiparty communication complexity establishing high discrepancy of \emph{generalized inner products}  over large finite fields \cite{babai1998cost}.

Throughout the paper, we state some questions and conjectures on possible generalizations and refinements of our results.

\subsection{Acknowledgements}
We thank Chris Laskowski, Aris Papadopoulos and Roland Walker for helpful discussions on some of these topics, and Caroline Terry and Mervyn Tong for their comments on the preliminary version of the paper. Part of the results in Section \ref{sec: n-determinacy for measures and n-distal regularity lemma} also appear in the MA thesis of Westhead \cite{WestheadThesis} under the supervision of Chernikov. Chernikov was partially supported by the NSF Research Grant DMS-2246598; and by the Deutsche Forschungsgemeinschaft
(DFG, German Research Foundation) under Germany's Excellence Strategy-EXC-2047/1-390685813. He thanks Hausdorff Research Institute for
Mathematics in Bonn, which hosted him during the Trimester Program
``Definability, decidability, and computability''. Westhead was partially supported by Laskowski's NSF Research Grant DMS-2154101.

\section{Preliminaries}\label{sec: prelims}

\subsection{Measures on Boolean algebras}\label{sec: meas on Bool alg}

Given a Boolean Algebra, $\Bfrak$, we let $S(\Bfrak)$ be its Stone dual, i.e.~the compact Hausdorff totally disconnected space  of ultrafilters on $\Bfrak$ with  the topology generated by a basis of clopen sets of the form $[B] := \{\Ucal \in S(\Bfrak): B\in \Ucal\}$. A $\sigma$-additive Borel measure $\rho$ on a topological space is \emph{regular} if for any Borel set $A$,  $\rho(A) = \sup\{\rho(F):F\sub A\text{ closed}\} = \inf\{\rho(O):A\sub O\text{ open}\}$. Given a finitely additive probability measure $\mu$ on  $\Bfrak$, equivalently on the clopen subsets of $S(\Bfrak)$, it extends uniquely to a $\sigma$-additive regular  probability measure $\tilde \mu$ on the Borel subsets of $S(\Bfrak)$ (see e.g.~the proof of \cite[Section 7.1]{simon2015guide}, which applies verbatim in this slightly greater level of generality). In a slight abuse of notation we will in general identify $\mu$ with $\tilde \mu$.

For a measure $\mu$ on a Boolean algebra $\Bfrak$, we let $S(\mu)$ denote the support of $\mu$, i.e.~the set of ultrafilters $\Ucal \in S(\Bfrak)$ so that $\mu(B) > 0$ for every $B \in \Ucal$. Then, by compactness of  $S(\Bfrak)$ and finite additivity of $\mu$, $S(\mu)$ is a non-empty closed subset of $S(\Bfrak)$ and $\tilde{\mu}(S(\mu)) = 1$.

    Given a measure $\mu$ on a Boolean algebra $\Bfrak$ with a subalgebra $\Afrak \sub \Bfrak$, we denote the restriction of $\mu$ to $\Afrak$ by $\mu|_\Afrak$. And we let $\widetilde{\mu}|_{S(\Afrak)} := \pi_{\ast} \widetilde{\mu}$ denote the pushforward of $\widetilde{\mu}$ under the canonical (continuous, surjective) restriction map $\pi:S(\B) \rightarrow  S(\A)$ defined via $\pi(\mathcal{U}) :=  \mathcal{U}|_{\Afrak}$. Note that for any Boolean algebras $\Afrak \sub \Bfrak$ and $\mu$ (respectively, $\nu$) a finitely additive probability measure on $\Afrak$ (respectively, on $\Bfrak$), if $\nu$ extends $\mu$, then also $\tilde \nu|_{S(\Afrak)} = \tilde\mu$ (indeed, $\tilde\nu |_{S(\Afrak)}$ is a regular extension of $\mu$, so must be equal to $\tilde \mu$ by uniqueness).

\begin{definition}\label{def: A-determined measure}
    Suppose $\Afrak\sub \Bfrak$ are Boolean algebras and $\mu$ is a finitely additive probability measure on $\Bfrak$. We say that $\mu$ is \emph{$\Afrak$-determined} if $\mu$ is the unique extension of $\mu|_{\Afrak}$ to a finitely additive probability measure on $\Bfrak$.
\end{definition}

We now give a useful characterization of $\Afrak$-determinacy (generalizing \cite[Lemma 4.1]{simon2012finding}).
\begin{definition}\label{def: bored of a set}
   Given Boolean algebras $\Afrak\sub \Bfrak$, let $\pi:S(\B) \rightarrow  S(\A)$ denote the canonical restriction. Given $B\in \B$, we  define its \emph{$\Afrak$-border} $\partial_\A B \subseteq S(\Afrak)$ to be 
   \[\{p\in S(\A):\text{ there are } q_0,q_1 \in S(\B)\text{ with } \neg B\in q_0, B\in q_1\text{ and } \pi(q_0)=\pi(q_1)=p\}.\]
\end{definition}

\begin{lemma}\label{border is closed}
    For any $B\in \Bfrak$ and subalgebra $\Afrak\sub \Bfrak$, $\partial_\A B$ is a closed subset of $S(\A)$.
\end{lemma}
\begin{proof}
    Suppose $p \in S(\A) \sm \partial_\A B$. Then, $\bigcap\limits_{A\in p}[A]_\B \sub [B^\varepsilon]_\B$ for some $\varepsilon\in \{0,1\}$, without loss of generality $\varepsilon =1$. So, $\bigcap\limits_{A\in p} [A\land \neg B]_\B=\es$. Since $[\neg B]_{\Bfrak}$ is compact, there is a finite subset  $p_0 \subseteq p$ so that $\bigcap\limits_{A\in p_0} [A\land \neg B]_\B=\es$, so $\bigcap\limits_{A\in p_0} [A]_\B\sub [B]_\B$. Let $D_p := \bigcap\limits_{A\in p_0} [A]_\A \sub S(\A)$. Now, $D_p$ is clopen, $p\in D_p$ and $D_p \cap\partial_\A B = \es$. Thus, $S(\A) \sm \partial_\A B = \bigcup\limits_{p \in S(\A) \sm \partial_\A B} D_p$, an open set. 
\end{proof}

\noindent We recall a general fact about extending finitely additive probability measures:
\begin{fact}\label{existence of measure extensions}\cite{los1949extensions}
    Suppose $\A \sub \B$ are Boolean algebras and $\mu$ is a finitely additive probability measure on $\A$. Then for each $B\in \B$, and $\alpha \in [0,1]$ such that $\sup\{\mu(A): A\in \A, A\sub B\} \leq \alpha \leq  \inf\{\mu(A):A\in \A, B\sub A\}$, there is a finitely additive probability measure $\nu$ on $\B$  extending $\mu$ with $\nu(B)=\alpha$.
\end{fact}

%\noindent This immediately implies the following:
%
%\begin{fact}\label{determination of measures via squashing}
%    Suppose $\mu$ is a finitely additive probability measure on some Boolean algebra $\Afrak$ and $\Bfrak \supseteq \Afrak$ is some larger Boolean algebra. There is a unique extension of $\mu$ to a finitely additive probability measure on $\Bfrak$ if and only if for each $\varepsilon>0$ and $B \in \Bfrak$, there are $A^+,A^-\in \Afrak$ such that:\begin{enumerate}
%        \item $A^- \sub B \sub A^+$
%        \item $\mu(A^+)-\mu(A^-)<\varepsilon$.
%    \end{enumerate}
%\end{fact}

\begin{prop}\label{border and determination}
    Suppose $\Afrak \sub \Bfrak$ are Boolean algebras and $\nu$ is a finitely additive probability measure on $\B$. Denote $\mu:=\nu|_\A$. The measure $\nu$ is $\A$-determined if and only if $\tilde \mu(\partial_\A B)=0$ for all $B\in \B$. 
\end{prop}

\begin{proof}
    For $B\in \Bfrak$ and  $\varepsilon\in \{0,1\}$, let $X_B^\varepsilon :=  \bigcup\{[A]_\A: A\in \A, [A]_\B \sub [B^{\varepsilon}]_\B\}$. Then, by the proof of Lemma \ref{border is closed}, $\partial_\A B=S(\A)\sm (X_B^0 \cup X_B^1)$. Both $X_B^\varepsilon$ are open, and disjoint, so $\mu(\partial_\A B)=0$ if and only if $\mu(X_B^0)+\mu(X_B^1)=1$. 
    
    Now, for any $\nu'$ extending $\mu$ to $\B$, $\tilde \mu(X_B^1)\leq \tilde \nu'([B])\leq 1-\tilde\mu(X_B^0)$. So, if $\tilde \mu(\partial_\A B)=0$ for each $B$, it is clear that $\nu'(B)$ is determined by $\mu$ for each $B$, so $\nu' = \nu$. 

    In the other direction, suppose that $\tilde \mu(\partial_\A B)>0$ for some $B\in \B$. So, $\tilde \mu(X_B^0) +\tilde\mu(X_B^1)<1$. Trivially, $\sup\{\mu(A): A\in \A, A\sub B^{\varepsilon}\}\leq \tilde{\mu}(X^\varepsilon_B)$. So,
    \begin{gather*}
    	1-\mu(X^0_B) \leq 1-\sup\{\mu(A): A\in \A, A \sub \neg B\} = 1-\sup\{1-\mu(\neg A): A\in \A, A\sub \neg B\}\\
    	= \inf\{\mu(\neg A): A\in \A, A \sub \neg B\} = \inf\{\mu(A): A\in \A, B \sub A\}.
    \end{gather*}
    Hence, $\sup\{\mu(A): A\in \A, A\sub B\} <  \inf\{\mu(A):A\in \A, B\sub A\}$. By Fact \ref{existence of measure extensions}, $\mu$ has two extensions to $\B$ that take distinct values on $B$.
\end{proof}

\subsection{Keisler measures}\label{sec: keisler meas intro}
We fix a complete $\mathcal{L}$-theory $T$ and $\cU \models T$ a monster model.  Given set $A \subseteq \cU$, a \emph{(Keisler) measure} over $A$ in the variables $x$ is a finitely additive probability measure on $\Lcal_x(A)$, the Boolean algebra of $A$-definable subsets of $\Mbb^x$. We denote by $\Mfrak_x(A)$ the set of Keisler measures over $A$ in $x$. 

\begin{definition} Take a global measure $\mu(x)\in \Mfrak_x(\Mbb)$ and a set $A\sub \Mbb$.
    \begin{enumerate}
        \item $\mu$ is \emph{$A$-invariant} if for any $\varphi(x,y)\in \Lcal$ and $b\equiv_A b'$, we have that $\mu(\varphi(x,b)) = \mu(\varphi(x,b')$. If $\mu$ is $A$-invariant, $q\in S_y(A)$, $a\in A$ and $\psi(x,y,z)\in \Lcal$, we write $\mu(\psi(x,q,a))$ to refer to $\mu(\psi(x,d,a))$ for some/any $d\models q$.
        \item $\mu$ is \emph{finitely satisfiable (f.s.) over $A$} if for any $\varphi(x,y)$ and $b$, if $\mu(\varphi(x,b))>0$, then there is a tuple $a$ in $A$ with $\models \varphi(a,b)$.
        \item $\mu$ is \emph{$A$-definable} if $\mu$ is $A$-invariant and, for each $\varphi(x,y)$ and $\varepsilon>0$, the set $\{q\in S_y(A): \mu(\varphi(x,q))>\varepsilon\}$ is open.
        \item $\mu$ is \emph{generically stable} if there is some small $A$ over which $\mu$ is f.s. and definable (this is not the right definition outside of NIP theories, but we are only interested in NIP theories in this paper). 
        \item $\mu$ \emph{forks over $A$} if there is $\varphi(x)\in \Lcal(\Mbb)$ with $\mu(\varphi(x))>0$ that forks over $A$.
        \item $\mu$ is \emph{Borel definable over $A$} if $\mu$ is $A$-invariant and, for each $\varphi(x,y)$ and $\varepsilon>0$, the set $\{q\in S_y(A): \mu(\varphi(x,q))>\varepsilon\}$ is Borel.
    \end{enumerate}
\end{definition}

\begin{fact}\label{inv iff borel def}\cite[Corollary 4.9]{hrushovski2011nip}
     Assume $T$ has NIP and $\M \prec \cU$ is small. Then a global measure $\mu$ is $\M$-invariant if and only if $\mu$ is Borel definable over $\M$. 
\end{fact}

%\begin{example}\label{average measures example}
%    If $T$ has NIP, we have a simple example of a generically stable measure, an \emph{average measure of an indiscernible sequence}. Suppose $(b_i)_{i\in [0,1]}$ is an indiscernible sequence indexed by the interval $[0,1]$. Let $\lambda$ denote the Lebesgue measure on $[0,1]$. Take $\varphi(x)\in \Lcal(\Mbb)$. Note that by NIP, the set $\{t\in [0,1]: \models \varphi(b_t)\}$ is a finite union of intervals. So, we can define the average measure by $Av(b_{[0,1]})(\varphi)= \lambda(\{t\in [0,1]: \models \varphi(b_t)\})$. Clearly, $Av(b_{[0,1]})$ is f.s. (hence invariant) over $\bigcup\{b: b\in b_{[0,1]}\}$. By compactness and the bound on alternation given by NIP, $Av(b_{[0,1]})$ is definable over $b_{[0,1]}$.
%\end{example}

The following is essentially \cite[Proposition 4.3]{hrushovski2011nip} (it is stated there with $\Ncal = \cU$, but the proof goes through only assuming that $\Ncal$ is $|\M|^+$-saturated):
\begin{fact}\label{invariant and non-forking measures}
     Assume that $T$ has NIP. Suppose $\M\models T$, $\Ncal \succeq \M$ is $|\M|^+$-saturated and $\mu\in \Mfrak_x(\Ncal)$. Then $\mu$ does not fork over $\M$ if and only if $\mu$ is $\Aut(\Ncal/\M)$-invariant.
\end{fact}

\begin{definition}\label{def: tens prod of meas}
\begin{enumerate}
    \item Given $\mu(x),\nu(y)$ measures over $A$, we define $\mu\times\nu$ to be the unique measure on the product Boolean algebra $\Lcal_{x}(A) \times \Lcal_{x}(A) \subseteq \Lcal_{x,y}(A)$ determined by $(\mu \times \nu)(\varphi(x)\wedge \psi(y)):= \mu(\varphi(x))\cdot\nu(\psi(y))$ for all $\varphi(x),\psi(y) \in \Lcal(A)$.
    \item Suppose $\mu(x)$ is an $A$-invariant global measure, $\nu(y)$ is a global measure, and $\M$ is a small model containing $A$. For any $\varphi(x,y,z)\in \Lcal$ and $b \in \cU^{z}$, we define $(\mu\otimes\nu)(\varphi(x,y,b)) := \int_{S_y(\Ncal)}\mu(\varphi(x,q,b))d\nu|_{\Ncal}$, where $\Ncal$ is any small model with $\M b \sub \Ncal$.
\end{enumerate}
\end{definition}

\begin{remark}\label{rem: tens prod inv meas NIP}
    The tensor product of invariant measures is well-defined in NIP theories by Fact \ref{inv iff borel def}, and does not depend on the choice of $\mathcal{N}$.  We can iterate the above to define $\bigotimes\limits_{i\in\Ical} \mu_i(x_i)$ for any linear order $\Ical$. Finite linear orders are clear, and since formulas may involve only finitely-many variables this is enough.
\end{remark}

\noindent The theory of the tensor product for invariant measures in NIP theories is developed in \cite{hrushovski2011nip}. The following basic observations can be found in e.g.~\cite[Section 7.4]{simon2015guide}. 
\begin{fact}\label{basic tensor facts} Assume $T$ has NIP. Then $\otimes$ is associative, but in general not commutative. If $\mu \otimes \nu = \nu \otimes \mu$, we say that $\mu$ and $\nu$ \emph{commute}. If $\mu$ and $\nu$ are both invariant (finitely satisfiable, definable, generically stable) over $A$, then $\mu \otimes \nu$ is invariant (finitely satisfiable, definable, generically stable) over $A$. 
\end{fact}

\begin{fact}\label{fact: gen stable measures}(T NIP) 
\begin{enumerate}
	\item If $\mu \in \mathfrak{M}_x(\cU)$ is generically stable and $\nu \in  \mathfrak{M}_x(\cU)$ is any invariant measure, then $\mu \otimes \nu = \nu \otimes \mu$ \cite[Lemma 3.1]{hrushovski2013generically}.
	\item If $\mu \in \mathfrak{M}_x(\cU)$ is generically stable over $\M$, $\Ncal \succ \M$ is $|\M|^{+}$-saturated and $\nu \in \mathfrak{M}_x(\Ncal)$ is $\Aut(\Ncal/\M)$-invariant, then $\mu|_{\Ncal} = \nu$  \cite[Proposition 3.3]{hrushovski2013generically}.
\end{enumerate}
\end{fact}

\begin{remark}\label{rem: pushforw under aut}
	Given a measure $\mu \in \mathfrak{M}_x(A)$ and $\sigma \in \Aut(\cU)$, letting $A' := \sigma(A)$ we have a measure $\sigma(\mu) \in  \mathfrak{M}_x(A')$ defined via  $\sigma(\mu)(\varphi(x,c)) := \mu(\varphi(x,\sigma^{-1}(c))$ for all $\varphi \in \mathcal{L}$ and tuples $c$ from $A'$. This corresponds to the pushforward of the associated Borel probability measures, as follows. Let $\sigma': S_{x}(A) \to S_{x}(A')$ be the map defined via $\sigma'(\tp(b/A)) := \tp(\sigma(b)/A')$. Then $\sigma'$ is continuous: for any $\psi(x,c) \in \mathcal{L}(A')$ we have
$$(\sigma')^{-1}([\psi(x,c)]) = [\psi(x,\sigma^{-1}(c))]$$
for the corresponding clopen sets. Then, for the corresponding extensions to regular Borel measures on $S_{x}(A)$ and $S_{x}(A')$ (see Section \ref{sec: meas on Bool alg}), we have $\widetilde{\sigma(\mu)} = \sigma'_{\ast} (\widetilde{ \mu})$ (the push-forward measure under the continuous $\sigma'$). Indeed, we have the equality on the clopens by definition, so we get equality on the  full Borel $\sigma$-algebra by uniqueness.
\end{remark}

\subsection{Forking in NIP theories}\label{sec: forking}

We will use some properties of forking independence $\ind$ in arbitrary theories, as well as in NTP$_2$, resilient, NIP and stable theories (where $a \ind_C b$ denotes $\tp(a/bC)$ does not fork over $C$, a non-symmetric condition in general).
\begin{fact}\label{fac: forking in general T}
	Let $T$ be arbitrary (see e.g.~\cite[Lemma 3.12]{chernikov2012forking}).
	\begin{enumerate}
		\item Assume that $a \ind_{B} b$. Then $\varphi(x,b)$ divides/forks over $B$ if and only if it divides/forks over $aB$.
		\item $a \ind_C b \Leftrightarrow a \ind_{\acl(C)} b$.
		\item (Left transitivity of forking) If $a \ind_{Dc} b$ and $c \ind_{D} b$, then $a c \ind_{D} b$.
	\end{enumerate}
\end{fact}

\begin{defn}\cite[Definition 4.8]{yaacov2014independence}
	A theory $T$ is \emph{resilient} if for any indiscernible sequence $\left( a_i : i \in \mathbb{Z}\right)$ and formula $\varphi(x,y) \in \mathcal{L}(\emptyset)$, if $\varphi(x,a_0)$ divides over $a_{\neq 0}$, then $\{ \varphi(x,a_i) : i \in I\}$ is inconsistent.
\end{defn}

\begin{remark}
	It is immediate from the definition that if $T$ is resilient, then $T_{A}$ is also resilient.
\end{remark}

\begin{fact}\cite[Proposition 4.11]{yaacov2008model}\label{fac: simple NIP resilient NTP2}
	\begin{enumerate}
		\item If $T$ is either simple or NIP, then $T$ is resilient.
		\item If $T$ is resilient, then $T$ is NTP$_2$.
	\end{enumerate}
\end{fact}

\begin{problem}\cite[Question 4.14]{yaacov2008model}
	Are there NTP$_2$ theories that are not resilient?
\end{problem}

\begin{fact}\cite[Proposition 5.1]{kaplan2014strict}\label{fac: resilient div witness}
	Assume that $T$ is resilient, $A$ any small set and $\left(a_i : i \in I \right)$ an $A$-indiscernible sequence such that $a_{\neq i} \ind_A a_i$ for all $i \in I$. If $\varphi(x,y) \in \mathcal{L}(A)$ and $\varphi(x,a_0)$ divides over $A$, then $\left\{ \varphi(x,a_i) : i \in I \right\}$ is inconsistent.
\end{fact}

\begin{defn}
\begin{enumerate}
	\item A small set $A \subseteq \mathbb{M}$ is an \emph{extension base} if every complete type $p(x)$ over $A$ admits a global extension non-forking over $A$.
	\item A theory $T$ is \emph{extensible} if every small set is an extension base.
\end{enumerate}	
\end{defn}

\begin{remark}
\begin{enumerate}
	\item $T$ is extensible if and only if every $1$-type over every small set has a global non-forking extension.
	\item Examples of extensible theories: $o$-minimal, simple, ACVF, $p$-adics, ... (\cite{hossain2022extension}).
\end{enumerate}	
\end{remark}

\begin{fact}\cite{chernikov2012forking}\label{fac: forking = div NTP2}
	If $T$ is NTP$_2$, $A$ is an extension base and $\varphi(x,a) \in \mathcal{L}(\mathbb{M})$ is a formula, then $\varphi(x,a)$ divides over $A$ if and only if it forks over $A$.
\end{fact}

\begin{defn}
	We write $a \ind^i_{A} b$ if there is a global type $p$ extending $\tp(a/Ab)$ which is Lascar-invariant over $A$, i.e.~for every $c \equiv^{L}_{A} d$ and $\varphi(x,y) \in \mathcal{L}(A)$, $\varphi(x,c) \in p \Leftrightarrow \varphi(x,d) \in p$.
\end{defn}

We will need some standard facts about forking and invariance:
\begin{fact}\label{fac: forking in NIP}
	\begin{enumerate}
	\item Assume $T$ is stable, $p$ is a global type, and $A$ a small set. Then $p$ does not fork over $A$ if and only if $p$ is $\acl^{\eq}(A)$-invariant (see e.g.~\cite{pillay1996geometric}).
		\item If $T$ is NIP and $A$ any set, then $a \ind_A b \iff a \ind^i_A b$ (see e.g.~\cite[Section 2]{chernikov2012forking}).
		\item  If  $\mathcal{I}$ is an $A$-indiscernible sequence and $a \ind^i_A \mathcal{I}$, then $\mathcal{I}$ is indiscernible over $aA$ (e.g.~\cite[Remark 2.20]{chernikov2012forking}).
		\item $a \ind^u_A b \implies a \ind^i_A b \implies a \ind_A b$ (see e.g.~\cite[Section 2]{chernikov2012forking}; here $a \ind^u_A b$ denotes that $\tp(a/bA)$ is finitely satisfiable in $A$). 
		\item If $T$ is NIP, $p \in S_x(\cU)$ a global type and $A \subseteq \cU$ small, the $p$ does not fork over $A$ if and only if $p$ is $\Aut(\M/\bdd(A))$-invariant, where $\bdd(A)$ denotes the bounded closure of $A$ in $\cU^{\heq}$ (\cite[Proposition 2.11]{hrushovski2011nip}).
	\end{enumerate}
\end{fact}

\begin{remark}
	We note that if $T$ is NTP$_2$ and $\ind = \ind^i$, then $T$ is NIP (\cite[Theorem 4.3]{chernikov2012forking}).
\end{remark}

\subsection{Higher-arity generalizations of distality}\label{sec: n-dist}
 
We recall $n$-distality and strong $n$-distality, as introduced by Walker \cite{walker2023distality} (with $1$-distality corresponding to the usual notion of distality considered by Simon \cite{simon2013distal}).

\begin{definition}
    \begin{enumerate}
        \item Given a sequence $\Ical$ of tuples in $\cU$, a \emph{cut} $\cfrak$  is a partition of $\Ical$ into two disjoint sets $\Ical_0,\Ical_1$ where for each $i\in \Ical_0$ and $j\in \Ical_1$, $i<j$. We say $\cfrak$ is \emph{Dedekind} if $\Ical_0$ has infinite cofinality and $\Ical_1$ has infinite coinitiality as subsequences of $\Ical$. 
        \item Given a set $A$ and a cut $\cfrak = (\Ical_0,\Ical_1)$ in an $A$-indiscernible sequence of tuples $\Ical$, we say that a tuple $a$ \emph{inserts into $\Ical$ preserving $A$-indiscernibility} if the sequence $\Ical_0\frown (a) \frown \Ical_1$ is $A$-indiscernible. Informally, we call this new sequence \emph{the result of inserting $a$ into $\cfrak$}.
        \item  Given an infinite sequence of tuples $\Ical$, a formula $\varphi(x_0,\dots,x_{n-1})\in \Lcal$ is in the \emph{Ehrenfeucht-Mostowski} type of $\Ical$ if for each strictly increasing tuple $(i_0,\dots, i_{n-1})$ from $\Ical$, $\models \varphi(i_0,\dots, i_{n-1})$. We refer to this set of formulas by $EM(\Ical)$.  
    \end{enumerate}
\end{definition}

\begin{definition}
    \begin{enumerate}
        \item Given $n \geq 1$, an infinite indiscernible sequence $\Ical$ is \emph{strongly $n$-distal} if for any tuples $a, b_0,\dots,b_{n-1}$, $\Ical'\models EM(\Ical)$ indiscernible over $b_0,\dots,b_{n-1}$  and $\cfrak$ a Dedekind cut in $\Ical'$, if  $a$ inserts into $\cfrak$ preserving $(b_i)_{i\neq j, i< n}$-indiscernibility for each $1\leq j \leq n-1$, then $a$ inserts into $\cfrak$ preserving $(b_i)_{i< n}$-indiscernibility.
        \item $T$ is \emph{strongly $n$-distal} if every infinite indiscernible sequence is strongly $n$-distal.
    \end{enumerate}
\end{definition}

%\begin{definition}\label{distality def}
%    \begin{enumerate}
%        \item For any set $A$, an infinite $A$-indiscernible sequence $\Ical$ is \emph{distal} if whenever there is $A$-indiscernible $\Ical' \models EM(\Ical)$, $\cfrak$ a Dedekind cut in $\Ical'$ and $a$ that inserts into $\cfrak$ preserving $\es$-indiscernibility, we have that $a$ inserts into $\cfrak$ preserving $A$-indiscernibility.
%        \item $T$ is \emph{distal} if all infinite indiscernible sequences are distal.
%    \end{enumerate}
%\end{definition}

\begin{remark}
    For $n=1$, this generalizes the so-called ``external characterization of distality'' in NIP theories \cite{simon2013distal}.
  \end{remark}
  
$N$-distality, as defined in \cite{walker2023distality}, generalizes the ``internal characterization of distality''. It was originally defined for a single sequence with multiple cuts, but a version using mutually indiscernible sequences is easily shown to be equivalent at the level of theories. In particular, we incorporate into the definition here a simplification afforded by \cite[Proposition 3.12]{walker2023distality} for the NIP context.
% Since the only sequences we are interested in are indexed by the interval $[0,1]$, we give the following definition.

\begin{definition}\label{def: insterts}
\begin{enumerate}
    \item Given a set $A$, a tuple of sequences $(\Ical_i)_{i<n}$ is \emph{mutually $A$-indiscernible} if for each $j<n$, $\Ical_j$ is indiscernible over $A\cup [\bigcup (\Ical_{i\neq j})]$. A \emph{tuple of (Dedekind) cuts in $(\Ical_i)_{i<n}$} is an $n$-tuple $(\cfrak_i)_{i<n}$ where each $\cfrak_i$ is a (Dedekind) cut in $\Ical_i$. 
    \item Given $m\leq n$, $(\Ical_i)_{i<n}$ mutually $A$-indiscernible and $(\cfrak_i)_{i<n}$ a tuple of cuts in $(\Ical_i)_{i<n}$, we say that a tuple of parameters $(a_i)_{i<n}$ \emph{$m$-inserts into $(\Ical_i)_{i<n}$ preserving $A$-indiscernibility} if for each strictly increasing subtuple of length $m$ from $n$, $(j_0,\dots,j_{m-1})$, the result of  simultaneously  inserting $a_{j_k}$ into $\cfrak_{j_k}$ leaves $(\Ical_i)_{i<n}$ mutually $A$-indiscernible.
\end{enumerate}
\end{definition}

\begin{definition}\label{def: n-dist mut ind seqs}
    \begin{enumerate}
        \item A mutually indiscernible $(n+1)$-tuple of dense endless sequences $(\Ical_i)_{i\leq n}$ equipped with a tuple $(\cfrak_i)_{i<n}$ of Dedekind cuts $\cfrak_i$ in $\Ical_i$ is \emph{$n$-distal} if for any $(a_i)_{i\leq n}$ that $n$-inserts preserving mutual indiscernibility, we have that the whole tuple $(a_i)_{i\leq n}$ $(n+1)$-inserts preserving mutual indiscernibility.
        \item A mutually indiscernible $(n+1)$-tuple of dense endless sequences $(\Ical_i)_{i\leq n}$ is called \emph{$n$-distal} if any tuple of sequences $(\Ical'_i)_{i\leq n}$ formed by taking cuts in $(\Ical_i)_{i\leq n}$ and removing any realizations so the cuts become Dedekind is (together with these cuts) $n$-distal.
        \item $T$ is $n$-distal if any mutually indiscernible $(n+1)$-tuple of sequences is $n$-distal.
    \end{enumerate}
\end{definition}

Recall also a more general notion of $n$-dependence introduced by Shelah (with $1$-dependence corresponding to NIP):
\begin{definition}
A partitioned formula $\varphi\left(x;y_{1},\ldots,y_{n}\right)$ has the \emph{$n$-independence property} (with respect to a theory $T$), if in some model of $T$ there is a sequence of tuples $\left(a_{1,i},\ldots,a_{n,i}\right)_{i\in\omega}$
such that for every $s\subseteq\omega^{n}$ there is a tuple $b_{s}$ with the following property:
\[
\models\varphi\left(b_{s};a_{1,i_{1}},\ldots,a_{n,i_{n}}\right)\Leftrightarrow\left(i_{1},\ldots,i_{n}\right)\in s\mbox{.}
\]
Otherwise we say that $\varphi\left(x,y_{1},\ldots,y_{n}\right)$ is \emph{$n$-dependent}.
A theory is \emph{$n$-dependent} if it implies that every formula is
$n$-dependent. 
\end{definition}

Recall a characterization of $n$-dependence using random ordered hypergraph indiscernibles. We let $G_{n,p} = (G,R, P_0, \ldots, P_{n-1},<)$ denote the countable generic ordered $n$-partite $n$-hypergraph, we denote its language by $\mathcal{L}_{\opg}$.

\begin{fact}\cite[Proposition 5.2]{chernikov2014n}\label{fac: Gnp indisc char}
	Assume that $\varphi(x;y_0, \ldots, y_{n-1})$ is not $n$-dependent. Then in $\cU$ there is a $G_{n+1,p}$-indiscernible $(a_g)_{g \in G_{n+1,p}}$ so that $\varphi$ encodes the edge relation on it, i.e.~for any $g_0 \in P_0, \ldots, g_{n} \in P_n$ we have $\models \varphi(a_{g_0}, \ldots, a_{g_{n}}) \iff (g_0, \ldots, g_{n}) \in R$.
\end{fact}
The following was already included in \cite[Proposition 6.7]{walker2023distality}, but we include it here for completeness:
\begin{prop}\label{prop: n-dist implies n-dep}
	Let $T$ be a complete theory in a language $\mathcal{L}$. If $T$ is $n$-distal, then $T$ is $n$-dependent.
\end{prop}
\begin{proof}
	Assume that $T$ is not $n$-dependent, and let $\varphi$ and $(a_g)_{g \in G_{n+1,p}}$ be as given by Fact \ref{fac: Gnp indisc char}.
	As $G_{n,p}$ embeds every countable ordered partite hypergraph, we can choose $g_{0,j} \in P_0, \ldots, g_{n,j} \in P_n $ for $j \in \mathbb{Q}$ such that:
	\begin{align}
	&g_{i,j} < g_{i,j'} \iff j < j'\textrm{, for all }i<n \textrm{ and } j,j' \in \mathbb{Q};\\
		&(g_{0,j_0}, \ldots, g_{n,j_n}) \in R \iff (j_0, \ldots, j_{n}) = (0,1, \ldots, n).\label{eq: choice of R}
	\end{align}
	Let $I := \mathbb{Q}\setminus\{0, 1, \ldots, n\}$. Then we have the following:
	\begin{itemize}
		\item for any $j^* \in \{0,1, \ldots, n\}$, the sequence $(g_{0,j}, \ldots ,g_{n,j})_{j \in \mathbb{Q} \setminus \{j^*\}}$ is clearly quantifier-free indiscernible in the language $\mathcal{L}_{\opg}$; as $(a_g)_{g \in G_{n+1,p}}$ is $G_{n+1,p}$-indiscernible, this implies that the sequence $(a_{g_{0,j}}, \ldots, a_{g_{n,j}})_{j \in \mathbb{Q}\setminus\{j^*\}}$ is $\mathcal{L}$-indiscernible;
		\item the sequence $(a_{g_{0,j}}, \ldots, a_{g_{n,j}})_{j \in \mathbb{Q}}$ is not indiscernible, as by (\ref{eq: choice of R}) we have 
		$$\models \varphi(a_{g_{0,j_0}}, \ldots, a_{g_{n,j_n}}) \iff (j_0, \ldots, j_n) = (0,1, \ldots, n).$$
	\end{itemize}
	Then $T$ is not $n$-distal, witnessed by the sequence $(\bar{a}_j)_{j \in I}$ with $\bar{a}_j := (a_{g_{0,j}}, \ldots, a_{g_{n,j}})$.
\end{proof}

%
%We note the following basic properties, the first two of which are clear. 
%
%\begin{fact}
%\begin{enumerate}
%    \item If $T$ is strongly $n$-distal and $k\geq n$, then $T$ is $k$-strongly distal. 
%    \item $T$ is distal if and only if it is $1$-strongly distal. 
%    \item There are $2$-strongly distal theories with IP \cite{Walker2019}.
%\end{enumerate}
%\end{fact}
We will further discuss the (strong) $n$-distality hierarchy, restricting to NIP or even stable theories, in Section \ref{sec: n-dist and n-triv}.

\subsection{Cylinder intersection sets}\label{sec: cyl inters sets}

Let $X_1,\dots,X_k$ be  sets and write $X := X_1\times\cdots\times X_k$. For $x=(x_1,\dots,x_k)\in X$ and $i\in[k]:=\{1,\dots,k\}$, write
\[
x_{-i}:=(x_1,\dots,x_{i-1},x_{i+1},\dots,x_k)\in \prod_{j\neq i}X_j.
\]

\begin{defn}\label{def: cylinder inters sets}
Fix $i\in[k]$. A set $C_i\subseteq X$ is called a \emph{cylinder in direction $i$} if there exists a function 
$c_i:\prod_{j\neq i}X_j\to\{0,1\}$
such that 
$\1_{C_i}(x)=c_i(x_{-i})$.  A set $E\subseteq X$ is called a \emph{$k$-ary cylinder intersection} if it can be written as
$
E=\bigcap_{i=1}^k C_i$
for some cylinders $C_i$ in direction $i$. If $c_i$ is the indicator function of a set $A_i \subseteq \prod_{j\neq i}X_j$, we will also write $E = \bigwedge_{i=1}^{k} A_i$ to denote the cylinder intersection.
\end{defn}

\begin{defn}
Given $E \subseteq X$ and $C \subseteq X$, we say that $C$ is \emph{$E$-homogeneous} if $C \subseteq E$ or $C \subseteq X \setminus E$.	In particular, we will be interested in $E$-homogeneous cylinder intersection sets $C$.
\end{defn}

We will also consider an analog restricted to Boolean algebras of definable sets:

\begin{definition}\label{def: cylinder Bool alg}
	 As usual, $\Lcal_x(A)$ refers to the set of $\Lcal$-formulas with free variables from $x$ and parameters from $A$, which we identify, up to logical equivalence, with definable sets in $\cU$. Given a partitioned tuple of variables $x_0,\dots,x_{n-1}$ and $0 \leq k\leq n$, we denote by $\Lcal_{x_0,\dots,x_{n-1}}^k(A)$ the Boolean algebra generated by $\bigcup_{I\sub n, |I|=k} [\Lcal_{x_I}(A)]$. When the variable partition is clear from context we write $\Lcal^k_n(A)$ for $\Lcal_{x_0,\dots,x_{n-1}}^k(A)$.
\end{definition}

\section{$N$-determinacy for measures and $n$-distal regularity lemma}\label{sec: n-determinacy for measures and n-distal regularity lemma}

%\subsection{$N$-determinacy for measures and $n$-distal regularity lemma}

%\begin{cor}
%	Let $T$ be $m$-distal, $\mu_0(x_0), \ldots, \mu_{n-1}(x_{n-1})$ generically stable and $R(x_0, \ldots, x_{n-1})$ a definable relation. Then for any $\varepsilon > 0$ there exist some sets $D^\varepsilon_{-}, D^{\varepsilon}_+ \in \mathcal{B}^x_{n,m}$ such that $D^\varepsilon_{-} \subseteq R \subseteq D^{\varepsilon}_+$ and $\mu( D^{\varepsilon}_+ \setminus D^\varepsilon_{-}) < \varepsilon$.
%\end{cor}
%\begin{proof}
%	Immediate by Proposition \ref{prop: m-det} and e.g. \cite[Theorem 3.7]{starchenko2016nip}.
%\end{proof}

%***
%Prove this at least in NIP theories first? Or we expect there are no strictly $n$-distal NIP theories? Well, there should be NIP theories of arity $n$ for each $n$ I guess. So trees are strictly $3$-distal? Take an ordered $C$-relation indiscernible based on a $C$-relation.
%
%For stable - Hrushovski constructions of varying arities?

%In stable (or NIP), would show that triviality of forking is equivalent to determinacy for measures...

%N-amalgamation/n-stationarity/ n-dependence/n-distality, etc for measures in simple theories

%***
%*** for g.s. measures - always have ``higher amalgamation'' by hypergraph removal lemma.

%
%
%Given a 2-dependent formula $\varphi(x,y;z)$ and a pair of mutually indiscernible sequences indexed by $[0,1]$, does $\varphi$ cut out measurable sets? (up to arbitrary unary sets somehow, maybe fix some ultrafilters on those)

\subsection{Indiscernible measures in strongly $n$-distal NIP theories}\label{sec: Indiscernible measures in strongly n-distal NIP theories}

In this section we generalize some results in NIP and strongly $n$-distal theories from indiscernible sequences to indiscernible measures.

 Typically we used $\Ical$ to denote a sequence of tuples in $\cU$, but we may also use it for the indexing sequence, for example using the notation $(b_i)_{i\in \Ical}$ for a sequence of tuples. For shorthand, whenever $\Jcal$ is a subsequence of $\Ical$, we may write $b_\Jcal := (b_i)_{i\in \Jcal}$. The following fact is standard (see e.g.~\cite[Proposition 3.3]{hrushovski2008groups}).
\begin{fact}\label{bounding disjoint positive measure sets}
    For all $\varepsilon>0$ and $m\in \omega$, there is $n\in \omega$ such that whenever $\mu(x)\in \Mfrak(\Mbb)$ and $(\varphi_i(x))_{i<n} \in \Lcal(\Mbb)^n$ with $\mu(\varphi_i)\geq \varepsilon$ for all $i$, for some $I\sub [n]$ with $|I|= m$ we have that $\mu(\bigwedge\limits_{i\in I}\varphi_i)>0$.
\end{fact}
   \noindent  Given a measure $\mu\in \Mfrak_x(A)$, a type $p\in S_x(A)$ is in the \emph{support} of $\mu$ if $\mu(\varphi)>0$ for each $\varphi \in p$. We let $S(\mu) \subseteq S_x(A)$ be the (closed, nonempty) set of types in the support of $\mu$.

\begin{definition}
Fix an infinite indexing sequence $\Ical$, a global measure $\mu$ and $B \subseteq \cU$.
\begin{enumerate}
    \item A sequence $b_\Ical$ is \emph{$\mu(x)$-indiscernible} if for each same-length increasing tuples $\ib$, $\jb$ from $\Ical$ and each compatible formula $\varphi(x,\zb) \in \mathcal{L}$, $\mu(\varphi(x,b_{\ib}))= \mu(\varphi(x,b_{\jb}))$.
    \item A measure $\mu(x_\Ical)$ (in an infinite tuple of variables $x_\Ical$) is \emph{indiscernible over $B$} if for each same-length increasing tuples $\ib$, $\jb$ from $\Ical$, and each compatible formula $\varphi(y) \in \Lcal(B)$, $\mu(\varphi(x_{\ib}))=\mu(\varphi(x_{\jb}))$.
    \item A measure $\mu(x_\Ical)$ is \emph{totally indiscernible over $B$} if for each same-length tuples $\ib$, $\jb$ from $\Ical$ (not necessarily increasing), and each compatible formula $\varphi(y) \in \Lcal(B)$, $\mu(\varphi(x_{\ib}))=\mu(\varphi(x_{\jb}))$.
\end{enumerate}
\end{definition}

\begin{lemma} \label{Artem lemma}
    Suppose $T$ has NIP and fix $\mu(x)\in \Mfrak(\Mbb)$. If $\Ical$ indexes an infinite $\mu$-indiscernible sequence $(b_i)_{i\in\Ical}$, then whenever $p\in S(\mu)$ and $a\models p|_{b_\Ical}$ (the global type $p$ restricted  to parameters from $b_\Ical$), we have that $b_\Ical$ is $a$-indiscernible.
\end{lemma}

\begin{proof}
    Take $\psi(x,y)\in \Lcal$ and $J,J'\sub \Ical$ of length $|y|$. 

 \noindent   \textit{Claim.} $\mu(\psi(x,b_J)\triangle \psi(x,b_{J'}))=0$.

   \noindent  \textit{Proof of claim.} Suppose otherwise. We can assume that $\Ical$ has no maximal element (inverting the order and passing to a subsequence if necessary). Hence we can find, for each $i\in \omega$, $J_i,J_i'$ copies of $J,J'$ (i.e.~the order type of $J\cup J'$ is identical to the order type of $J_i\cup J_i'$) such that $J_i\cup J'_i < J_{i+1}\cup J_{i+1}'$ for all $i \in \omega$. By $\mu$-indiscernibility, for each $i$, $\mu(\psi(x,b_{J_i})\triangle \psi(x,b_{J_i'}))=\varepsilon$ for some fixed $\varepsilon >0$. Hence, by Fact \ref{bounding disjoint positive measure sets}, compactness and indiscernibility of $\Ical$, we can choose $a\models \bigwedge\limits_{i\in \omega} (\psi(x,b_{J_i})\triangle \psi(x,b_{J'_i}))$. By passing to a subsequence, we may further assume that for some $t \in \{0,1\}$, $a\models \bigwedge\limits_{i\in \omega} (\psi^t(x,B_{J_i})\wedge \psi^{1-t}(x,B_{J_i}))$. Let $c_i:=b_{J_i}$ if $i$ is even, and $c_i=b_{J_i'}$ if $i$ is odd. This sequence $(c_i : i \in \omega)$ is indiscernible (by our choice of $J_i,J_i'$ and indiscernibility of $\Ical$) and admits infinite alternation with respect to $\psi(a,y)$, contradicting NIP.
    \hfill$\blacksquare_{\textit{Claim }}$

    As  $p \in S(\mu)$, $a\models p|_{b_\Ical}$ and $\neg(\psi(x,b_J)\triangle\psi(x,b_{J'}))\equiv \psi(x,b_J)\leftrightarrow\psi(x,b_{J'})$, by the Claim we have  $\models \psi(a,b_J)\leftrightarrow\psi(a,b_{J'})$ for all $J,J'\sub \Ical$ of the same length. And as $\psi$ was arbitrary, this implies that $\mathcal{I}$ is $a$-indiscernible. 
\end{proof}

\noindent Using this, we see that strong $n$-distality lifts from sequences indiscernible over parameters to sequences indiscernible over measures:
\begin{lemma}\label{measure strong 2-distality}
    Assume $T$ has NIP. Suppose that $\Ical := \Ical_0+*+\Ical_1$ indexes an infinite strongly $n$-distal indiscernible sequence $(b_i)_{i\in \Ical}$ and $\omega(x_0,\dots, x_{n-1})\in \Mfrak(\Mbb)$ (where each $x_i$ is possibly an infinite tuple of variables). If $b_{\Ical_0+\Ical_1}$ is $\omega$-indiscernible and $b_\Ical$ is $\omega|_{(x_i)_{i\neq j}}$-indiscernible for each $j<n$, then $b_\Ical$ is $\omega$-indiscernible.
\end{lemma}

\begin{proof}
    Suppose otherwise, let $\bar{x} := (x_0,\dots, x_{n-1})$. Then there are finite $J,J'\sub \Ical$ of the same order type and $\psi \in \mathcal{L}$ with $\omega(\psi(\xb,b_J))-\omega(\psi(\xb,b_{J'}))>0$. In particular $\omega(\psi(\xb,b_J)\triangle\psi(\xb,b_{J'}))>0$. Hence, there is $p(\bar{x})\in S(\omega)$ with $\psi(\xb,b_J)\triangle\psi(\xb,b_{J'})\in p$. Easily, $p|_{(x_i)_{i\neq j}}\in S(\omega|_{(x_i)_{i\neq j}})$ for each $j$. So, by Lemma \ref{Artem lemma}, fixing $\ab = (a_i)_{i<n}\models p|_{b_\Ical}$, $\Ical$ is $(a_i)_{i\neq j}$-indiscernible for each $j$ and $\Ical_0+\Ical_1$ is $\ab$-indiscernible. By strong $n$-distality, then, $\Ical$ is $\ab$-indiscernible. But this is a contradiction since $\models \psi(\ab,b_J)\triangle\psi(\ab,b_{J'})$.
\end{proof}

%\begin{proof}
%Say $\mathcal{I}_t = (b_i : i \in I_t)$ and $\mathcal{I} = (b_i : i \in I)$ for $I = I_1 + (\ast) + I_2$ and $b_{\ast} = b$. 
%	Assume the conclusion fails, then there are some  $\varphi(x,\bar{y}) \in L$ and finite $J,J' \subseteq I$ with $|J| = |J'|$ so that $|\omega(\varphi(\bar{x}; b_{J})) - \omega (\varphi(\bar{x}; b_{J'}))| > 0$. Then $\omega \left(\varphi(\bar{x}; b_{J}, b) \triangle \varphi(\bar{x}; b_{J'}, b) \right) > 0$, so there exists $p(\bar{x}) \in S(\omega)$ so that $\left(\varphi(\bar{x}; b_{J}, b) \triangle \varphi(\bar{x}; b_{J'}, b) \right) \in p$. Note that $ p|_{x_{[k] \setminus \{t\}}}  \in S \left(\omega|_{x_{[k] \setminus \{t\}}} \right)$ for all $t \in [k]$.
%	 Let $\bar{a} = (a_1, \ldots, a_k) \models p|_{\mathcal{I}}$, in particular $\mathcal{I}_1 + b + \mathcal{I}_2$ is not $\bar{a}$-indiscernible. By assumption and Lemma \ref{lem: in support implies indiscernible} we have that $\mathcal{I}_1 + \mathcal{I}_2$ is $\bar{a}$-indiscernible and $\mathcal{I}_1 + b + \mathcal{I}_2$ is $(a_{[k] \setminus t})$-indiscernible, so by strong $k$-distality we must have $\mathcal{I}_1 + b + \mathcal{I}_2$ is $\bar{a}$-indiscernible, a contradiction.
%\end{proof}

We introduce generalizations of some notions for sequences of tuples from \cite[Section 2]{simon2013distal}  to measures. Given $\mu$ a measure in the variables $(x_i)_{i\in \Ical}$ and $\Jcal \sub \Ical$, we refer to the natural restriction $\mu((x_i)_{i\in \Jcal})$ by $\mu_\Jcal$.

\begin{definition} Fix an infinite indexing sequence $\Ical$. 
    \begin{enumerate}
        \item  An indiscernible measure $\mu_\Ical$ and an indiscernible sequence $b_\Ical$ are called \emph{$\mu_\Ical$-weakly linked} if for each $\Ical_1, \Ical_2 \sub \Ical$ with $\Ical_1 \cap  \Ical_2 = \emptyset$, $\mu_{\Ical_1}$ is indiscernible over $b_{\Ical_2}$. (Note that we require $\Ical_1\cap\Ical_2 =\es$ but not $\Ical_1 < \Ical_2$.)
        \item  An indiscernible measure $\mu_\Ical$ and an indiscernible sequence $(b_i)_{i\in \Ical}$ are called \emph{$b_{\Ical}$-weakly linked} if  for each $\Ical_1, \Ical_2 \sub \Ical$ with $\Ical_1 \cap  \Ical_2 = \emptyset$, $b_{\Ical_2}$ is $\mu_{\Ical_1}$-indiscernible.
    \end{enumerate}
\end{definition}

Generalizing a standard fact for indiscernible sequences, NIP also gives a bound on alternation for indiscernible measures. It is noted in \cite[Corollary 2.12]{hrushovski2013generically}, as a corollary of the two-sided measure theoretic result from \cite[Theorem 5.3]{ben2009continuous}:
\begin{fact}\label{alternation for a measure over a tuple}
    Assume that $T$ has NIP. Suppose $\mu\in \Mfrak_{(x_i)_{i\in \omega}}(\Mbb)$ is indiscernible, $b\in \Mbb$, $\varphi(x,y)\in \Lcal$ and $\varepsilon>0$. Then there is $N\in \omega$ such that whenever $i_0<\dots <i_{N-1} \in \omega$, there is $0<k<N$ such that $|\mu(\varphi(x_{i_{k-1}},b))-\mu(\varphi(x_{i_k},b))|<\varepsilon$.
\end{fact}

\noindent We will need some auxiliary lemmas generalizing various results about indiscernible sequences in NIP theories to indiscernible measures.

\begin{lemma}\label{shrinking of indiscernibles and total indiscernibility for measures}
    Assume that $T$ has NIP and $\Ical$ is infinite. If $\mu_\Ical$ is totally indiscernible and $b\sub \Mbb$ finite, then there is $\Ical_1\sub \Ical$ with $|\Ical_1|\leq |T|$ such that $\mu_{\Ical\sm \Ical_1}$ is totally indiscernible over $b$.
\end{lemma}

\begin{proof}
    We may assume that $\Ical$ is dense and endless since re-enumerating the variables preserves total indiscernibility.  Fix $\varphi(x_0,\dots, x_{k-1},;b)\in \Lcal(b)$.     By Fact \ref{alternation for a measure over a tuple} and total indiscernibility, for each $n\in \Nbb$ there is an integer $N_n$ such that there exists a set of $N_n$ disjoint $k$-tuples $(\ib)_{i<N_n}$ from $\Ical$ with $|\mu(\varphi(x_{\ib},b))-\mu(\varphi(x_{\overline{i+1}},b))|\geq \frac{1}{2n}$, but there does not exist a set of $N_n+1$ disjoint $k$-tuples from $\Ical$ with the same property. Hence, removing finitely many indices, we may ensure that $|\mu(\varphi(x_{\ib},b))-\mu(\varphi(x_{\jb},b))|< \frac{1}{n}$ for each $\ib,\jb$ remaining $k$-tuples from $\Ical$. Hence, choosing these indices for each $n$ and removing the countably many variables corresponding to these indices, we can ensure that the value $\mu(\varphi(x_{\ib},b))$ is constant for all $k$-tuples $\ib$. Now, repeating this for every formula $\varphi \in \Lcal(b)$, the result follows. 
\end{proof}

\begin{lemma}\label{existence of funny measure}
    Fix an indexing sequence $\Ical$, a sequence  $c_\Ical$ of tuples in $\cU^z$, and a global measure $\mu(x_\Ical)$. Then we may define a measure $\eta_\Ical$ in the variables $(x_i,z_i)_{i\in \Ical}$ (where each $z_i$ is a copy of $z$) by $\eta(\varphi(x_\Ical,z_\Ical)) := \mu(\varphi(x_\Ical,c_\Ical))$. 
\end{lemma}
\begin{proof}
    It is easy to check each condition of the definition of a finitely additive probability measure. 
\end{proof}
\noindent The following is a generalization of \cite[Corollary 2.15]{simon2013distal} from $1$-distality and sequences to $n$-distality and measures:
\begin{lemma} \label{key lemma with strong indiscernibility}
    Assume that $T$ has NIP and $\Ical$ is infinite. Fix a tuple of variables, $\ub$, partitioned into subtuples $\ub = x_0\dots x_{n-1}$ and form a sequence of copies of $\ub$, $(\ub_i)_{i\in \Ical}$. Denote by $\ub_i^j$ the copy of $x_j$ in $\ub_i$. Suppose that a global measure $\omega(\ub_\Ical)$ is totally indiscernible (over $\emptyset)$, an indiscernible sequence $c_\Ical$ is strongly $n$-distal and the measure $\eta((\ub_i,z_i)_{i\in \Ical})$ defined by $\eta(\varphi(\ub_\Ical,z_{\Ical})):=\omega(\varphi(\ub_\Ical,c_\Ical))$ (Lemma \ref{existence of funny measure}) is indiscernible. If the sequence $c_\Ical$ is $\omega|_{(\ub_\Ical^j)_{j\neq k}}$-indiscernible for each $k<n$, then $c_\Ical$ is $\omega$-indiscernible and $\omega(\ub_\Ical)$ is indiscernible over $c_\Ical$.
\end{lemma}

\begin{proof}
    Throughout, we use the consequence of strong $n$-distality given by Lemma \ref{measure strong 2-distality}.

\noindent    \textit{Claim 1.} We may assume that the order $\Ical$ is dense and arbitrarily large. 

\noindent    \textit{Proof of Claim 1.} 
    Assume that we have a counterexample to the lemma indexed by $\Ical$, and let $\Jcal$ be an arbitrary linear order. Assume that the counterexample has $c_\Ical$ not $\omega$-indiscernible (the case that the counterexample is given instead by $\omega$ not being indiscernible over $c_\Ical$ is exactly analogous). This means we can find some finite increasing subtuple $\mathcal{I}_0$ of $\mathcal{I}$, say of length $l$, formula $\psi(\bar{u}_{\mathcal{I}_0},\yb) \in \mathcal{L}$ where $\yb$ is the order type of a finite subtuple of $c_\Ical$, and $\cb, \db$ subtuples of $c_{\Ical_0}$ corresponding to $\yb$, and $\varepsilon>0$ such that $
    \omega(\psi(\ub_{\Ical_0},\cb))-\omega(\psi(\ub_{\Ical_0},\db))>\varepsilon$. 

 By indiscernibility of $\eta$, given any increasing $l$-tuple $\mathcal{I}'_0$ from $\mathcal{I}$, it also witnesses that $c_\Ical$ is not $\omega$-indiscernible with respect to $\psi$ and $\varepsilon$, in the same way as $\mathcal{I}_0$ (i.e.~the copies of $\cb$, $\db$ are chosen in $c_{\mathcal{I}'_0}$ in the same way)  --- call this observation $(*)$.

 By Ramsey and compactness, let $(c'_i : i \in \Jcal)$  be an indiscernible sequence realizing $\EM(c_\Ical)$. Fix an arbitrary increasing $l$-tuple $\Jcal_0$ from $\Jcal$. 
    
    We use compactness of the space of global Keisler measures in the variables $(\ub_j)_{j \in \Jcal}$. Recall (see e.g.~\cite[Proposition 2.11]{chernikov2022definable}) that the topology on this space is generated by the basic open sets of the form
    \[
    \left\{\omega' \in \mathfrak{M}_{\ub_\Jcal}(\cU) : \bigwedge_{i =1}^{N} r_i<\omega'(\varphi_i(\ub_\Jcal)) <s_i \right\}
    \]
   for some $N\in \Nbb$, $r_i, s_i\in [0,1]$ and $\varphi_i\in \Lcal(\Mbb)$. Observe that the set of $\omega' \in \mathfrak{M}_{\ub_\Jcal}(\cU)$ satisfying the following  conditions is closed in the topology (expressed by the intersection of complements of basic open sets):
   \begin{enumerate}
        \item the measure $\eta'((\ub_i,z_i)_{i\in \Jcal})$ defined by $\eta'(\varphi(\ub_\Jcal,z_{\Jcal})):=\omega'(\varphi(\ub_\Jcal,c'_\Jcal))$ is indiscernible;
        \item $\omega'(\ub_\Jcal)$ is totally indiscernible;
        \item $(c'_\Jcal)$ is $\omega'|_{(\ub_\Jcal^j)_{j\neq k}}$-indiscernible for each $k<n$;
        \item $\Jcal_0$ witnesses that $c'_\Jcal$ is not $\omega'(\ub_\Jcal)$-indiscernible with respect to $\psi$ and $\varepsilon$. 
    \end{enumerate}

    Thus, by compactness of $\mathfrak{M}_{\ub_\Jcal}(\cU)$, we  only need to show that there is a measure $\omega'$  satisfying these conditions restricted to a finite set of indices $\Jcal_1 \subseteq \Jcal$, which without loss of generality we assume contains $\Jcal_0$ (closed sets are intersections of complements of basic open sets whose conditions only concern finitely many indices). So let $\Ical_1$ be a copy of $\Jcal_1$ in $\Ical$ (exists as $\Ical$ is infinite) and fix an automorphism $\sigma$ of $\cU$ mapping $c'_{\Ical_1}$ to $c_{\Jcal_1}$. Then we let $\omega''\in \Mfrak_{(\ub_i)_{i\in \Jcal_1}}(\Mbb)$ be given by $\omega''(\varphi((\bar{u}_i : i \in \Jcal_1), \bar{a})) := \omega(\varphi((\bar{u}_i : i \in \Ical_1), \sigma(\bar{a})))$ for all $\varphi \in \mathcal{L}$ and tuples $\bar{a}$ in $\cU$. Let $\omega' \in \mathfrak{M}_{\ub_\Jcal}(\cU)$ be obtained by extending $\omega''$ arbitrarily (e.g.~we can first take the product measure with arbitrary measures on the new variables $\bar{u}_{\Jcal \setminus \Jcal_1}$, e.g.~measures concentrated on a point, and then extend from the product subalgebra to the full algebra $\mathcal{L}_{\bar{u}_{\Jcal}}(\cU)$ by Fact \ref{existence of measure extensions}). Using observation $(*)$ and  preservation under automorphism we conclude that $\omega'$ satisfies the given finite set of conditions involving indices from $\Jcal_1$. 
    \hfill$\blacksquare_{\textit{Claim }1}$

Hence, we continue with the proof assuming that $\Ical$ is dense and $|\Ical| \geq |T|^{+}$.

\noindent    \textit{Claim 2.} $\omega_\Ical$, $(c_i)_\Ical$ are $\mu_\Ical$-weakly linked.

 \noindent    \textit{Proof of Claim 2.}
 Let $\Ical_1, \Ical_2 \sub \Ical$ with $\Ical_1 \cap  \Ical_2 = \emptyset$ be given. Without loss of generality $\Ical_1, \Ical_2$ are finite. Let $C := (c_i)_{i\in \Ical_2}$. By Lemma \ref{shrinking of indiscernibles and total indiscernibility for measures}, we can find a set $\Ical' \sub \Ical$ of size at most $|T|$ such that $\mu|_{\Ical\sm \Ical'}$ is (totally) indiscernible over $C$. As the order is dense, there is an order preserving partial bijection of $\mathcal{I}$ fixing $\mathcal{I}_2$ pointwise and sending $\mathcal{I}_1$ to a subset of $\Ical\sm \Ical'$. Then, by indiscernibility of $\eta$, it follows that  $\mu_{\Ical_1}$ is also indiscernible over $(c_i)_{i\in \Ical_2}$.
    \hfill$\blacksquare_{\textit{Claim }2}$

\noindent     \textit{Claim 3.} $\omega_\Ical$, $(c_i)_\Ical$ are $b_{\Ical}$-weakly linked.

 \noindent    \textit{Proof of Claim 3.} This follows combining Claim 2 and indiscernibility of $\eta$.
    \hfill$\blacksquare_{\textit{Claim }3}$
    
    It is now fairly easy to conclude. We again assume the order is dense. Then pick a finite $\Ical_1\sub\Ical$ and note that $(c_i)_{\Ical\sm \Ical_1}$ is $\omega_{\Ical_1}$-indiscernible by $b_{\Ical}$-weakly linked. By applying Lemma \ref{measure strong 2-distality} a finite number of times, the whole sequence $(c_i)_{i \in \Ical}$ is $\omega_{\Ical_1}$-indiscernible too. Hence, by finiteness of formulas, $(c_i)_{i \in \Ical}$ is $\omega$-indiscernible. 
    
    It remains to show that $\omega$ is $c_\Ical$-indiscernible. Fix finite tuples $\ib$, $\ib'$ from $\Ical$ with the same order type, and fix $\varphi(\bar{u}_{\ib},\cb)\in \Lcal \left( (c_i)_\Ical \right)$.  Then, by $\mu_\Ical$-weakly linked, we have $\omega(\varphi(\ub_{\ib},\cb'))=\omega(\varphi(\ub_{\ib'},\cb'))$ for any copy $\cb'$ of $\cb$ in $\mathcal{I}$ away from $\ub_{\ib},\ub_{\ib'}$.  Then, moving $\cb'$  to $\cb$  by $\omega$-indiscernibility of $(c)_{\Ical}$, we conclude  that  $\omega(\varphi(\ub_{\ib},\cb))= \omega(\varphi(\ub_{\ib'},\cb))$.
\end{proof}

\subsection{$N$-smooth measures in strongly $n$-distal NIP theories}\label{sec: n-smooth measures in strongly n-distal NIP theories}

In this section we generalize Simon's result that in a distal theory all generically stable measures are smooth \cite[Proposition 2.27]{simon2013distal} to strongly $n$-distal NIP theories.

%\begin{definition}
%    Fix a Boolean subalgebra $\Afrak$ of $\Lcal_x(\Mbb)$. A global measure $\omega(x)$ is \emph{$\Afrak$-smooth} if there is a small model $\M$ such that for any $\omega'(x) \in \Mfrak(\cU)$ extending $\omega|_{\Bfrak(\Afrak \cup \Lcal_x(\M))}$ we have that $\omega'=\omega$. In this case we say $\omega$ is \emph{$\Afrak$-smooth over $\M$}.
%\end{definition}

\begin{definition}
   Given a subset $\Gamma$ of a larger ambient Boolean algebra, we denote by $\Bfrak(\Gamma)$ the smallest Boolean subalgebra containing $\Gamma$ --- i.e.~the Boolean algebra generated by $\Gamma$. 
    Given Boolean subalgebras $\Afrak_i$ of an ambient larger Boolean algebra and measures $\mu_i$ on $\Afrak_i$ for $i \in I$, we say that the tuple of measures $(\mu_i : i \in I)$ is \emph{compatible} if there is a measure $\nu$ on $\Bfrak(\bigcup_{i \in I} \Afrak_i)$ extending $\mu_i$ for all $i \in I$.
\end{definition}

We introduce a  higher arity generalization of the notion of a \textit{smooth} measure:
\begin{definition}\label{def: n-smooth}
     A global measure $\omega(x_0,\dots,x_{n-1})\in \Mfrak(\Mbb)$ is \emph{$n$-smooth over $A \subseteq \cU$} if for any $\omega'(x_0,\dots,x_{n-1})\in \Mfrak(\Mbb)$ so that $\omega'|_{A} = \omega|_{A}$ and $\omega'|_{(x_i : i \neq j)} = \omega|_{(x_i : i \neq j)}$ for all $0 \leq j < n$, we must have $\omega' = \omega$. And $\omega$ is \emph{$n$-smooth} if it is $n$-smooth over some small $A \subset \cU$.
\end{definition}

\begin{remark}\label{rem: weak n-smooth}
	We will also say that $\omega(x_0,\dots,x_{n-1})\in \Mfrak(\Mbb)$ is \emph{weakly $n$-smooth over $A$} if for any $\omega'(x_0,\dots,x_{n-1})\in \Mfrak(\Mbb)$, if
	$$\omega'|_{\mathfrak{B} \left(\mathcal{L}^{n-1}_{x_0, \ldots, x_{n-1}}(\cU) \cup \mathcal{L}_{x_0, \ldots, x_{n-1}}(A) \right)} = \omega|_{\mathfrak{B} \left(\mathcal{L}^{n-1}_{x_0, \ldots, x_{n-1}}(\cU) \cup \mathcal{L}_{x_0, \ldots, x_{n-1}}(A) \right)}$$
	then $\omega' = \omega$. 
	
	And we say that $\omega(x_0,\dots,x_{n-1})\in \Mfrak(\Mbb)$ is \emph{strongly $n$-smooth over $A \subseteq \cU$} if for any small tuple $a$ and any $\omega'(x_0,\dots,x_{n-1})\in \Mfrak(A a)$ so that $\omega'|_{A} = \omega|_{A}$ and $\omega'|_{(x_i : i \neq j)} = (\omega|_{(x_i : i \neq j)})|_{Aa}$ for all $0 \leq j < n$, we must have $\omega' = \omega|_{Aa}$.

	Note that obviously strongly $n$-smooth implies $n$-smooth implies weakly $n$-smooth.
\end{remark}

\begin{remark}
    A  global measure $\omega(x_0)$ is smooth, in the usual sense, exactly when it is $1$-smooth, if and only if it is weakly $1$-smooth (in which case the subalgebra $\Lcal^{0}_{x_0}(\Mbb)$ is trivial), if and only if it is strongly $n$-smooth.
\end{remark}

The following lemma is a generalization of  \cite[Theorem 3.16]{keisler1987measures} (see also \cite[Proposition 7.9]{simon2015guide}) in the case $n=1$:

\begin{lemma}\label{existence of smooth extensions of products}
    Assume $T$ has NIP, $n \geq 1$ and $\M \models T$. 
    Given any global measure $\omega'(x_0, \ldots, x_{n-1})$, there is a  global measure $\omega(x_0, \ldots, x_{n-1})$  which is $n$-smooth over some small $\mathcal{N} \succeq \M$ and satisfies $\omega|_{\M} = \omega'|_{\M}$ and $\omega|_{(x_i : i \neq j)} = \omega'|_{(x_i:i \neq j)}$ for all $j < n$. 
\end{lemma}

\begin{proof}
    Suppose otherwise, and let $\bar{x} = (x_0, \ldots, x_{n-1})$. We choose an increasing continuous sequence of small models $(\M_\alpha)_{\alpha< |T|^+}$ and a sequence of measures $(\omega_\alpha(\bar{x}))_{\alpha< |T|^+}$ by with $\omega_\alpha \in \Mfrak_{\bar{x}}(\M_\alpha)$,  $\omega_\alpha$ compatible with $(\omega'|_{(x_i : i \neq j)} : j < n)$ and $\omega_{\alpha}|_{\M_\beta} = \omega_\beta$ for all $\beta \leq  \alpha < |T|^{+}$, as follows.

    Let $\M_0 := \M$ and $\omega_0 := \omega'|_\M$. Suppose $1 \leq \alpha < |T|^{+}$ and we have chosen $\M_{\beta}, \omega_{\beta}$ for $\beta < \alpha$. Let also $\omega'_{\beta}(\bar{x})$ be a global extension of $\omega_{\beta}$ and of $(\omega'|_{(x_i : i \neq j)} : j < n)$ (exists by the compatibility requirement and Fact \ref{existence of measure extensions}).
    
    Assume first that $\alpha = \gamma + 1$ is a successor ordinal. By supposition, there exist $\varphi_\alpha(\bar{x},b_{\alpha})\in \Lcal(\Mbb)$, $\varepsilon_\alpha>0$ and global extensions $\omega^0_{\alpha}, \omega^1_{\alpha} \in \Mfrak(\Mbb)$ of $\omega_{\gamma}$, with each $\omega^i_\alpha$ also extending each of $(\omega'|_{(x_i : i \neq j)} : j < n)$, and so that $\omega^1_\alpha(\varphi_\alpha(x,b_\alpha))-\omega^0_\alpha(\varphi_\alpha(x,b_\alpha))\geq 4\varepsilon_\alpha$ (note that there exists at least one such global extension $\omega'_{\gamma}$, hence there have to be two distinct ones to fail the conclusion of the lemma). Let $\omega'_{\alpha} :=  \frac{1}{2}(\omega_\alpha^0 + \omega_\alpha^1) \in \mathfrak{M}_{\bar{x}}(\cU)$,   note that $\omega'_{\alpha}$ extends each of $(\omega'|_{(x_i : i \neq j)} : j < n)$. Take $\M_{\alpha}$ to be a small model containing $\M_{\gamma} b_\alpha$ and let $\omega_{\alpha} := \frac{1}{2}(\omega_\alpha^0 + \omega_\alpha^1)|_{\M_{\alpha}} \in \mathfrak{M}_x(\M_{\alpha})$. Clearly $\omega_\alpha$ extends $\omega_{\gamma}$, and is compatible with $(\omega'|_{(x_i : i \neq j)} : j < n)$ witnessed by $\omega'_{\alpha}$.
    
    Assume now that $\alpha$ is a limit ordinal. Then we let $\M_{\alpha} := \bigcup_{\beta < \alpha}\M_{\beta}$ and $\omega_{\alpha} := \bigcup_{\beta < \alpha}\omega_{\beta} \in \mathfrak{M}_{\bar{x}}(\M_{\alpha})$. Note that the set of measures in $\mathfrak{M}_{\bar{x}}(\cU)$ that both  extend $\omega_{\alpha}$ and each of $(\omega'|_{(x_i : i \neq j)} : j < n)$ is closed, given by the intersection of complements of basic open sets (see the proof of Lemma  \ref{key lemma with strong indiscernibility}). Any finite list of these basic open sets only involves finitely many parameters from $\M_{\alpha}$, hence there is some $\beta < \alpha$ so that $\omega'_{\beta}$ is in the intersection of their complements, hence by compactness of $\mathfrak{M}_{\bar{x}}(\cU)$, we find a global measure $\omega'_{\alpha}$ extending both $\omega_{\alpha}$ and each of $(\omega'|_{(x_i : i \neq j)} : j < n)$.

    Now, for each $\alpha < |T|^{+}$ and $\theta(x)\in \Lcal(\M_\alpha)$, we have that  $\omega^i_{\alpha+1}(\theta(x)\triangle \varphi_{\alpha+1}(x,b_{\alpha+1}))\geq 2\varepsilon_{\alpha+1}$ for at least one $i \in \{0,1\}$.
      Indeed, we have $\omega^0_{\alpha+1}(\theta(x)) = \omega^1_{\alpha+1}(\theta(x))$ (as  $\omega^0_{\alpha+1}|_{\M_\alpha} = \omega^1_{\alpha+1}|_{\M_\alpha}$), so if $\omega^i_{\alpha+1}(\theta(x)\triangle \varphi_{\alpha+1}(x,b_{\alpha+1})) < 2\varepsilon_{\alpha+1}$ for both $i \in \{0,1\}$ we would have $\omega^1_{\alpha+1}(\varphi_{\alpha+1}(x,b_{\alpha+1}))-\omega^0_{\alpha+1}(\varphi_{\alpha+1}(x,b_{\alpha+1})) < 4\varepsilon_{\alpha+1}$, contradicting the assumption. Hence $\omega_{\alpha+1}(\theta(x)\triangle \varphi_{\alpha+1}(x,b_{\alpha+1}))\geq \varepsilon_{\alpha+1}$. 

    Noting that we may assume that all $\varepsilon_\alpha$ are rational, and $|T|^+$ is uncountable, passing to a subsequence we may assume that $\varphi_\alpha =\varphi$ and $\varepsilon_\alpha =\varepsilon$ for all $\alpha$. Let $\omega^*:= \cup_{\alpha<|T|^+} \omega_\alpha$, a measure over a small model $\cup_{\alpha<|T|^+} \M_\alpha$. Also, for any $\alpha<\beta$, $\omega^*(\varphi(x,b_\alpha)\triangle \varphi(x,b_\beta))\geq \varepsilon$. Using Ramsey and compactness (with the measure incorporated into the language), we may assume that the sequence $(b_\alpha)$ is indiscernible. Then, by NIP, the partial type $\{\varphi(x,b_{2\alpha})\triangle \varphi(x,b_{2\alpha+1}): \alpha<\omega\}$ is inconsistent.  Using Fact \ref{bounding disjoint positive measure sets} and compactness, this is a contradiction.
\end{proof}

\begin{remark}
 The same proof shows that if $T$ has NIP, given any small model $\M \models T$ and a global measure $\omega'(x_0, \ldots, x_{n-1})$, there is a  global measure $\omega(\bar{x})$  which is weakly $n$-smooth over some small $\mathcal{N} \succeq \M$ and satisfies $\omega|_{\mathfrak{B} \left(\mathcal{L}^{n-1}_{x_0, \ldots, x_{n-1}}(\cU) \cup \mathcal{L}_{x_0, \ldots, x_{n-1}}(M) \right)} = \omega'|_{\mathfrak{B} \left(\mathcal{L}^{n-1}_{x_0, \ldots, x_{n-1}}(\cU) \cup \mathcal{L}_{x_0, \ldots, x_{n-1}}(M) \right)}$. 
\end{remark}

\begin{lemma}\label{smooth extensions of products are invariant}
    If a global measure $\omega(x_0,\dots,x_{n-1})$ is $n$-smooth over  a small set $A$ and $\omega|_{(x_i)_{i\neq j}}$ is $A$-invariant for each $j<n$, then $\omega$ is $A$-invariant.
\end{lemma}

\begin{proof}

Let $\sigma\in \Aut(\Mbb/A)$ be arbitrary. Then, by assumption, $\sigma(\omega)$ extends  $\sigma(\omega|_{(x_i)_{i\neq j}}) = \omega|_{(x_i)_{i\neq j}}$ for all $j < n$ and $\sigma(\omega|_{A}) = \omega|_{A}$, hence $\sigma(\omega) = \omega$ by $n$-smoothness of $\omega$ over $A$.
\end{proof}

The following is a key proposition generalizing \cite[Proposition 2.27]{simon2013distal} in the case $n=1$.

\begin{proposition}\label{2-determinacy with a type}
    Assume $T$ is NIP and strongly $n$-distal, and $\mu_0(x_0),\ldots,\mu_{n-1}(x_{n-1})$ are global measures generically stable over a small model $\M$. Then $\mu_0\otimes\dots\otimes\mu_{n-1}$ is $n$-smooth over $\M$.
\end{proposition}

\begin{proof}
   Let $\xb: = (x_0,\dots,x_{n-1})$. Fix $a$ an arbitrary tuple in $\Mbb$. Let  $\omega'(\xb)\in \Mfrak(\M a)$ be arbitrary with $\omega'\supseteq (\mu_0\otimes \dots \otimes\mu_{n-1})|_\M$ and compatible with $(\bigotimes_{i \neq j}\mu_i(x_i) : j < n)$. Since $T$ has NIP, by Lemma \ref{existence of smooth extensions of products}, there is a global measure $\omega(\bar{x}) \in \Mfrak(\Mbb)$ which is  $n$-smooth over some small set $B\supseteq \M a$ and so that  $\omega|_{\M a} = \omega'$ and $\omega|_{(x_i : i \neq j)} = \bigotimes_{i \neq j}\mu_i(x_i)$ for all $j < n$. We will show that $\omega' = (\mu_0\otimes \dots \otimes\mu_{n-1})|_{\M a}$, which is sufficient.

     We choose a particular Morley sequence in $\tp(B/\M)$ as follows. Let $p$  be a global coheir of $\tp(B/\M)$. Let $B_0 := B$ and fix a small $|B_0|^+$-saturated, $|B_0|^+$-homogeneous model $\M_0 \supseteq B_0$. For each $i  \in \mathbb{N}$, by induction on $i$, choose a $|\M_{i-1} \cup B_i|^+$-saturated and $|\M_{i-1} \cup B_i|^+$-homogeneous small model $\M_i \supseteq \M_{i-1} \cup B_i$. Then, choose $B_{i+1}$ to be a realization of $p|_{\M_i}$
    
    For each $i$, there is a canonical $\tau_i\in \Aut(\Mbb/\M)$ satisfying $\tau_i(B)=B_i$ according to the enumeration of the variables of $p$ (in particular, $\tau_0|_{B_0} =\id|_{B_0}$). Hence, we can define a global measure $\omega_i(\bar{x}_i)$ via $\omega_i(\varphi(\xb_i;\bb)):=\omega(\varphi(\xb;\tau_i^{-1}(\bb)))$  for all $\varphi \in \mathcal{L}$ and $\bar{b}$ tuple in $\cU$ (identifying $\xb = \xb_0$ and $\omega = \omega_0$ from now on).

   For all $i \in \mathbb{N}$, we have $\omega_i|_{(x_k : k \neq j)} = \bigotimes_{k \neq j} \mu_k(x_k)$ for all $j < n$ (since each $\mu_k$ is $\M$-invariant, hence also $\bigotimes_{k \neq j} \mu_k(x_k)$ is $M$-invariant for all $j<n$, see Fact \ref{basic tensor facts})

    Furthermore, $\omega_i$ is $n$-smooth over $B_i$ for all $i \in \mathbb{N}$. Indeed, let $\omega_i''\in \Mfrak_{\bar{x}}(\Mbb)$  be arbitrary such that $\omega''_i|_{B_i} = \omega_i|_{B_i}$ and $\omega''_i|_{(x_k : k \neq j)} = \omega_i|_{(x_k : k \neq j)} = \bigotimes_{k \neq j} \mu_k(x_k)$ for all $j < n$.  Then $\tau_i^{-1}(\omega''_i)$  satisfies $\tau_i^{-1}(\omega''_i)|_{B_0} = \omega_0|_{B_0}$ and $\tau_i^{-1}(\omega''_i)|_{(x_k : k \neq j)} = \omega_0|_{(x_k : k \neq j)}$ for all $j < n$, so $\tau_i^{-1}(\omega''_i) = \omega_0$ as $\omega_0$ is $n$-smooth over $B_0$. So $\omega''_i = \tau_i(\omega_0) = \omega_i$. 

    Now, for each $i \in \mathbb{N}$, as  $\bigotimes_{k \neq j} \mu_k(x_k)$ is  $\M$-invariant,  hence $B_i$-invariant ($\M \subseteq B_i$), and $\omega_i$ is $n$-smooth over $B_i$, Lemma \ref{smooth extensions of products are invariant} implies that $\omega_i$ is $B_i$-invariant. Hence, we can form a global measure $\lambda((\xb_i)_{i\in \omega}):=\bigotimes\limits_{i\in \omega} \omega_i(\xb_i)$ (Remark \ref{rem: tens prod inv meas NIP}).

  \noindent  \textbf{Claim 1.} $\lambda((\xb_i)_{i\in \omega})$ is totally indiscernible over $\M$.

   \noindent \textbf{Proof of Claim 1.} 
%    We aim to show, by induction on the exponent, that $\lambda|_\M =(\mu_0\otimes \dots \otimes\mu_{n-1})^\omega|_\M$. 
    We show that for each $i$,  $(\bigotimes\limits_{j<i} (\omega_j(\bar{x}_j))|_\M = (\bigotimes\limits_{j<i} (\mu_0\otimes \dots \otimes\mu_{n-1} (\bar{x}_j)))|_\M$. This is sufficient: each $\mu_i$ is generically stable over $\M$, so is each product $\mu_0 \otimes \dots \otimes \mu_{n-1}$; since generically stable measures in NIP theories commute (see Fact \ref{fact: gen stable measures}), this implies $\lambda$ is totally indiscernible over $\M$.

    The base case $i=1$ follows from the fact that
    $\omega_0|_{\M} = \omega'|_\M = \mu_0\otimes\dots\otimes \mu_{n-1}|_{\M}$ by assumption. So, suppose inductively that $(\bigotimes\limits_{j<i} \omega_j)|_\M = (\bigotimes\limits_{j<i} (\mu_0\otimes \dots \otimes\mu_{n-1}))|_\M$, and we prove the claim for $i+1$.

 \noindent   \textit{Subclaim 1.} $\omega_i|_{\M_{i-1}} = (\mu_0\otimes \dots \otimes\mu_{n-1})(\xb_i)|_{\M_{i-1}}$
    
 \noindent    \textit{Proof of Subclaim 1.} As $\omega_i$ is $B_i$-invariant, by Fact \ref{invariant and non-forking measures}, $\omega_i$ does not fork over $B_i$. Suppose that $\varphi(\xb_i;\M_{i-1})$ has positive $\omega_i$ measure, so $\varphi(\xb_i;\M_{i-1})$ does not fork over $B_i$. Choose $a_{\xb}$ from $\Mbb$ such that $\models \varphi(a_{\xb};\M_{i-1})$ and $\tp(a_{\xb}/\M_{i-1})$ does not fork over $B_i$. Now, we also know that $\tp(B_i/\M_{i-1})$ does not fork over $\M$ (by finite satisfiability, Fact \ref{fac: forking in NIP}). By left-transitivity of forking (Fact \ref{fac: forking in general T}), we conclude that $\tp(a_{\xb}/\M_{i-1})$ does not fork over $\M$. In particular, $\varphi(\xb_i;\M_{i-1})$ does not fork over $\M$. Since $\varphi$ was arbitrary, we see that $\omega_i|_{\M_{i-1}}$ does not fork over $\M$. Then, again by Fact \ref{invariant and non-forking measures} and $|\M|^+$-saturation of $\M_{i-1}$, $\omega_i|_{\M_{i-1}}$ is $\Aut(\M_{i-1}/\M)$-invariant. Also, $(\mu_0\otimes \dots \otimes\mu_{n-1})(\xb_i)$ is $\M$-invariant and generically stable, and $\omega_i|_{\M} = (\mu_0\otimes \dots \otimes\mu_{n-1})|_{\M}$. So, by Fact \ref{fact: gen stable measures}, $\omega_i|_{\M_{i-1}} = (\mu_0\otimes \dots \otimes\mu_{n-1})(\xb_i)|_{\M_{i-1}}$.

\medskip

  \noindent  \textit{Subclaim 2.} $((\bigotimes\limits_{j<i}\omega_j)\otimes \omega_i)|_{B_{<i}} = ((\bigotimes\limits_{j<i}\omega_j)\otimes (\mu_0\otimes \dots \otimes \mu_{n-1})_{\bar{x}_i})|_{B_{<i}}$, where $B_{<i} = \bigcup_{j<i} B_j$. 
    
 \noindent   \textit{Proof of Subclaim 2.} Now, $\bigotimes\limits_{j<i}\omega_j$ is $B_{<i}$-invariant (as $\omega_j$ is $B_j$-invariant) and by Subclaim 1 we have that $\omega_i|_{\M_{i-1}} = (\mu_0\otimes \dots \otimes \mu_{n-1})|_{\M_{i-1}}$. So for any formula $\varphi(\bar{x}_{\leq i}, B_{<i})$ we have (note that $B_{<i} \subseteq \M_{i-1}$):
    \begin{gather*}((\bigotimes\limits_{j<i}\omega_j)\otimes \omega_i)(\varphi(\xb_{<i},\xb_i;B_{<i})) = \int_{q\in S_{\xb_i}(\M_{i-1})} (\bigotimes\limits_{j<i}\omega_j)(\varphi(\xb_{<i};q,B_{<i}))d\omega_i\\ = \int_{q\in S_{\xb_i}(\M_{i-1})} (\bigotimes\limits_{j<i}\omega_j)(\varphi(\xb_{<i};q,B_{<i}))d((\mu_0\otimes \dots \otimes \mu_{n-1})(\xb_i))\\ = ((\bigotimes\limits_{j<i}\omega_j)\otimes (\mu_0\otimes \dots \otimes \mu_{n-1})(\varphi(\xb_{<i},\xb_i;B_{<i})).
    \end{gather*}

\medskip

  \noindent  \textit{Subclaim 3.} $(\bigotimes\limits_{j<i}\omega_j) \otimes (\mu_0\otimes \dots \otimes \mu_{n-1}) = (\mu_0\otimes \dots \otimes \mu_{n-1}) \otimes (\bigotimes\limits_{j<i}\omega_j)$.

\noindent    \textit{Proof of Subclaim 3.} Since $(\mu_0\otimes \dots \otimes \mu_{n-1})$ is generically stable (Fact \ref{basic tensor facts}) and each $\omega_j$ is invariant, $(\mu_0\otimes \dots \otimes \mu_{n-1})$ commutes with each $\omega_j$ (see Fact \ref{fact: gen stable measures}). The subclaim follows by associativity of the tensor product.

\medskip

 \noindent   \textit{Subclaim 4.} $(\mu_0\otimes \dots \otimes \mu_{n-1})\otimes (\bigotimes\limits_{j<i}\omega_j))|_\M = ((\mu_0\otimes \dots \otimes \mu_{n-1})\otimes (\bigotimes\limits _{j<i}(\mu_0\otimes \dots \otimes \mu_{n-1})))|_\M$. 

  \noindent  \textit{Proof of Subclaim 4.} We have $(\bigotimes\limits_{j<i}\omega_j)|_\M=(\bigotimes\limits _{j<i}(\mu_0\otimes \dots \otimes \mu_{n-1}))|_\M$ by the inductive hypothesis, and $\mu_0\otimes \dots \otimes \mu_{n-1}$ is $\M$-invariant, so the subclaim follows by the same argument as Subclaim 2. 

\medskip

    Applying Subclaims 2--4 consecutively (as $\M\sub B_{<i}$), we conclude that $(\bigotimes\limits_{j\leq i}\omega_j)|_\M = (\bigotimes\limits _{j\leq i}(\mu_0\otimes \dots \otimes \mu_{n-1}))|_\M$.  By induction, Claim 1 follows. 
    \hfill$\blacksquare_{\textit{Claim }1}$

    Now, define a new measure $\eta((\xb_i,y_i)_{i\in \omega})\in \Mfrak(\M)$, where $y_i$ is a variable of the same sort as $B_i$, via $\eta(\varphi((\xb_i,y_i)_{i\in \omega})) := \lambda(\varphi((\xb_i,B_i)_{i\in \omega}))$ for all $\varphi \in \mathcal{L}(\M)$ (see Lemma \ref{existence of funny measure}).

\medskip

   \noindent \textbf{Claim 2.} $\eta((\xb_i,y_i)_{i\in \omega})$ is an indiscernible measure (over $\M$).

  \noindent  \textbf{Proof of Claim 2.}
    Choose  any $l$, $k_0<\dots<k_i<\dots < k_l$ and $k_i<j<k_{i+1}$ in $\mathbb{N}$. For convenience, we denote $k_i =: i$. We also index tuples of parameters, sets or variables as follows: $(\alpha_{k_n})_{n<i} =: \alpha_-$ and $(\alpha_{k_n})_{i < n \leq l} =: \alpha_+$. We also let $\omega_- := \bigotimes_{n < i} \omega_{k_n}, \omega_{+} := \bigotimes_{i < n \leq l} \omega_{k_n}$. The cases that $i=0$ or $i=l$ are obvious modifications of the following. Fix $\varphi(\xb_-,\xb_i,\xb_+,B_-,B_i,B_+)\in \Lcal(B_-B_iB_+)$ with $\varphi(\xb_-,\xb_i,\xb_+,y_-,y_i,y_+)\in \Lcal(\M)$. It suffices to show that: \[
    (\omega_-\otimes \omega_i \otimes \omega_+)(\varphi(\xb_-,\xb_i,\xb_+,B_-,B_i,B_+)) = (\omega_-\otimes \omega_j \otimes \omega_+)(\varphi(\xb_-,\xb_j,\xb_+,B_-,B_j,B_+)).
    \]
    First, choose $N\in \Nbb$ bigger than the index of any of the parameters. Recall that we chose each $\M_n$ to be $|\M_{n-1}|^+$-saturated and $|\M_{n-1}|^+$-homogeneous, and $B_{-} B_i B_{+} \subseteq \M_N$. So, fix $\sigma \in \Aut(\M_N/\M)$ with $\sigma|_{B_-,B_+} = \id|_{B_-,B_+}$ and $\sigma(B_i)= B_j$ canonically.
  Now, by definition of the tensor product, 
  \begin{gather*}
  	(\omega_-\otimes \omega_i \otimes \omega_+)(\varphi(\xb_-,\xb_i,\xb_+,B_-,B_i,B_+)) \\
  	= \int_{q\in S_{\xb_+}(\M_N)} (\omega_-\otimes \omega_i)(\varphi(\xb_-,\xb_i,q,B_-,B_i,B_+))d\omega_+|_{\M_N}.
  \end{gather*}

    We note here that the extension of each $\omega_i$ to a regular Borel measure is the pushforward of the Borel extension of $\omega_0$ by the continuous map on the space of types induced by $\tau_i$ (see Remark \ref{rem: pushforw under aut}). We use variants of this observation together with the change of variable formula for pushforwards to obtain the result. (Specifically, the change of variables formula applies to the pushforward of the corresponding Borel measures by the continuous map on the spaces of types induces by the automorphism.) 
    
    Fix $q\in S_{\xb_+}(\M_N)$. Note that $\sigma(q) \in S_{\xb_+}(\M_N)$, and $q$ is realized by some $c$ in $\M_{N+1}$ (since $\M_{N+1}$ is $|\M_N|^+$-saturated) and $\sigma$ extends to an automorphism of $\M_{N+1}$ ($\M_{N+1}$ is $|\M_N|^+$-homogeneous), which then extends to an automorphism of $\cU$. We also denote this extended automorphism by $\sigma$. Then we have \begin{align*}
        (\omega_-\otimes \omega_i)(\varphi(\xb_-,\xb_i,q,B_-,B_i,B_+))
        &= \int_{r\in S_{\xb_i}(\M_{N+1})} \omega_-(\varphi(\xb_-,r,c,B_-,B_i,B_+))d\omega_i \\
        &= \int_{r\in S_{\xb_i}(\M_{N+1})} \omega_-(\varphi(\xb_-,\sigma(r),\sigma(c),B_-,B_j,B_+))d\omega_i\\
        &= \int_{r\in S_{\xb_j}(\M_{N+1})} \omega_-(\varphi(\xb_-,r,\sigma(c),B_-,B_j,B_+))d\omega_j\\
        &= (\omega_-\otimes \omega_j)(\varphi(\xb_-,\xb_j,\sigma(q),B_-,B_j,B_+)).
    \end{align*}

\noindent In the above, the first equality is by definition. The second equality follows from $B_-$-invariance of $\omega_-$ and $\sigma(B_{-} B_i B_{+})= B_{-} B_j B_{+}$. The third equality follows from the following observation. For any $\psi(\xb_i,b) \in \Lcal(\Mbb)$, define a new measure via  $\omega_i'(\psi(\xb_i,b)) := \omega_j(\psi(\xb_j,\sigma(b)))$. Then $\omega_i'|_{B_i}=\omega_i|_{B_i}$ since  $\sigma|_{B_i} = (\tau_j \circ \tau_i^{-1})|_{B_i}$. And $\omega'_i|_{(x_k : k \neq j)} = \sigma^{-1} \left( \bigotimes_{k \neq j} \mu_{k}(x_k) \right) = \bigotimes_{k \neq j} \mu_{k}(x_k)$ for all $k<n$,  by $\M$-invariance of $\bigotimes_{k \neq j} \mu_{k}$. So, $\omega_i'=\omega_i$ because $\omega_i$ is $n$-smooth over $B_i$, hence $\omega_i(\psi(\xb_i,b)) = \omega_j(\psi(\xb_j,\sigma(b)))$. Since the measure on Borel sets is determined by its value on clopens, this shows that the Borel extension of $\omega_j$ is the pushforward of the Borel extension of $\omega_i$ by the continuous map $\sigma$ induces on types, $\widetilde{\omega_j|_{\M_{N+1}}} = \sigma'_{\ast}(\widetilde{\omega_i|_{\M_{N+1}}})$ (see Remark \ref{rem: pushforw under aut} and the notation there). We can then apply the change of variable formula for pushforward measures to obtain the third equality. The final equality is by definition of the tensor product again.

We also have that, \begin{align*}
    & \int_{q\in S_{\xb_+}(\M_N)} (\omega_-\otimes \omega_j)(\varphi(\xb_-,\xb_j,\sigma(q),B_-,B_j,B_+))d\omega_+ \\&= \int_{q\in S_{\xb_+}(\M_N)} (\omega_-\otimes \omega_j)(\varphi(\xb_-,\xb_j,q,B_-,B_j,B_+))d\sigma_*(\omega_+)
    \\ &= \int_{q\in S_{\xb_+}(\M_N)} (\omega_-\otimes \omega_j)(\varphi(\xb_-,\xb_j,q,B_-,B_j,B_+))d\omega_+.
\end{align*}

\noindent Here, the first equality is the change of variable formula for pushforward measures (again, applied with respect to the Borel extensions and the continuous map on types over $\mathcal{M}_N$ induced by the automorphism). The second equality is given by the following. Since each $\omega_k$ is $B_k$-invariant, $\omega_+$ is $B_+$-invariant. So, since $\sigma$ fixes $B_+$, $\omega_+$ is invariant under the action of $\sigma_*$. Putting these together, we find that 
\begin{gather*}      
	(\omega_-\otimes \omega_i \otimes \omega_+)(\varphi(\xb_-,\xb_i,\xb_+,B_-,B_i,B_+)) \\
	= (\omega_-\otimes \omega_j \otimes \omega_+)(\varphi(\xb_-,\xb_j,\xb_+,B_-,B_j,B_+)).
\end{gather*}
    \hfill$\blacksquare_{\textit{Claim }2}$
    
    Now, $\lambda$ is totally indiscernible over $\M$ and $\eta$ is indiscernible over $\M$ (by Claims 1 and 2). Define $\xb_i^j$ to be the $j$th coordinate of $\xb_i$. Then for each $k<n$, $\lambda|_{(\xb_i^j)_{j\neq k, i\in \omega}} = \bigotimes\limits_{i\in \omega}(\bigotimes\limits_{j\neq k} \mu_j(\xb_i^j)))$. This measure is $\M$-invariant and $(B_i)_{i\in \omega}$ is $\M$-indiscernible, so the sequence $(B_i)_{i\in \omega}$ is $\lambda|_{(\xb_i^j)_{j\neq k, i\in \omega}}$-indiscernible for each $k<n$. Hence, by strong $n$-distality and Lemma  \ref{key lemma with strong indiscernibility}, we have that $(B_i)_{i \in \mathbb{N}}$ is $\lambda$-indiscernible. In particular  $\omega_i(\varphi(\xb_i,B_0)) = \omega_i(\varphi(\xb_i,B_i)) = \omega_0(\varphi(\xb_0,B_0))$, for any $\varphi \in \mathcal{L}$. Hence, recalling that $\omega_i|_{(x_k : k \neq j)} = \bigotimes_{k \neq j} \mu_k(x_k)$ for all $j < n$, by $n$-smoothness of $\omega_0$ over $B_0$ we get  $\omega_i=\omega_0$. But, we have already seen in Subclaim 1 that $\omega_1|_{B_0}=(\mu_0\otimes \dots \otimes \mu_{n-1})|_{B_0}$, so $\omega_0|_{B_0}= (\mu_0\otimes \dots \otimes \mu_{n-1})|_{B_0}$. Hence, $\omega' = \omega_0|_{\M a} = (\mu_0 \otimes \dots \otimes \mu_{n-1})|_{\M a}$, as required.
\end{proof}

\begin{problem}\label{prob: str form of gs implies n-smooth}
	In Proposition \ref{2-determinacy with a type}, can we conclude that $\mu_0 \otimes \ldots \otimes \mu_{n-1}$ is \emph{strongly} $n$-smooth over $\M$ (see Remark \ref{rem: weak n-smooth})? We will say that $T$ \emph{satisfies strong form of Proposition \ref{2-determinacy with a type}} if this holds. 
\end{problem}
\begin{remark}
We note that at least this stronger conclusion holds for types in stable strongly $n$-distal theories. Namely, let $p_0(x_0), \ldots, p_{n-1}(x_{n-1})$ be global types generically stable over $\M$, $b$ is any tuple in $\cU$, and   assume that $(a_0, \ldots, a_{n-1})$ in $\cU$ are arbitrary so that $a_i \models p_i|_{\M b}$ and $(a_0, \ldots, a_{n-1}) \models (p_0 \otimes \ldots \otimes p_{n-1})|_{\M}$. In particular $a_i \ind_{\M} b$ for all $i$. By total $(n-1)$-triviality (using Proposition \ref{prop: total triv in stable}) this implies $a_0 \ldots a_{n-1} \ind_{\M} b$, hence $(a_0, \ldots, a_{n-1}) \models (p_0 \otimes \ldots \otimes p_{n-1})|_{\M b}$ by stationarity of $(p_0 \otimes \ldots \otimes p_{n-1})|_{\M}$.
\end{remark}

\subsection{N-distal packing lemma, $n$-determinacy for measures and hypergraph regularity lemma} \label{sec: N-distal packing lemma, n-determinacy for measures and hypergraph regularity lemma}

\begin{fact}\cite[Proposition 7.6, Theorem 7.7]{walker2023distality} \label{fac: det for types}
	Assume $T$ is $k$-distal, for $1 \leq k < \omega$, and $p_i \in S_{x_i}(\cU)$ are global invariant types that are pairwise commuting, for $i < k+1$. Then the $(k+1)$-tuple $(p_0, \ldots, p_k)$ is \emph{$k$-determined}, i.e.~$\bigcup_{u \subseteq k+1, |u|=k} \bigotimes_{i \in u} p_i(x_i) \vdash p_0(x_0) \otimes \ldots \otimes p_k(x_k)$. And if $T$ is NIP, this actually characterizes $k$-distality.
\end{fact}

Here we are interested in an analog of this for Keisler measures, and connections to hypergraph regularity. The following definition generalizes the corresponding notion for types:
\begin{definition}\label{def: n-det for tuples of measures}
    Given $1 \leq k<n$, an $n$-tuple of global invariant (Keisler) measures $(\mu_0(x_0),\dots,\mu_{n-1}(x_{n-1}))$ is \emph{$k$-determined} if $\bigotimes\limits_{i<n} \mu_i(x_i)$ is $\Lcal^k_{x_0,\dots,x_{n-1}}(\Mbb)$-determined (see Definitions \ref{def: A-determined measure} and \ref{def: cylinder Bool alg}).
\end{definition}

\begin{remark}
    A pair of global measures is $1$-determined if and only if they are orthogonal in the sense of \cite[Corollary 2.5]{hrushovski2013generically}.
\end{remark}

\begin{conjec}\label{prop: m-det}
	Suppose $T$ is $m$-distal and $n>m>0$. If $\mu_0(x_0), \ldots, \mu_{n-1}(x_{n-1})$ are generically stable, then the product $\mu:=\mu_1 \otimes \ldots \otimes\mu_{n-1}$ is $m$-determined.
\end{conjec}

Proposition \ref{2-determinacy with a type} should be sufficient to 
derive $n$-determinacy for $\otimes$-products of generically stable measures (and it is a standard way to do this in the case $n=1$, see \cite[Corollary 2.5]{hrushovski2013generically}). Here we consider a higher arity version.

The following is an easy consequence of definability of measures (see e.g.~ \cite[Fact 2.3]{chernikov2022definable}):
\begin{lemma}\label{gs and formula defining measure}
Suppose $T$ has NIP and $\mu(x)\in\Mfrak(\Mbb)$ is definable over a small model $\M$. Then for any formulas $\gamma_0(x,z), \gamma_1(x,z)\in\Lcal(\M)$ and $\varepsilon>0$, there is a formula $\psi(z)\in \Lcal(\M)$ such that $(1)\implies (2) \implies (3)$ for all $c\in \Mbb^z$, where:
\begin{enumerate}
    \item  $|\mu(\gamma_0(x,c))-\mu(\gamma_1(x,c))|<\varepsilon$;
    \item $\models \psi(c)$;
    \item $|\mu(\gamma_0(x,c))-\mu(\gamma_1(x,c))|<2\varepsilon$.
\end{enumerate}
\end{lemma}

\begin{prop}\label{prop: dist packing lemma}
	Let $T$ be NIP and strongly $n$-distal. For every $\varphi(x_0, \ldots,x_{n-1};x_{n})$ and $\varepsilon>0$ there exist some $k \in \mathbb{N}$ and formulas  $\gamma_{i}(x_0, \ldots, x_{n-1}; z_0) \in \Lcal$, $\rho_i(x_{n}; z_0) \in \Lcal$ and $\chi_i^-(x_0,\dots, x_{n-1};z), \chi_i^+(x_0,\dots, x_{n-1};z) \in \Lcal^{n}_{x_0, \ldots, x_{n-1}, z}$ for $i < k$, where $z = z_0 z_1$ and each $\chi_i^-, \chi_i^+$ is given by a Boolean combination of $\gamma_{j}(x_0, \ldots, x_{n-1}; z_0), j <k$ and some formulas of the form $ \delta_{u}((x_t : t \in u ); z_1) \in \mathcal{L}$ for $u \subsetneq \{0, \ldots, n-1\}$, satisfying the following.   Given a tuple of global generically stable measures $\bar{\mu} = (\mu_0(x_0), \ldots, \mu_{n-1}(x_{n-1}))$ there exists some $e_{\bar{\mu}} \in \cU^{z_0}$ so that: 
	$(\rho_i(\cU;e_{\bar{\mu}}) : i < k)$ partition $\cU^{x_n}$ and for every  $b \in \rho_i(\cU^{x_{n}})$ there is some $d_b \in \cU^{z_1}$ with 
     \begin{enumerate}
         \item $\chi_i^-(x_0,\dots, x_{n-1}; e_{\bar{\mu}}, d_b)\rightarrow \varphi(x_0,\dots, x_{n-1},b) \rightarrow \chi_i^+(x_0,\dots, x_{n-1}; e_{\bar{\mu}}, d_b)$,
        \item $\mu_0\otimes \dots \otimes\mu_{n-1}(\chi_i^+(x_0,\dots, x_{n-1}; e_{\bar{\mu}}, d_b)) - \mu_0\otimes \dots \otimes\mu_{n-1}(\chi_i^-(x_0,\dots, x_{n-1}; e_{\bar{\mu}}, d_b))< \varepsilon$.
     \end{enumerate}
\end{prop}

\begin{remark}\label{rem: n-dependent packing lemma}
	We will refer to Proposition \ref{prop: dist packing lemma} as the \emph{$n$-distal packing lemma}, since it is a stronger form of the $n$-dependent packing lemma (i.e.~the packing lemma for families of sets of finite $\VC_n$-dimension, generalizing the classical Haussler's packing lemma in the case $n=1$) established in \cite{chernikov2020hypergraph}:
	\begin{fact}(No assumption on $T$.) 
		For every $n$-dependent formula $\varphi(x_0, \ldots,x_{n-1};x_{n})$ and $\varepsilon>0$ there exists some $k \in \mathbb{N}$ and formulas $\gamma_{i}(x_0, \ldots, x_{n-1}; z_0) \in \Lcal$ (each given by an  instance of the form $\varphi(x_0, \ldots, x_{n-1}; z_{0,j})$ where $z_0 = (z_{0,0}, \ldots, z_{0,k-1})$), $\rho_i(x_{n}; z_0) \in \Lcal$ and $\chi_i(x_0,\dots, x_{n-1};z) \in \Lcal^{n}_{x_0, \ldots, x_{n-1}, z}$ for $i < k$, where $z = z_0 z_1$ and each $\chi_i$ is given by a Boolean combination of $\gamma_{j}(x_0, \ldots, x_{n-1}; z_0)$ and some formulas 	 $ \delta_{u}((x_t : t \in u ); z_1) \in \mathcal{L}$ for $u \subsetneq \{0, \ldots, n-1\}$ (each given by replacing some, and at least one, of $x_i$ by $z_{1,j}$ in $\varphi(x_0, \ldots, x_{n-1})$, where $z_1 = (z_{1,j})$) satisfying the following.  Given a tuple of global generically stable measures $\bar{\mu} = (\mu_0(x_0), \ldots, \mu_{n-1}(x_{n-1}))$ (or just definable measures so that $\mu_i$ commutes with $\mu_j$ for all $0 \leq i,j <n$) there exists some $e_{\bar{\mu}} \in \cU^{z_0}$ so that: $(\rho_i(\cU;e_{\bar{\mu}}) : i < k)$ partition $\cU^{x_n}$ and for every  $b \in \rho_i(\cU^{x_{n}})$ there is some $d_b \in \cU^{z_1}$ with 
     \begin{itemize}
         \item $\mu_0\otimes \dots \otimes\mu_{n-1} \Big( \varphi(x_0,\dots, x_{n-1},b) \  \triangle \  \chi(x_0,\dots, x_{n-1}; e_{\bar{\mu}}, d_b) \Big) < \varepsilon$.
     \end{itemize}
	\end{fact}
\end{remark}

\begin{proof}[Proof of Proposition \ref{prop: dist packing lemma}]
	First we prove a non-uniform version of the statement where the formulas in the conclusion are allowed to depend on the measures. So fix some $\mu_0, \ldots, \mu_{n-1}$ generically stable, and let $\M$ be a small model so that all of $\mu_i$ are generically stable over it. 
    Fix $\varphi(x_0,\dots,x_{n-1},x_n)\in \Lcal(\Mbb)$ and $\varepsilon>0$.

    By Proposition \ref{2-determinacy with a type}, $\mu_0(x_0) \otimes \ldots \otimes \mu_{n-1}(x_{n-1})$ is $n$-smooth over $\M$, so in particular weakly $n$-smooth over $\M$ (Remark \ref{rem: weak n-smooth}). Then, using   Fact \ref{existence of measure extensions}, we have that for any $b \in \cU^{x_n}$ there are some 
    $\chi_b^-(x_0,\dots, x_{n-1},z), \chi_b^+(x_0,\dots, x_{n-1},z) \in   \Lcal^{n}_{x_0,\dots,x_{n-1},z}$, $e_b \in \M^{z_0}$ and $c_b\in \Mbb^{z_1}$ with so both $\chi_b^-(x_0,\dots, x_{n-1}; e_b, c_b)$, $\chi_b^+(x_0,\dots, x_{n-1}; e_b, c_b)$ are  Boolean combinations of some $\gamma_{b,i}(x_0, \ldots, x_{n-1}; e_b) \in \Lcal(\M)$ and some $ \delta_{u,b}((x_t : t \in u ); c_b)$,     
    so that:
    \begin{enumerate}
        \item $\chi_b^-(x_0,\dots, x_{n-1};e_b, c_b)\rightarrow \varphi(x_0,\dots, x_{n-1},b) \rightarrow \chi_b^+(x_0,\dots, x_{n-1};e_b, c_b)$;
        \item $\mu_0\otimes \dots \otimes\mu_{n-1}(\chi_b^+(x_0,\dots, x_{n-1};e_b,c_b))- \mu_0\otimes \dots \otimes\mu_{n-1}(\chi_b^-(x_0,\dots, x_{n-1};e_b,c_b))<\varepsilon$.
    \end{enumerate}

     \noindent By  $\M$-invariance of $\mu_0\otimes \dots \otimes \mu_{n-1}$, if $\chi_b^-,\chi_b^+, e_b,c_b$ satisfy the above conditions 1,2 for some $b\models q\in S_{x_n}(\M)$ in $\cU$, then for every $b'\models q$ there is some $c_{b'}$ such that $\chi_b^+,\chi_b^-, e_b, c_{b'}$ satisfy the conditions 1,2. So, for each $q\in S_{x_n}(\M)$ we choose some $b\models  q$ and denote $\chi_q^+:=\chi_b^+$, $\chi_q^-:=\chi_b^-$ and $e_q := e_b$.

     Now, for each $\chi_q^+(x, z), \chi_q^-(x,z)$ we can choose a formula $\psi_q(z)\in \Lcal(\M)$ as in Lemma \ref{gs and formula defining measure} with respect to $\mu_0 \otimes \ldots \otimes \mu_{n-1}$, and  let $\rho_q(x_n) \in \Lcal(\M)$ be the formula 
     \begin{gather*}
     	\exists z_1( \forall x_0,\dots,x_{n-1}(\chi^-_q(x_0,\dots,x_{n-1}; e_q, z_1) \rightarrow \varphi(x_0,\dots,x_n)\rightarrow \\
     	\chi_q^+(x_0,\dots,x_{n-1}, e_q, z_1)) \wedge \psi_q(e_q,z_1)).
     \end{gather*}
      By the previous paragraph and $(1)\Rightarrow(2)$ in Lemma \ref{gs and formula defining measure}, the clopen sets $([\rho_q] : q \in S_{x_n}(\M))$ cover the compact space $S_{x_n}(\M)$. So we can choose $q_0,\dots,q_{k-1}\in S_{x_n}(\M)$ such that $([\rho_{q_i}])_{i<k}$ still covers $S_{x_n}(\M)$. For $i<k$, define $\rho_i(x_n) := \rho_{q_i}(x_n)\wedge (\wedge_{j<i}\neg \rho_{q_j}(x_n))$. Then $(\rho_i(x_n))_{i<k}$ partitions $\cU^{x_n}$. In an abuse of notation, for each $i<k$ we denote $\chi^-_{i}:= \chi^-_{q_i}$, $\chi^+_{i}:= \chi^+_{q_i}$, $\psi_i(z):= \psi_{q_i}(z)$ and $e_i := e_{q_i}$.

     So, by $(2) \Rightarrow (3)$ from Lemma \ref{gs and formula defining measure}, for each $b \in \cU^{x_n}$ there is a unique $i<k$ (such that $\models \rho_i(b)$) and some $d_b \in \cU^z$ satisfying the following:
     \begin{enumerate}
         \item $\chi_i^-(x_0,\dots, x_{n-1},e_i, d_b)\rightarrow \varphi(x_0,\dots, x_{n-1},b) \rightarrow \chi_i^+(x_0,\dots, x_{n-1},e_i,d_b)$,
        \item $\mu_0\otimes \dots \otimes\mu_{n-1}(\chi_i^+(x_0,\dots, x_{n-1},e_i,d_b)) - \mu_0\otimes \dots \otimes\mu_{n-1}(\chi_i^-(x_0,\dots, x_{n-1},e_i,d_b))<2\varepsilon$.
     \end{enumerate}

Now we show that $k, \psi_i, \rho_i, \chi_i^{-}, \chi_i^{+}$ can be chosen independently of $\bar{\mu}$. This follows by a standard compactness/``ultraproduct of counterexamples'' argument, using that ultraproducts of generically stable measures in NIP theories remain generically stable (thanks to the VC-theorem). We use here the formalism from \cite[Section 3.4]{chernikov2025externally}. Namely, in an NIP theory $T$, given a tuple of variables $x$ we identify the set of global generically stable measures $\mathfrak{M}_{x}^{\gs}(\cU)$ with an $\emptyset$-hyperdefinable set $\widetilde{\mathcal{M}}_{x}$, and let $\mu \in \mathfrak{M}_x^{\gs}(\cU) \mapsto [\mu] \in \widetilde{\mathcal{M}}_x$ denote the bijection.  By \cite[Proposition 3.29]{chernikov2025externally}, the map 
$$([\mu_0], \ldots, [\mu_{n-1}]) \in \widetilde{\mathcal{M}}_{x_0} \times \ldots \times \widetilde{\mathcal{M}}_{x_{n-1}} \mapsto [\mu_0 \otimes \ldots \otimes \mu_{n-1}] \in \widetilde{\mathcal{M}}_{x_0, \ldots, x_{n-1}}$$
is $\emptyset$-type-definable. Using this and \cite[Proposition 3.27]{chernikov2025externally}, for every fixed  $\varphi(x_0, \ldots, x_{n}) \in \Lcal$, $\varepsilon > 0$, $k \in \mathbb{N}$, and $\gamma_i, \rho_i, \chi_i^{+}, \chi_i^{-}$, the set $X_{\varphi, \bar{\gamma}, \bar{\rho}, \bar{\chi}^+, \bar{\chi}^{-}}$ of tuples $([\mu_0], \ldots, [\mu_{n-1}]) \in \widetilde{\mathcal{M}}_{x_0} \times \ldots \times \widetilde{\mathcal{M}}_{x_{n-1}}$ so that $\gamma_i, \rho_i, \chi_i^{+}, \chi_i^{-}$ do not satisfy the conclusion with respect to it is $\emptyset$-type-definable. And, if the conclusion of the proposition fails, the intersection of any finitely many sets of this form is non-empty (as otherwise it would hold by coding finitely many formulas into one). Hence, by saturation of $\cU$, the intersection of all sets of this form is non-empty, say it contains a tuple $([\mu^{\ast}_0], \ldots, [\mu^{\ast}_{n-1}])$. But then the generically stable measures $\mu^{\ast}_0, \ldots, \mu^{\ast}_{n-1}$ fail the non-uniform version of the conclusion proved above.
%The sets
%\begin{gather*}
%	 X_i := \left\{[\mu] \in \widetilde{\mathcal{M}}_y: \forall a \in \M^x \mu(\varphi(a,y)) \leq \frac{1}{i} \right\},\\
%	  Y := \left\{ [\mu] \in \widetilde{\mathcal{M}}_y: \mu^{\otimes d}\left(\exists x\bigwedge_{t \in [d]}\varphi\left(x,y_{t}\right)\right)\geq\alpha \right\}
%\end{gather*}
%\noindent are type-definable  (by \cite[Remark 3.27, Proposition 3.29]{chernikov2025externally}). Note that $X_{i+1} \subseteq X_{i}$ and $[\mu_i]  \in X_i \cap Y$. It follows by saturation of $\M$ that there exists some $\mu \in \mathfrak{M}_x(\M)$ with $[\mu] \in Y \cap \bigcap_{i \in \omega} X_i$. That is, $\mu^{\otimes d}\left(\exists x\bigwedge_{t \in [d]}\varphi\left(x,y_{t}\right)\right)\geq\alpha $ and $\mu(\varphi(a,y)) = 0$ for all $a \in \M^x$. But this contradicts the choice of $d$.
\end{proof}

\begin{remark}\label{rem: stronger form packing lemma}
	If $T$ satisfies strong form of Proposition \ref{2-determinacy with a type} (Problem \ref{prob: str form of gs implies n-smooth}), then in Proposition \ref{prop: dist packing lemma} we could use $d_b = b$. 
\end{remark}

\begin{problem}
	Does Proposition \ref{prop: dist packing lemma} hold in arbitrary (strongly)  $n$-distal theories (without the NIP assumption)?
\end{problem}

As one consequence of the $n$-distal packing lemma, we obtain $n$-determinacy for measures under an additional assumption.

\begin{definition}\label{def: Skolem functions}
    A theory $T$ in a language $\Lcal$ has \emph{definable Skolem functions} if for each formula $\varphi(x,y)$ possibly with parameters there is a formula $\theta(x,y)$ also possibly with parameters satisfying the following:\begin{enumerate}
        \item $\models \forall x,y (\theta(x,y) \rightarrow \varphi(x,y))$
        \item $\models \forall x \exists z \forall y (\theta(x,y) \rightarrow y=z)$
    \end{enumerate}
\end{definition}

\begin{proposition}\label{2-determinacy}
     Assume $T$ is NIP, strongly $n$-distal and either has definable Skolem functions, or satisfies strong form of Proposition \ref{2-determinacy with a type} (Problem \ref{prob: str form of gs implies n-smooth}). Then for any global generically stable measures   $\mu_0(x_0),\dots,\mu_{n-1}(x_{n-1})$ and any global measure $\mu_n(x_n)$, the $(n+1)$-tuple $(\mu_0,\dots,\mu_{n-1},\mu_n)$ is $n$-determined (Definition \ref{def: n-det for tuples of measures}). 
\end{proposition}

\begin{proof}  
Let $\M$ be a small model so that all $\mu_i$ are generically stable over $\M$.    
   Fix $\varphi(x_0, \ldots,x_{n-1};x_{n})$ and $\varepsilon>0$, and let $k \in \mathbb{N}$ and formulas  $\gamma_{i}(x_0, \ldots, x_{n-1}; z_0) \in \Lcal$, $\rho_i(x_{n}; z_0) \in \Lcal$ and $\chi_i^-(x_0,\dots, x_{n-1};z), \chi_i^+(x_0,\dots, x_{n-1};z) \in \Lcal^{n}_{x_0, \ldots, x_{n-1}, z}$ for $i < k$ be as given by Proposition \ref{prop: dist packing lemma}, we will also follow the notation from its proof. If $T$ has definable Skolem functions, for each $i<k$ we may find formulas $\theta_i(x_n,z)\in \Lcal$ as in Definition \ref{def: Skolem functions} for the formula 
   \begin{gather*}
   	\rho'_i(x_n,z): = \\
   	\forall x_0,\dots,x_{n-1}(\chi^-_i(x_0,\dots,x_{n-1}, z) \rightarrow \varphi(x_0,\dots,x_n)\rightarrow \chi_i^+(x_0,\dots,x_{n-1}, z)) \wedge \psi_i(z)
   \end{gather*}
     (that is, $\rho_i(x_n)$ dropping quantification over $z$); or, if $T$ satisfies strong form of Proposition \ref{2-determinacy with a type}, we take  $\theta_i(x_n,z) := (x_n = z)$ (using Remark \ref{rem: stronger form packing lemma}). So, from now on, we denote by $d_b$ the unique element of $\Mbb^z$ satisfying $\models\theta_i(b,d_b)$ where $i$ is such that $\models \rho_i(b)$. We note that $d_b$ satisfies (1),(2) from the proof of Proposition \ref{prop: dist packing lemma}.

     Fix $i<k$. The formula $\chi^-_i(x_0,\dots,x_{n-1},z)$ may be written as a disjunction of a conjunction of formulas using only $n$ of the variables $x_0,\dots,x_{n-1},z$. We will refer to this as the disjunctive normal form of $\chi^-_i$. We obtain a new formula by replacing each of these basic formulas, $\alpha(x_0,\dots,x_{n-1},z) \in \Lcal^n_{x_0,\dots,x_{n-1},z}(\M)$ with a new formula $\tilde\alpha(x_0,\dots,x_{n-1},x_n)\in \Lcal^n_{x_0,\dots,x_{n-1},x_n}(\M)$ as follows.
     \begin{itemize}
         \item If the basic formula $\alpha$ does not mention $z$, let $\tilde\alpha:=\alpha$.
         \item If the basic formula $\alpha$ does mention $z$, it must mention at most $n-1$ of the variables $x_0,\dots,x_{n-1}$ --- without loss of generality suppose $\alpha = \alpha(x_1,\dots,x_{n-1}, z)$. Define $\tilde\alpha:= \exists z (\theta_i(x_n,z)\wedge \alpha(x_1,\dots,x_{n-1},z))$. Note that this is a formula using only $n$ of the variables $x_0,\dots,x_n$.
     \end{itemize}

    \noindent  The formula $\xi^-_i(x_0,\dots,x_n)$ that results from replacing basic formulas in the above manner is in $\Lcal^n_{x_0,\dots,x_n}(\M)$. 
     
    Consider a basic formula $\alpha(x_0,\dots,x_{n-1},z)$ from $\chi^-_i$. Assume $\models \alpha(a_0,\dots,a_{n-1},d_b)$. Since $\models \theta_i(b,d_b)$, in either case we see that $\models \tilde\alpha(a_0,\dots,a_{n-1},b)$. Suppose conversely that $\models \tilde\alpha(a_0,\dots,a_{n-1},b)$. Then, since $d_b$ is the unique element of $\Mbb^z$ satisfying $\theta_i(b,z)$, in fact $\models \alpha(a_0,\dots,a_{n-1},d_b)$. Thus, $\models \forall x_0,\dots,x_{n-1}(\alpha(x_0,\dots,x_{n-1},d_b) \leftrightarrow \tilde \alpha(x_0,\dots,x_{n-1},b))$ for each basic $\alpha$. So, in fact $\models \forall x_0,\dots,x_{n-1}(\chi^-_i(x_0,\dots,x_{n-1},d_b) \leftrightarrow \xi^-_i(x_0,\dots,x_{n-1},b))$.

     We may carry out precisely the same process to construct $\xi^+_i(x_0,\dots,x_n)$, and conclude the following for each $b\models \rho_i(x_n)$ (call these conditions (\textdagger)):
    \begin{enumerate}
         \item $\models \forall x_0,\dots, x_{n-1}(\xi_i^-(x_0,\dots, x_{n-1},b)\rightarrow \varphi(x_0,\dots, x_{n-1},b) \rightarrow \xi_i^+(x_0,\dots, x_{n-1},b))$;
        \item $(\mu_0\otimes \dots \otimes\mu_{n-1})(\xi_i^+(x_0,\dots, x_{n-1},b))- (\mu_0\otimes \dots \otimes\mu_{n-1})(\xi_i^-(x_0,\dots, x_{n-1},b))< \varepsilon$.
     \end{enumerate}

     Now assume $\omega(x_0,\dots, x_n)\in \Mfrak(\Mbb)$ is an arbitrary extension of $(\mu_0\otimes \dots\otimes \mu_n)|_{\Lcal_{x_0,\dots,x_n}^n(\Mbb)}$. From (1) in conditions (\textdagger) we see that: \[
     \omega(\xi_i^-(x_0,\dots, x_n)\wedge \rho_i(x_n)) \leq \omega(\varphi(x_0,\dots, x_n) \wedge \rho_i(x_n)) \leq \omega(\xi^+_i(x_0,\dots, x_n) \wedge \rho_i(x_n)).
     \]

     \noindent The formulas on the left and right hand sides are elements of $\Lcal^n_{x_0,\dots,x_n}(\M)$, hence $\omega$ agrees with $\mu_0\otimes \dots \otimes \mu_n$ on each. So, 
     \begin{gather*}
     	\omega(\xi_i^-(x_0,\dots, x_n)\wedge \rho_i(x_n)) =\\
     	 \int_{q\in S_{x_n}(\M)} (\mu_0\otimes \dots \otimes \mu_{n-1})(\xi_i^-(x_0,\dots,x_{n-1}, q)\wedge \rho_i(q))d\mu_n|_\M\\ =
          \int_{q\in [\rho_i(x_n)]} (\mu_0\otimes \dots \otimes \mu_{n-1})(\xi_i^-(x_0,\dots, x_{n-1}, q))d\mu_n|_\M \textrm{ and}\\
            \omega(\xi_i^+(x_0,\dots, x_n)\wedge \rho_i(x_n)) =\\
             \int_{q\in S_{x_n}(\M)} (\mu_0\otimes \dots \otimes \mu_{n-1})(\xi_i^+(x_0,\dots, x_{n-1}, q)\wedge \rho_i(q))d\mu_n|_\M\\ =
          \int_{q\in [\rho_i(x_n)]} (\mu_0\otimes \dots \otimes \mu_{n-1})(\xi_i^+(x_0,\dots, x_{n-1}, q))d\mu_n|_\M.
     \end{gather*}

  \noindent  Hence, by (2) from conditions (\textdagger), \[\omega(\xi_i^+(x_0,\dots, x_n)\wedge \rho_i(x_n)) - \omega(\xi_i^-(x_0,\dots, x_n)\wedge \rho_i(x_n)) < \mu_n(\rho_i)\varepsilon.\]

  Letting 
  \begin{gather*}
  	\xi^+(x_0,\dots, x_n) := \bigvee_{i<k} \left( \xi_i^+(x_0,\dots, x_n) \land \rho_i(x_n) \right),\\
  	\xi^-(x_0,\dots, x_n) := \bigvee_{i<k} \left( \xi_i^-(x_0,\dots, x_n) \land \rho_i(x_n) \right),
  \end{gather*}
   we have $\xi^-(x_0,\dots, x_n)  \rightarrow \varphi(x_0,\dots, x_{n-1}, x_{n}) \rightarrow \xi^+(x_0,\dots, x_n) $, and since $([\rho_i])_{i<k}$ partitions $S_{x_n}(\M)$, from the above we get 
  \[\omega( \xi^+(x_0,\dots, x_n) ) - \omega(\xi^-(x_0,\dots, x_n)) < \varepsilon.\]

  \noindent As $\varepsilon>0$ and $\varphi$ were arbitrary, this shows that $\omega = \mu_0 \otimes \dots \otimes \mu_n$. 
\end{proof}

\begin{problem}
	 Does Proposition \ref{2-determinacy} hold without the assumption of definable Skolem functions?
\end{problem}

%\begin{cor}
%    Suppose $T$ has NIP, is strongly $n$-distal and has definable Skolem functions. Take global measures $\mu_0(x),\dots ,\mu_{n-1}(x_{n-1})$ generically stable and $\mu_n(x_n)$ arbitrary, and $R(x_0,\dots,x_{n-1},x_n)$ a definable relation. Then for any $\varepsilon>0$ there are definable sets $D_\varepsilon^-, D_\varepsilon^+\in \Lcal_{x_0,\dots,x_{n-1},x_n}^2(\Mbb)$ such that $D_\varepsilon^-\sub R \sub D_\varepsilon^+$ and $(\mu\otimes\dots \otimes \mu_{n-1}\otimes\mu_n)(D^+_\varepsilon\sm D^-_\varepsilon)<\varepsilon$.
%\end{cor}
%
%\begin{proof}
%    Easy consequence of \ref{2-determinacy} and \ref{determination of measures via squashing}.
%\end{proof}

However, we still get a  (not necessarily definable) $n$-distal regularity lemma   generalizing the case $n=1$ from \cite{chernikov2018regularity}:
\begin{cor}\label{cor: str n-dist reg lemma}
	Let $T$ be NIP and strongly $n$-distal. Given $\varphi(x_0, \ldots, x_{n}) \in \Lcal$ and $\varepsilon > 0$ there exists $k \in \mathbb{N}$ satisfying the following. Given any finitely supported measures $\mu_0(x_0), \ldots, \mu_n(x_n)$, for every $0 \leq i \leq n$ there is a  partition $(P^i_{\ell} : \ell < k)$ of $\prod_{j \neq i}\cU^{x_j}$ so that $\sum_{(\ell_0, \ldots, \ell_n)} \mu_0 \otimes \ldots \otimes \mu_n \left( \bigwedge_{i =0}^{n} P^i_{\ell_i} \right) < \varepsilon$, where the sum is over all cylinder intersections sets   $\bigwedge_{i =0}^{n} P^i_{\ell_i} \subseteq \prod_{0 \leq j \leq n}\cU^{x_j}$ (Definition \ref{def: cylinder inters sets}) that are \emph{not $\varphi$-homogeneous}.
	
	If $T$ has Skolem functions or satisfies strong form of Proposition \ref{2-determinacy with a type}, this holds for all generically stable $\mu_i$, with $P^i_{\ell}$ uniformly definable (i.e.~each defined by an instance of the same fixed formula depending only on $\varphi$ and $\varepsilon$). 
\end{cor}
\begin{proof}
Let $T^{\Sk}$ be a Skolemization of $T$ in the language $\Lcal^{\Sk}$ (note that $T^{\Sk}$ need not be NIP/strongly $n$-distal). 
	By the packing lemma in $T$ (Proposition \ref{prop: dist packing lemma}),  $k \in \mathbb{N}$ and formulas  $\gamma_{i}, \rho_i, \chi_i^-, \chi_i^+$ can be chosen in $\Lcal$ only depending on $\varphi$ and $\varepsilon$, and independently of $\bar{\mu} = (\mu_0, \ldots, \mu_{n-1})$.	Hence, by the proof of Proposition \ref{2-determinacy}, $\xi_i^-, \xi_i^{+}$ are given by instances of a formula in $\Lcal^{\Sk}$ chosen only depending on $\varphi$ and $\varepsilon$ (with parameters depending on $\bar{\mu}$). 
	
	But now, given any finitely supported (so in particular generically stable)  measures $\mu_0, \ldots, \mu_{n}$, they extend uniquely to measures in $T^{\Sk}$ and $\otimes$ corresponds to just the product measure, so even though $\xi_i^-, \xi_i^{+}$ need not be $\Lcal$-definable, the integrals with respect to $\mu_n$ over the $\mu_0 \otimes \ldots \otimes \mu_{n-1}$-measure of the fibers of $\xi_i^-, \xi_i^{+}$ are still well defined (and correspond simply to a weighted finite sum over the support of $\mu_n$), so we still have $\xi^-(x_0,\dots, x_n)  \rightarrow \varphi(x_0,\dots, x_{n-1}, x_{n}) \rightarrow \xi^+(x_0,\dots, x_n) $ and $\mu_0 \otimes \ldots \otimes \mu_n ( \xi^+(x_0,\dots, x_n) ) - \mu_0 \otimes \ldots \otimes \mu_n(\xi^-(x_0,\dots, x_n)) < \varepsilon$.
	
	So both $\xi^+, \xi^-$ can be written as $\bigvee_{t < k_1} \bigwedge_{0 \leq i \leq n} \psi^t_i((x_j : j \in \{0, \ldots, n\} \setminus \{i\}))$
	for some $k_1 = k_1(\varphi, \varepsilon)$ and $\psi^t_i \in \Lcal^{\Sk}(\cU)$.	For each fixed $0 \leq i \leq n$, let $(P^i_{\ell})_{\ell < k_2}$ list the atoms of the Boolean algebra of subsets of $\prod_{j \neq i}\cU^{x_j}$ generated by the formulas $\psi^t_i$ for all $t$ that appear in $\xi^+, \xi^-$, we have $k_2 = k_2(\varphi, \varepsilon)$. Now let $(\ell_0, \ldots, \ell_n)$ be arbitrary. We identify formulas with the sets they define. If $\bigwedge_{i =0}^{n} P^i_{\ell_i}$ is not disjoint from $\xi^-$, then it is contained in $\xi^-$, hence  contained in $\varphi$. If it is not contained in $\xi^{+}$, then it is disjoint from $\xi^{+}$, hence contained in $\neg \varphi$. It follows that all $\bigwedge_{i =0}^{n} P^i_{\ell_i}$ that are not $\varphi$-homogeneous are contained in $\xi^+ \setminus \xi^-$, so their total measure is at most $\varepsilon$.
	
	And if $T$ had Skolem functions, the proof above goes through for arbitrary generically stable measures $\mu_0, \ldots, \mu_{n}$, showing that additionally $P^i_{\ell}$ are definable by instances of formulas chosen depending only on $\varphi, \varepsilon$ (as Boolean combinations of bounded size of such formulas).
\end{proof}

\begin{problem}
	Does Corollary \ref{cor: str n-dist reg lemma} hold without assuming NIP?
		In the case $n=1$ (distal regularity lemma established in \cite[Proposition 5.3]{chernikov2018regularity}), one has uniform definability of the partition and polynomial dependence of its size on $\frac{1}{\varepsilon}$ without any additional assumption. Does this still hold for higher $n$? \cite[Theorem 5.8]{chernikov2018regularity} also establishes a regularity lemma for hypergraphs of any arity in $1$-distal theories. Does  Corollary \ref{cor: str n-dist reg lemma} generalize to from $(n+1)$-ary hypergraphs to $n'$-ary hypergraphs with $n' \geq n+1$ in strongly $n$-distal theories?
\end{problem}

\begin{remark}\label{rem: n-dep reg lemma}
	We note that Corollary \ref{cor: str n-dist reg lemma} is a (non-definable)  strengthening of the $n$-dependent hypergraph regularity lemma established in \cite{chernikov2020hypergraph} (to which we refer for a more precise version of the statement):
	\begin{fact}
	(No assumption on $T$.)  Given an $n$-dependent formula $\varphi(x_0, \ldots, x_{n}) \in \Lcal$ and $\varepsilon > 0$ there exists $k \in \mathbb{N}$ satisfying the following. Given any global generically stable Keisler  measures $\mu_0(x_0), \ldots, \mu_n(x_n)$ (or just definable measures so that $\mu_i$ commutes with $\mu_j$ for all $0 \leq i,j \leq n$), for every $0 \leq i \leq n$ there is a  partition $(P^i_{\ell} : \ell < k)$ of $\prod_{j \neq i}\cU^{x_j}$, with each $P^i_{\ell}$ defined by an instance of the same fixed formula (in fact, a Boolean combination of instances of $\varphi$) depending only on $\varphi$ and $\varepsilon$, so that $\sum_{(\ell_0, \ldots, \ell_n)} \mu_0 \otimes \ldots \otimes \mu_n \left( \bigwedge_{i =0}^{n} P^i_{\ell_i} \right) < \varepsilon$, where the sum is over all cylinder intersections sets   $\bigwedge_{i =0}^{n} P^i_{\ell_i} \subseteq \prod_{0 \leq j \leq n}\cU^{x_j}$ (Definition \ref{def: cylinder inters sets}) that are not \emph{$\varepsilon$-homogeneous} with respect to $\varphi$, i.e.~for which $\frac{\mu_0 \otimes \ldots \otimes \mu_n \left( \varphi \, \cap \, \bigwedge_{i =0}^{n} P^i_{\ell_i} \right)}{\mu_0 \otimes \ldots \otimes \mu_n \left(\bigwedge_{i =0}^{n} P^i_{\ell_i} \right)} \notin [0, \varepsilon) \cup (1-\varepsilon, 1]$.
	\end{fact} 
\end{remark}

As an immediate application of the $n$-distal hypergraph regularity lemma (Corollary \ref{cor: str n-dist reg lemma}), we see that $(n+1)$-hypergraphs definable in strongly $n$-distal NIP theories satisfy the following hypergraph version of the strong Erd\H{o}s-Hajnal property:

\begin{cor}\label{cor: n-str EH}
	Let $T$ be NIP and strongly $n$-distal. Then every definable relation $\varphi(x_0, \ldots, x_{n}) \in \Lcal$ satisfies the \emph{$n$-strong Erd\H{o}s-Hajnal property}, or \emph{$n$-sEH}: there exists $\alpha > 0$ satisfying the following. Given  any finitely supported measures $\mu_0(x_0), \ldots, \mu_n(x_n)$, there exists a cylinder intersection set $C = \bigwedge_{i =0}^{n} C_{i} \subseteq \prod_{0 \leq j \leq n}\cU^{x_j}$ (with $C_i \subseteq \prod_{j \neq i}\cU^{x_j}$) so that $\mu_0 \otimes \ldots \otimes \mu_n \left(C \right) \geq \alpha$ and $C$ is $\varphi$-homogeneous.
	
	If $T$ has Skolem functions or satisfies strong form of Proposition \ref{2-determinacy with a type}, this holds for all generically stable $\mu_i$, with $C$ uniformly definable (i.e.~defined by an instance of the same fixed formula depending only on $\varphi$). 
\end{cor}
\begin{proof}
	Let $k = k(\varphi)$ be as given by Corollary \ref{cor: str n-dist reg lemma} with $\varepsilon = \frac{1}{2}$. Then for any finitely supported $\mu_0, \ldots, \mu_n$ we have partitions $(P^i_{\ell} : \ell < k)$ of $\prod_{j \neq i}\cU^{x_j}$ so that $\sum_{(\ell_0, \ldots, \ell_n)} \mu_0 \otimes \ldots \otimes \mu_n \left( \bigwedge_{i =0}^{n} P^i_{\ell_i} \right) < \frac{1}{2}$, where the sum is over all cylinder intersections sets   $\bigwedge_{i =0}^{n} P^i_{\ell_i} \subseteq \prod_{0 \leq j \leq n}\cU^{x_j}$ that are \emph{not $\varphi$-homogeneous}. Then there has to exist some $(\ell_0, \ldots, \ell_n)$ so that $C := \bigwedge_{i =0}^{n} P^i_{\ell_i}$ is $\varphi$-homogeneous and $\mu_0 \otimes \ldots \otimes \mu_n \left( C \right) \geq  \frac{1}{2 k^{n+1}}$. So the claim holds with $\alpha = \alpha(\varphi) := \frac{1}{2 k^{n+1}}$.
\end{proof}

\begin{problem}
	Corollary \ref{cor: n-str EH} provides a partial generalization of \cite[Theorem 3.1]{chernikov2018regularity} in the case $n=1$. The result there also concludes that $C$ uniformly definable (without any additional assumptions), and moreover that this \emph{definable} strong Erd\H{o}s-Hajnal property characterizes distality \cite[Theorem 6.10]{chernikov2018regularity}. We leave open a generalization of this part for higher $n$.
\end{problem}

We can rephrase Corollary \ref{cor: n-str EH} as a statement for families of finite hypergraphs:
\begin{defn}\label{def: k-sEH for finite}
Fix $k \geq 1$, and let $\mathcal{H}$ be a family of finite (uniform) $(k+1)$-partite $(k+1)$-hypergraphs. We say that $\mathcal{H}$ satisfies the \emph{$k$-strong Erd\H{o}s-Hajnal property}, or \emph{$k$-sEH}, if there is $\alpha >0$ so that: for every $H = (V_0, \ldots, V_k; E) \in \mathcal{H}$ with $E \subseteq V := V_0 \times \ldots \times V_k$ there exists a cylinder intersection subset $C \subseteq V$ with $|C| \geq \alpha |V|$ so that either $C \subseteq E$ or $C \subseteq V \setminus E$.
\end{defn}

\begin{remark}
Note that $1$-sEH corresponds to the usual strong 	Erd\H{o}s-Hajnal property for families of finite graphs considered in the literature (see \cite{fox2008erdHos}, and also \cite{chernikov2018regularity} for discussion). Indeed, in this case a cylinder intersection set is just a rectangle of the form $C = C_1 \times C_2$ for some $C_i \subseteq V_i$, and  $|C| \geq \alpha |V_1||V_2|$ if and only if $|C_i| \geq \sqrt{\alpha} |V_i|$ for both $i$ ---    which corresponds to an $E$-homogeneous pair of sets in the usual formulation.
\end{remark}

\noindent Now given a structure $M$ and a partitioned formula $\varphi(x_0, \ldots, x_{n})$ with all $x_i$ tuples of variables, we have the associated family $\mathcal{H}_{\varphi}$ of all finite $(n+1)$-partite $(n+1)$-hypergraphs $H = (V_0, \ldots, V_n; E)$ where $V_i \subseteq M^{x_i}$ are finite subsets and $E = \{(a_0, \ldots, a_n) \in V_0 \times \ldots \times V_n : M \models \varphi(a_0, \ldots, a_n) \}$. Then Corollary \ref{cor: n-str EH} says in particular that if $\Th(M)$ is NIP and strongly $n$-distal, then for every formula $\varphi(x_0, \ldots, x_{n})$  the family of finite hypergraphs $\mathcal{H}_{\varphi}$ satisfies the $n$-strong Erd\H{o}s-Hajnal property in the sense of Definition \ref{def: k-sEH for finite}.

\subsection{Average measures and $n$-distality}\label{sec: n-distality vs av meas}

In this section we demonstrate that $n$-distal (rather than strongly $n$-distal) NIP theories are characterized by $n$-determinacy for $(n+1)$-tuples of measures obtained by averaging over mutually indiscernible sequences, generalizing the results in  \cite[Proposition 2.21]{simon2013distal} in the case $n=1$.

\begin{example}\label{average measures example}\cite{simon2012finding, hrushovski2013generically}
    Assume $T$ is NIP and $(b_t)_{t\in [0,1]}$ is an indiscernible sequence indexed by the interval $[0,1]$. Let $\lambda$ denote the Lebesgue measure on $[0,1]$. Take $\varphi(x)\in \Lcal(\Mbb)$. By NIP, the set $\{t\in [0,1]: \models \varphi(b_t)\}$ is a finite union of intervals. So, we can define the \emph{average measure of $(b_i)$} by $\Av(b_{[0,1]})(\varphi) :=  \lambda(\{t\in [0,1]: \models \varphi(b_t)\})$. Clearly, $\Av(b_{[0,1]})$ is f.s. (hence invariant) over $b_{[0,1]}$. By compactness and the bound on alternation given by NIP, the measure $\Av(b_{[0,1]})$ is also definable over $b_{[0,1]}$, hence it is generically stable.
\end{example}

%We will use average measures to show that, provided $T$ has NIP, $n$-distality is necessary for $(n+1)$-tuples of generically stable measures to be $n$-determined. We also show that under the assumption of $n$-distality and NIP, $(n+1)$-tuples of average measures formed over mutually indiscernible sequences are $n$-determined.

%The approach of this section essentially Indeed, Proposition 2.21 is essentially the case $n=1$ in the following.

\begin{notation}
    In this section $\Ical$ will typically denote a sequence of parameters (rather than its indexing set), since the sequences we consider in this section are indexed by $[0,1]$ (though sometimes it will be convenient to abuse the notation by conflating elements of the sequences with their indices). If $\Ical$ is a sequence of parameters indexed by $[0,1]$ and $A\sub [0,1]$, we will denote by $\Ical_A$ the subsequence of $\Ical$ indexed by $A$. 
    Given a formula $\varphi(x)\in \Lcal$ and a partition $x = yz$ (despite the notation, we do not assume that the partition respects the ordering on $x$), a \emph{$(\varphi, y)$-type over a set $A \subseteq \cU^z$} is a maximal consistent set of formulas of the form $\varphi(y,c)$ for $c\in A$. 
\end{notation}

We first note that the tensor product of average measures over  mutually indiscernible sequences behaves as expected.

\begin{fact}\label{NIP implies boxes on mut ind}(see e.g.~\cite[Lemma 3.5]{chernikov2024n}) 
    Assume $T$ is NIP and $(\Ical_i)_{i<n}$ are mutually indiscernible sequences with $\Ical_i$ a sequence of elements in $\cU^{x_i}$ indexed by $[0,1]$.
    Then for  each $\varphi(x)\in \Lcal(\Mbb)$  with $x = x_0 \ldots x_{n-1}$ 
    there are finite partitions of each $\Ical_i$ into intervals $(V_{i, l} : l <k)$ so that  $\varphi(\prod\limits_{i<n}\Ical_i)$ is given by a finite disjoint union of (some of the) products of intervals $\prod\limits_{i<n}V_{i,l_i}$ --- we will informally refer to such sets as \emph{boxes}. 
   \end{fact}

\begin{lemma}\label{tensor product of average measures is lebesgue}
    Suppose $T$ has NIP and $(\Ical_i)_{i<m}$ are mutually indiscernible with each $\Ical_i$ indexed by $[0,1]$. Let $\mu_i:= \Av(\Ical_i)\in \mathfrak{M}_{x_i}(\cU)$. Then $(\bigotimes\limits_{i<m} \mu_i)(\varphi(x_0, \ldots, x_{m-1})) = \lambda_m(\{(i_0,\dots,i_{m-1}): \models \varphi(a_{i_0},\dots,a_{i_{m-1}})\})$ for all $\varphi \in \Lcal(\cU)$, where $\lambda_m$ is the Lebesgue measure on $[0,1]^m$. 
\end{lemma}

\begin{proof}
    The case $m=1$ is trivial, we then argue by induction on $m \geq 2$.  For any small model $\Ncal$ containing the parameters of $\varphi(x_0, \ldots, x_{m-1})$ and $\Ical_{<m}$, we have 
     \[\left( \left(\bigotimes\limits_{i<m-1} \mu_i \right) \otimes \mu_{m-1} \right)(\varphi((x_i)_{i<m}))= \int_{q\in S_{x_{m-1}}(\Ncal)} \left(\bigotimes\limits_{i<m-1}\mu_i \right)(\varphi((x_i)_{i<{m-1}};q)d\mu_{m-1}.\] 

    By Fact \ref{NIP implies boxes on mut ind},  there are finite partitions of each $\Ical_i$ into intervals $(V_{i, l} : l <k)$ so that  $\varphi(\prod\limits_{i<m}\Ical_i)$ is given by a finite disjoint union of (some of the) products of intervals $\prod\limits_{i<m}V_{i,l_i}$.

%    
%     a finite disjoint union of boxes $(U_k)_{k<l}$. For each $k<l$, we have $U_k = V_k\times W_k$ for a box $V_k\sub \prod\limits_{i<m-1} \Ical_i$ and an interval $W_k\sub \Ical_{m-1}$. For each $\eta \in 2^l$, let $W_\eta:= \bigcap\limits_{k<l} W_k^{\eta(i)}$ (where $W_i^1 := W_i$ and $W_i^0 := \Ical_{m-1}\sm W_i$).
    
     Now, for each $k < l$, all elements of $V_{m-1,l}$ have the same $(\varphi,x_{m-1})$-type over $\Ical_{<m-1}$. Thus, there are finitely many $(\varphi,x_{m-1})$-types over $\Ical_{<m-1}$ realized in $\Ical_{m-1}$, enumerate them as $(p_j)_{j<N}$. Now, suppose some $q\in S_{x_{m-1}}(\Ncal)$ entails some other $(\varphi,x_{m-1})$-type over $\Ical_{<m-1}$. Then, by definition of the average measure, necessarily $q$ contains some instance of $\varphi$ that has $\mu_{m-1}$-measure $0$, so $q\notin S(\mu_{m-1})$. In the integral above we can restrict to types in the support of $\mu_{m-1}$ (the support is a closed set of full measure in the type space). We pick formulas $(\psi_j(x_m))_{j<N}$ from $\Lcal(\Ical_{<m-1})$ such that $\psi_j$ isolates $p_j$ among $(p_j)_{j<N}$ (given by Boolean combinations of instances of $\varphi$). Now, the sets $Z_j := [\psi_j(x_{m-1})]\cap S(\mu_{m-1})$ form a partition of $S(\mu_{m-1})$. The integrand $(\bigotimes\limits_{i<m-1}\mu_i)(\varphi((x_i)_{i<m-1};q))$ is constant on each $Z_j$: by the induction hypothesis, for a given $Z_j$ this constant value is given by  $\lambda_{m-1}(\varphi(\prod\limits_{i<m-1}\Ical_i;b_j))$ for some fixed $b_j$ realizing $\psi_j(x_{m-1})$ and $\tp(b_j/\Ncal)\in S(\mu_{m-1})$. 
     And for each $j$, $\mu_{m-1}(\psi_j(x_{m-1}))$ is given by the Lebesgue measure $\lambda_1$ of the corresponding union of intervals that it defines in $\Ical_{m-1}$ by the case  $m=1$. We can then calculate the integral of this simple function with respect to this partition of the domain summing over boxes, which gives the required value.
\end{proof}

\begin{definition}
\begin{enumerate}
	\item  Given a sequence $\Ical$, a \emph{polarized Dedekind cut} (abbreviated as \emph{p.D.~cut}) is a cut $(A,B)$ of $\Ical$ (i.e.~$A \cap B = \emptyset, A \cup B = \Ical$) with a distinguished polarity $\cfrak \in \{A,B\}$ for which the $\cfrak$ has infinite cofinality if $\cfrak = A$ and infinite coinitiality if $\cfrak = B$. If the sequence is indexed by a Dedekind-complete linear order, we denote by $\End(\cfrak) \in \Ical$ the least upper bound of $A$ if $\cfrak = A$, or the greatest lower bound of $B$ if $\cfrak = B$ (note here that $\End(\cfrak)$ is not an element of $\cfrak$).  
	\item  Given an $m$-dimensional product of ordered sets $\prod\limits_{i<m} U_i$ and  $\ub\in \prod\limits_{i<m} U_i$, we say a set $S$ is an \emph{$m$-ant with respect to the center $\ub$} if $S=\prod\limits_{i<m} V_i$ where $V_i\in \{\{v\in U_i: v>\ub(i)\},\{v\in U_i: v<\ub(i)\}\}$.
\end{enumerate}
   
\end{definition}

\begin{remark}
    Given $(\Ical_i)_{i<m}$ mutually indiscernible with p.D.~cuts $\cfrak_i$ in $\Ical_i$ and each sequence Dedekind complete, $\prod\limits_{i<m} \cfrak_i$ is an $m$-ant with center $(\End(\cfrak_i))_{i<m}$. We call this the \emph{induced $m$-ant}.
\end{remark}

\begin{definition}
    Suppose $T$ has NIP, $(\Ical_i)_{i<m}$ are mutually indiscernible and $\bar \cfrak:=(\cfrak_i)_{i<m}$ is a tuple of p.D.~cuts $\cfrak_i$ in $\Ical_i$. Given any set $A$, let $\lim_A(\bar\cfrak)$ consist of the formulas $\varphi(x_0,\dots,x_{m-1})\in \Lcal(A)$ such that there are $a_i\in \cfrak_i$ with $\models \varphi(b_0,\dots,b_{m-1})$ whenever each $b_i$ is an element of $\Ical_i$ strictly between $a_i$ and $\End(\cfrak_i)$ in the ordering. 
\end{definition} 

\begin{remark}
    In the case that the sequences are indexed by $[0,1]$ it is equivalent to define the limit with respect to the Euclidean metric restricted to the induced $m$-ant. We will use this formulation when convenient.
\end{remark}

\begin{lemma}\label{NIP limit type complete}
    If $T$ has NIP, then for any $(\Ical_i)_{i<m}$ mutually indiscernible with a tuple of p.D.~cuts $(\cfrak_i)_{i<m}$, $\lim_A(\bar\cfrak)$ is a complete type over $A$. 
\end{lemma}

\begin{proof}  
Let $\rb := (\End(\cfrak_i))_{i<m}$. Suppose for contradiction that there is $\varphi(x_0, \ldots, x_{m-1}) \in \Lcal(A)$ and a (coordinate-wise) monotone  sequence $\jb_k$ from the induced $m$-ant with $\models \varphi(\jb_k)$ if and only if $k$ is even, and $\jb_k \rightarrow \rb$ in the Euclidean metric (of the indices). By mutual indiscernibility, $(\jb_k)_{k < \omega}$ is indiscernible, and hence this contradicts NIP.
\end{proof}

\begin{lemma}\label{ordering cuts}
    Suppose $\Ical$ is indiscernible indexed by $[0,1]$ and not totally indiscernible. Then all p.D.~cuts have different limit types over some small set containing $\Ical$. In fact, all p.D.~cuts $\cfrak$ with $\End(\cfrak) \in \Ical_{(0,1)}$ have different limit types over $\Ical$. 
\end{lemma}

\begin{proof}
    Suppose to begin that $\End(\cfrak)\in (0,1)$. Then suppose $\Jcal$ is a subinterval containing an open neighborhood of $\End(\cfrak)$ but omitting open neighborhoods of $0$ and $1$. Since $\Ical$ is not totally indiscernible, there are $0\leq j <k$,  $i_0<\dots<i_k$ from $\Ical$ and $\varphi \in \Lcal$ such that $\models \varphi(i_0,\dots, i_j, i_{j+1},\dots,i_k)$ and $\models \neg\varphi(i_0,\dots, i_{j+1}, i_{j},\dots,i_k)$ (using indiscernibility and the fact that every permutation is a composition of transpositions of adjacent elements). By using indiscernibility to shift the parameters $i_{
    <j}$ and $i_{>j+1}$ outside of $\Jcal$, we have a formula with parameters from $\Ical\sm \Jcal$ that orders $\Jcal$. So, there is a formula with parameters from $\Jcal \cup \{\End(\cfrak)\}$ that distinguishes the p.D.~cut $\cfrak$ from the p.D.~cut with the same endpoint and opposite polarity. In addition, for any $\alpha\in (0,1)\sm \{\End(\cfrak)\}$, there is a formula $\varphi_\alpha(x)$ with parameters from $\Ical$ such that $\models \varphi_\alpha(a_{\End(\cfrak)})$ but $\models \neg \varphi_\alpha(a_{\beta})$ for any $\beta$ in some open neighborhood of $\alpha$. So, if $\dfrak$ is a p.D. cut with $\End(\dfrak)\in (0,1)\sm \{\End(\cfrak)\}$, then $\lim_\Ical(\cfrak)\neq \lim_\Ical(\dfrak)$. This shows that p.D.~cuts with distinct endpoints in $\Ical_{(0,1)}$ have distinct limit types over $\Ical$.

    To conclude the full result, consider an automorphism $\sigma\in \Aut(\Mbb)$ mapping $\Ical$ onto the subinterval indexed by $[\frac{1}{4}, \frac{3}{4}]$. Then consider $\sigma^{-1}(\Ical)$. This is an indiscernible sequence containing $\Ical$ as a convex subset that omits open neighborhoods of $0$ and $1$. By the argument above, p.D. cuts of $\Ical$ have distinct limit types over $\sigma^{-1}(\Ical)$.
\end{proof}

\begin{lemma} \label{support of tensor product of averages is the limit types precise lemma}
    Suppose $T$ has NIP, $(\Ical_i)_{i<m}$ are mutually indiscernible with each $\Ical_i$ indexed by $[0,1]$. Denote $\mu_i:= \Av_\Mbb(\Ical_i)$. If $p\in S(\bigotimes\limits_{i<m} \mu_i)$, then there are p.D.~cuts $\cfrak_i$ in $\Ical_i$ such that $p$ is the limit type (over $\Mbb$) of the tuple of p.D.~cuts $(\cfrak_i)_{i<m}$. Moreover, if $p\in  S(\bigotimes\limits_{i<m} \mu_i)$ and $p|_{x_i} = \lim(\cfrak_i)$, then $p = \lim((\cfrak_i)_{i<m})$. 
\end{lemma}

\begin{proof}
Suppose $p(x) \in S(\bigotimes\limits_{i<m} \mu_i)$ and consider the set 
$\Gamma:=\bigcap\limits_{\varphi\in p} \cl \left(\left\{\jb \in [0,1]^m:  \models \varphi(a_{\jb})\right\} \right)$.
 This is an intersection of closed sets in a compact space whose finite subintersections are non-empty (by the fact that $p$ is in the support and Lemma \ref{tensor product of average measures is lebesgue}), so it a non-empty set. Choose $\rb\in \Gamma$. Since $p$ is a complete global type in the support, $x\neq a_{\rb} \in p$. So, for each $\varphi\in p$, the formula $\psi_\varphi:=\varphi \land (x\neq a_{\rb})$ is also in $p$. But $\rb \in \cl(\{\jb \in [0,1]^m: \models \psi_\varphi(a_{\jb})\})$. Hence, there is a sequence $(\jb^\varphi_k)_{k\in 
\omega} \sub [0,1]^m$ with $\models \varphi(a_{\jb_k^\varphi})$ and $\jb^\varphi_k \rightarrow \rb$ in the Euclidean metric.  

We want to choose these sequences from the same $m$-ant  for all $\varphi$ simultaneously. This involves choosing a polarity (with respect to $\rb(i)$) for each coordinate $i$. Now, if $\rb(i)\in \{0,1\}$ for any $i<m$, there is only one choice of polarity for the coordinate $i < m$. So, fix $i$ and assume $\rb(i)\notin\{0,1\}$. Note that if any sequence $\Ical_i$ is totally indiscernible over $\Ical_{\neq i}$, applying an automorphism reversing the order on $\Ical_i$ and fixing $\bar{r}(i)$ and $\Ical_{\neq i}$ shows that for any $\varphi$ we can choose a sequence  $(\jb^\varphi_k)_{k\in 
\omega}$ with $i$th coordinates on the same side of $\rb(i)$. So, we can disregard this case too. Now, if $\Ical_i$ is not totally indiscernible over the other sequences, there is a formula (with parameters) that orders any subsequence of $\Ical_i$ omitting an infinite initial segment and an infinite terminal segment (by the shifting argument in the proof of Lemma \ref{ordering cuts}). Since $\rb(i)\notin \{0,1\}$, and $\jb_k^\varphi(i)\rightarrow \rb(i)$, we can just assume that $\Ical_i$ is ordered by such a formula, $x_i<x_i'$. Now, either $x_i<a_{\rb(i)}\in p$ or $a_{\rb(i)}<x_i\in p$. By modifying each $\psi_\varphi$ with this formula as another conjunct, we can force the $i$th coordinate of $\jb^\varphi_k$ to have the same polarity with respect to $\rb(i)$, uniformly in $\varphi$. Concluding, we can do this for every coordinate and hence restrict the $\jb^\varphi_k$ to a single $m$-ant, uniformly in $\varphi$. Then by the argument given in the proof of Lemma \ref{NIP limit type complete}, by NIP this suffices to show that $p$ is the limit type of the tuple of cuts inducing the $m$-ant with the common polarization and center $\rb$.

Now we show the moreover part. If the sequence $\Ical_i$  is totally indiscernible over the other sequences, we can replace $\cfrak_i$ with any other p.D. cut $\dfrak$ in $\Ical_i$ without changing the limit type of the entire tuple of cuts (taking automorphisms). If $\Ical_i$ is not totally indiscernible over the other sequences, the shifting argument from Lemma \ref{ordering cuts} shows that the cut $\cfrak_i$ is uniquely determined by the restriction to the variable $x_i$. Thus, we are done.
\end{proof}

We thus get a generalization of \cite[Lemma 2.20]{simon2013distal}:
\begin{cor}
    With notation as in Lemma \ref{support of tensor product of averages is the limit types precise lemma}, 
    $$S \left(\bigotimes\limits_{i<n}\mu_i \right)=\left\{\lim((\cfrak_i)_{i<n}): \cfrak_i\text{ a p.D. cut in }\Ical_i\right\}.$$
\end{cor}

\begin{proof}
    We have $S(\bigotimes\limits_{i<n}\mu_i)\sub\{\lim((\cfrak_i)_{i<n}: \cfrak_i\text{ a p.D. cut in }\Ical_i\}$ by Lemma \ref{support of tensor product of averages is the limit types precise lemma}. And $\{\lim((\cfrak_i)_{i<n}): \cfrak_i\text{ a p.D. cut in }\Ical_i\}\sub S(\bigotimes\limits_{i<n}\mu_i)$ is immediate by Lemma \ref{tensor product of average measures is lebesgue}.
\end{proof}

For the remainder of this section we let $1 \leq n < \omega$ and denote by $(\Ical_i)_{i\leq n}$ a tuple of mutually indiscernible sequences each indexed by $[0,1]$ and $\mu_i:=\Av_\Mbb(\Ical_i)$ for each $i$. We will further denote $\nu:=(\bigotimes\limits_{i\leq n} \mu_i)|_{\Lcal^n_{x_0,  \ldots, x_{n}}(\Mbb)}$ (we will simply write $\Lcal^n_{n+1}$ from now on).

\begin{lemma}\label{realising the box type of averages is to insert into some cuts}
    Suppose $T$ has NIP. Suppose $p((x_i)_{i\leq n}) \in S(\nu)$ (the support of $\nu$, see Section \ref{sec: meas on Bool alg}). Then, for any $\ab \models p|_{(\Ical_i)_{i\leq n}}$, $\ab$ $n$-inserts (see Definition \ref{def: insterts}) into some tuple of Dedekind cuts $(\cfrak_i)_{i\leq n}$ where $\cfrak_i$ is a cut in $\Ical_i$. 
\end{lemma}

\begin{proof}
   Let $\ab \models p|_{(\Ical_i)_{i\leq n}}$ be given. By definition of $\nu$ we have $p_i:=p|_{x_i}\in S(\mu_i)$. So, $p_i=\lim(\cfrak_i)$ for some p.D.~cut $\cfrak_i$ in $\Ical_i$. Fix any injection $\sigma:n\rightarrow n+1$, but (without loss of generality, to ease notation) suppose $\sigma = \id_n$. We need to show that $\ab_{i<n}$ inserts into the $n$-tuple of cuts $(\cfrak_i)_{i<n}$ while preserving the mutual indiscernibility of the $(n+1)$-tuple of sequences. By the moreover part of Lemma \ref{support of tensor product of averages is the limit types precise lemma}, since $p|_{x_0\dots x_{n-1}} \in S(\bigotimes\limits_{i<n} \mu_i)$, $p|_{x_0\dots x_{n-1}} = \lim((\cfrak_i)_{i<n})$, so we are done.
\end{proof}

\begin{notation}
    In the following, we denote $\partial_{\Lcal^n_{x_0,  \ldots, x_{n}}(\Mbb)} \varphi$ by $\partial_n\varphi$ (identifying formulas with the subsets of $\cU$ they define, see Definition \ref{def: bored of a set}).
\end{notation}

Finally, we can deduce $n$-determinacy for measures in this special case: 
\begin{proposition}\label{determination of average measures}
    Suppose $T$ has NIP. If $(\Ical_i)_{i\leq n}$ is $n$-distal (Definition \ref{def: n-dist mut ind seqs}), then $(\mu_i)_{i\leq n}$ is $n$-determined (Definition \ref{def: n-det for tuples of measures}).
\end{proposition}

\begin{proof}
    Suppose that $(\Ical_i)_{i\leq n}$ is $n$-distal, but for contradiction $\nu(\partial_n\varphi)>0$ (applying Proposition \ref{border and determination}). By Lemma \ref{tensor product of average measures is lebesgue} we have that for every formula $\psi(\bar{x}) \in \Lcal(\cU)$, if $\bigotimes_{i \leq n} \mu_i(\psi(\bar{x})) > 0$, then $\psi(\cU)$ is infinite. Then, using that $\partial_n\varphi$ and $S(\nu)$ are both given by intersections of clopens, by compactness of the space $S_{(\bar{x}_i)_{i < \omega}}(\cU)$ (with $\bar{x}_i$ a copy of $\bar{x}$), we can find infinitely many elements in $\partial_n\varphi\cap S(\nu)$.  By Lemma \ref{realising the box type of averages is to insert into some cuts}, any realization of one of these types (over $(\Ical_i)_{i\leq n}$) will $n$-insert into some tuple of p.D. cuts. We can clearly assume these cuts do not share any endpoints. By choosing a subsequence of an arbitrary enumeration, we can ensure that the endpoints of these cuts are monotone in each coordinate. We can then choose realizations alternating with respect to $\varphi$, since each type is in the border of $\varphi$. By inductively applying $n$-distality to insert these $(n+1)$-tuples (removing the original endpoints of the cut so that the sequences remain dense), this contradicts NIP.    
\end{proof}

Now we consider the converse implication.

\begin{definition}
    Suppose $T$ is NIP and $(\Ical_i)_{i\leq n}$ is not $n$-distal. We will say that $((\cfrak_i)_{i\leq n},\ab,\varphi(\xb))$ \emph{witnesses} that $(\Ical_i)_{i\leq n}$ is not $n$-distal if $\ab$ $n$-inserts into $(\cfrak_i)_{i\leq n}$, $\models\varphi(\ab)$, and $\varphi(\xb)\in \Lcal(\Ical_{\leq n}) \sm \lim_{\Ical_{\leq n}}(\cfrak_{\leq n})$.
\end{definition}

\begin{remark}
    By indiscernibility, it is clear that $\lim_{\Ical_{\leq n}}(\cfrak_{\leq n})$ is well-defined even for non-polarized Dedekind cuts, and that $\ab$ inserts into $\Ical_{\leq n}$ just in case it realizes this type. 
\end{remark}

\begin{proposition}
     If $T$ is NIP and $(\mu_i)_{i\leq n}$ is $n$-determined, then $(\Ical_i)_{i\leq n}$ is $n$-distal.
\end{proposition}

\begin{proof}
    Suppose that $(\Ical_i)_{i\leq n}$ is not $n$-distal. Then we can fix any tuple of cuts $(\cfrak_i)_{i\leq n}$, remove the realization of each cut so they become Dedekind (note this has no effect on the average measure) and a tuple $\ab$ that $n$-inserts into these Dedekind cuts but does not insert into them.
    
    By NIP and base change in the form \cite[Lemma 3.11]{walker2023distality} (based on \cite[Lemma 2.8, Corollary 2.9]{simon2013distal}), we can ensure that $\ab$ (in a bigger monster model) realizes $p_n:=\lim_\Mbb((\cfrak^{-}_{i} : i \leq n ))|_{\Lcal^n_{(n+1)}(\Mbb)}$ (where $\cfrak^-$ denotes the lower portion of a cut $\cfrak$). Let $\varphi((x_i)_{i\leq n})$ with parameters from $(\Ical_i)_{i\leq n}$  witness the failure of $n$-distality, with respect to $(\ab, \cfrak_{\leq n})$. 
    
    Hence, $p_n \in \partial_n \varphi$. Since the $\cfrak_i$ are Dedekind, there are open intervals $U_i$ around each $\cfrak_i$ which avoid the parameters of $\varphi$. Hence, by indiscernibility and automorphisms, we can do the same process for any other tuple of Dedekind cuts $(\dfrak_i)_{i\leq n}$ whose endpoints come from the $U_i$. In this case, $\lim_\Mbb(\dfrak^-_{\leq n})|_{\Lcal^n_{n+1}(\Mbb)}\in \partial_n \varphi$ also. Note that $\lambda_{n+1}(\prod\limits_{i \leq n} U_i)>0$, and we aim to show that $\nu(\partial_n\varphi) \geq \lambda_{n+1}(\prod\limits_{i\leq n} U_i)>0$.

    Now, $\partial_n\varphi$ is a closed set (Lemma \ref{border is closed}), so by regularity, $\nu(\partial_n\varphi)=\inf\{\nu(D): D\supseteq \partial_n\varphi\text{ clopen}\}$. We will show that for any clopen $D\supseteq \partial_n\varphi$, i.e.~$D:=[\psi]$ for some $\psi\in \Lcal^n_{n+1}(\Mbb)$, $\neg \psi$ relatively defines a measure zero (in Lebesgue measure) subset of $\prod\limits_{i\leq n} U_i$. Suppose otherwise, then by NIP and Fact \ref{NIP implies boxes on mut ind}, $\neg \psi(\prod\limits_{i\leq n}U_i)$ contains an $(n+1)$-dimensional box of positive measure. Choose a tuple of Dedekind cuts $(\dfrak_i)_{i\leq n}$ with endpoints $\End(\dfrak_i)$ in the interior of this box. Then $\neg \psi \in \lim_\Mbb(\dfrak^{-}_{\leq n})$, which by the above contradicts $D\supseteq \partial_n\varphi$. Hence, the border has positive measure $\nu(\partial_n\varphi)>0$, so $(\mu_i)_{i\leq n}$ is not $n$-determined by Proposition \ref{border and determination}.
\end{proof}

\begin{cor}
    Suppose $T$ has NIP. If all $(n+1)$-tuples of generically stable measures  are $n$-determined, then $T$ is $n$-distal.
\end{cor}

 \section{Higher distality and triviality of forking}\label{sec: n-dist and n-triv}

 We consider the following refinement of higher distality (where for each $k \in \omega$, $k = \{0, 1, \ldots, k-1\}$):
 \begin{defn}\label{def: (k,l)-distality}
 For $k \in \mathbb{N}_{\geq 1}$ and $0 \leq \ell \leq k$, we say that $T$ is \emph{$(k,\ell)$-distal} if the following holds.	Given any $\emptyset$-indiscernible sequence $\mathcal{I} = \mathcal{I}_0 + \mathcal{I}_1 + \ldots + \mathcal{I}_{(k+1) - \ell }$ of tuples in $\cU^x$ with each $\mathcal{I}_i$ indexed by $\mathbb{Q}$, $(b_i : i < k +1 - \ell )$ tuples $\cU^x$ and $a_0, \ldots, a_{\ell-1}$ tuples in $\cU^y$, if
 \begin{itemize}
 	\item for any $u \subseteq k+1-\ell$ and $v \subseteq \ell$ with $|u| + |v| = k$, we have that $\mathcal{I}[b_i : i \in u]$ is indiscernible over $(a_i : i \in v)$ (where $\mathcal{I}[b_i : i \in u]$ denotes the sequence obtained from $\mathcal{I}$ by inserting $b_i$ in the cut between $\mathcal{I}_i$ and $\mathcal{I}_{i+1}$ for all $i \in u$ simultaneously),
 \end{itemize}
 then 
 \begin{itemize}
 	\item $\mathcal{I}[b_i : i \in k+1-\ell ]$ is indiscernible over $(a_i : i \in \ell)$.
 \end{itemize}
 \end{defn}

\begin{remark}
	For any $1 \leq k < \omega$, we refer to $(k,0)$-distality as \emph{$k$-distality} and to $(k,k)$-distality as \emph{strong $k$-distality}. This agrees with the corresponding notions considered by Walker \cite{walker2023distality}.
	\end{remark}

%
%
%\begin{remark}
%	Mutually indiscernible sequences ***
%\end{remark}

For simplicity of presentation we will sometimes restrict to the case $k=2$ (and point out when things generalize to arbitrary $k$). In this case Definition \ref{def: (k,l)-distality} specializes to:
\begin{defn}
\begin{enumerate}
\item $T$ is \emph{$2$-distal} (i.e.~$(2,0)$-distal) if for any $\emptyset$-indiscernible sequence $\mathcal{I}_0  + \mathcal{I}_1 + \mathcal{I}_2 + \mathcal{I}_3$, if each of $\mathcal{I}_0 + b_0 + \mathcal{I}_1 + b_1 + \mathcal{I}_2 + \mathcal{I}_3$, $\mathcal{I}_0 +  \mathcal{I}_1 + b_1 + \mathcal{I}_2 + b_2 + \mathcal{I}_3$ and $\mathcal{I}_0 + b_0 + \mathcal{I}_1  + \mathcal{I}_2 + b_2 + \mathcal{I}_3$ is $\emptyset$-indiscernible, then $\mathcal{I}_0 + b_0 + \mathcal{I}_1 + b_1 + \mathcal{I}_2 + b_2 + \mathcal{I}_3$ is indiscernible.

	\item $T$ is \emph{$2^+$-distal} (i.e.~$(2,1)$-distal) if for any $\emptyset$-indiscernible sequence $\mathcal{I}_0 + b_0 + \mathcal{I}_1 + b_1 + \mathcal{I}_2$ and tuple $a$, if each of $\mathcal{I}_0 + b_0 + \mathcal{I}_1 +\mathcal{I}_2, \mathcal{I}_0 + \mathcal{I}_1 + b_1 + \mathcal{I}_2$  is indiscernible over $a$, then $\mathcal{I}_0 + b_0 + \mathcal{I}_1 + b_1 + \mathcal{I}_2$ is also indiscernible over $a$.
	\item $T$ is \emph{strongly $2$-distal} (i.e.~$(2,2)$-distal) if for  any sequence $\mathcal{I}_0  + b_0 + \mathcal{I}_1$ 
	and tuples $a_0, a_1$, if $\mathcal{I}_0 + \mathcal{I}_1$ is indiscernible over $a_0 a_1$, $\mathcal{I}_0 + b_0 + \mathcal{I}_1$ is indiscernible over $a_0$ and $\mathcal{I}_0 + b_0 + \mathcal{I}_1$ is indiscernible over $a_1$, then  $\mathcal{I}_0 + b_0 + \mathcal{I}_1$ is indiscernible over $a_0 a_1$.
\end{enumerate}
\end{defn}

\begin{prop}\label{prop: refinements of n-distality}
	$T$ is strongly $2$-distal $\implies$ $T$ is $2^+$-distal $\implies$
 $T$ is $2$-distal.
 \end{prop}
 \begin{proof}

 	Assume $T$ is strongly $2$-distal, but not $2^+$-distal. Then there exist an $\emptyset$-indiscernible $\mathcal{I}_0 + b_0 + \mathcal{I}_1 + b_1 + \mathcal{I}_2$ and a tuple $a$, so that each of $\mathcal{I}_0 + b_0 + \mathcal{I}_1 +\mathcal{I}_2, \mathcal{I}_0 + \mathcal{I}_1 + b_1 + \mathcal{I}_2$  is indiscernible over $a$, but $\mathcal{I}_0 + b_0 + \mathcal{I}_1 + b_1 + \mathcal{I}_2$ is not indiscernible over $a$. Then there exist a formula $\varphi \in \mathcal{L}(\emptyset)$  and increasing finite tuples $\bar{i}_t \in I_t^{<\omega}$ so that $\models \varphi \left(b_{\bar{i}_0} , b_0, b_{\bar{i}_1}, b_1, b_{\bar{i}_2}, a \right)$, but $\models  \neg \varphi \left(b_{\bar{i}_0} , b_{i_0}, b_{\bar{i}_1}, b_{i_1}, b_{\bar{i}_2}, a  \right)$ for any $i_0, i_1 \in I_0 + I_1 + I_2$ with $\bar{i}_0 < i_0 < \bar{i}_1 < i_1 < \bar{i}_2$. 
 Fix some $i^* \in I_1$ with $i^* < \bar{i}_1$. Let $a'_0$ be the tuple consisting of all elements appearing in $\left( b_{i} : i \in I_1, i^* < i \right)  + b_1 + \mathcal{I}_2$, $a'_1 := a$, $\mathcal{I}'_0 := \mathcal{I}_0, \mathcal{I}'_1 := \left( b_i : i \in I_1,  i < i^*\right)$. It follows from the assumptions that $\mathcal{I}'_0 + \mathcal{I}'_1$ is indiscernible over $a'_0 a'_1$, $\mathcal{I}'_0 + b_0 + \mathcal{I}'_1$ is indiscernible over $a'_0$ and $\mathcal{I}'_0 + b_0 + \mathcal{I}'_1$ is indiscernible over $a'_1$. Hence $\mathcal{I}'_0 + b_0 + \mathcal{I}'_1$ is indiscernible over $a'_0 a'_1$ by strong $2$-distality. In particular, as $\models \varphi \left(b_{\bar{i}_0} , b_0, b_{\bar{i}_1}, b_1, b_{\bar{i}_2}, a \right)$, this implies $\models \varphi \left(b_{\bar{i}_0} , b_{i'}, b_{\bar{i}_1}, b_1, b_{\bar{i}_2}, a \right)$ for some/any $i'' < i^*$ in $I_1$, which by indiscernibility of $\mathcal{I}_0 + \mathcal{I}_1 + b_1 + \mathcal{I}_2$ over $a$ implies $\models \varphi \left(b_{\bar{i}_0} , b_{i'}, b_{\bar{i}_1}, b_{i''}, b_{\bar{i}_2}, a \right)$ for some/any $i'' < \bar{i}_2$ in $I_2$ --- contradicting the choice of $\varphi$.
 
 Similarly,  assume $T$ is $2^+$-distal, and assume that each of $\mathcal{I}_0 + b_0 + \mathcal{I}_1 + b_1 + \mathcal{I}_2 + \mathcal{I}_3$, $\mathcal{I}_0 +  \mathcal{I}_1 + b_1 + \mathcal{I}_2 + b_2 + \mathcal{I}_3$ and $\mathcal{I}_0 + b_0 + \mathcal{I}_1  + \mathcal{I}_2 + b_2 + \mathcal{I}_3$ is $\emptyset$-indiscernible. Let $a := b_2 + \mathcal{I}_3$. It follows by $2^+$-distality that $\mathcal{I}_0 + b_0 + \mathcal{I}_1 + b_1 + \mathcal{I}_2$ is indiscernible over $b_2 + \mathcal{I}_3$, and since also $\mathcal{I}_0  + \mathcal{I}_1  + \mathcal{I}_2 + b_2 + \mathcal{I}_3$ is indiscernible, this implies that $\mathcal{I}_0 + b_0 + \mathcal{I}_1 + b_1 + \mathcal{I}_2 + b_2 + \mathcal{I}_3$ is indiscernible --- so $T$ is $2$-distal.
 \end{proof}

\begin{remark}\label{rem: k,l dist implic}
	The proof easily generalizes to show that $(k,\ell)$-distality implies $(k',\ell')$-distality for any $k' \geq k$ and $\ell' < \ell$ (in any theory).
\end{remark}

 \begin{defn}\label{def: k-indisc triv}
 For $k \geq 1$, a theory $T$ is \emph{indiscernibly $k$-trivial} if for any infinite sequence $\mathcal{I}$ and tuples $(a_t : t < k + 1)$, if $\mathcal{I}$ is indiscernible over $(a_t : t \in s)$ for every $s \subseteq \{0,1,\ldots, k\}$ with $|s| = k$, then $\mathcal{I}$ is indiscernible over $(a_t : t < k + 1)$. We say that $T$ is \emph{endlessly} indiscernibly $k$-trivial if the same holds restricting to $\mathcal{I}$ indexed by an infinite linear order without the first or last element.
 \end{defn}
 
 \begin{remark}\label{rem: endless vs indisc triv}
 \begin{enumerate}
 	\item The case $k=1$ corresponds to (endless) indiscernible triviality considered in \cite{braunfeld2021characterizations}. 
 	\item Note that if $T$ is stable, then indiscernible $k$-triviality is equivalent to endless indiscernible $k$-triviality, using that every infinite $A$-indiscernible sequence is totally indiscernible over $A$.
 	\item Outside of stability, this is not the case already for dp-minimal theories. Indeed, let $T$ be the theory of infinitely branching dense trees, in the language with only the meet function $\land$ viewed as a semilattice (see e.g.~\cite[Section 3.5]{chernikov2025semi}). So the tree order is $x \leq y \iff x \land y = x$.
Let $(a_i : i \in \omega)$ be a sequence with $a_i > a_j$ in the tree order for $i < j$ and $b_1 \neq b_2$ so that $b_1 \land b_2 = a_0$.
By quantifier elimination  $(a_i)$ is indiscernible over each of $b_1$ and $b_2$ separately, but not over both. However, by \cite{parigot1982theories}, $T$ is monadically NIP, hence dp-minimal and endlessly indiscernibly $1$-trivial by \cite{braunfeld2021characterizations, braunfeld2024corrigenda}.
 \end{enumerate}
 \end{remark}

 \begin{remark}\label{rem: indisc triv def remarks}
 	\begin{enumerate}
 		\item If $T$ is indiscernibly $k$-trivial, then $T_A$ (obtained by naming an arbitrary small set of parameters $A$) is also indiscernibly $k$-trivial.
 		
 		Indeed, if $\mathcal{I}$ is indiscernible over $(a_t : t \in s) A$, let $\mathcal{I}' := \left(b_i b' : i \in I \right)$ for $b'$ a tuple enumerating $A$. Then $\mathcal{I}'$ is indiscernible over $(a_t : t \in s)$, hence over $(a_t : t < k+1)$ by indiscernible $k$-triviality, but this implies that $\mathcal{I}$ is indiscernible over $A(a_t : t < k+1)$.
 		
 		\item If $T$ is indiscernibly $k$-trivial, $m \geq k+1$, $(b_t : t < m)$ are tuples and $\mathcal{I}$ is indiscernible over $(b_t : t \in s)$ for each $s \subseteq \{0, \ldots, m -1 \}$ with $|s| = k$, then $\mathcal{I}$ is indiscernible over all $(b_t : t < m)$ simultaneously (by a simple induction). 
 	\end{enumerate}
 \end{remark}

We recall some variants of the triviality of forking in stable theories considered by Poizat \cite{goode1991some}.

\begin{defn}
	Let $T$ be a stable (or just simple) theory and $k \geq 1$. Then $T$ (or, rather, forking in $T$) is: 
	\begin{enumerate}
		\item \emph{$k$-trivial} if for any tuples $\left(a_i : i < k+2 \right)$ and a small set $A$, if every $k+1$ of the $a_i$'s form an independent set over $A$ (in the sense of forking), then $\{a_i : i < k+2\}$ is also an independent set over $A$.
		\item \emph{totally $k$-trivial} if for any tuples $a$, $\left(b_i : i < k+1 \right)$ and a set $A$,  if $a$ is independent from any $k$ of the $b_i$'s over $A$, then it is also independent from all $k+1$ of the $b_i$'s over $A$ (note that we are not requiring the $b_i$'s to be themselves independent over $A$).
	\end{enumerate}
\end{defn}

\begin{fact}\label{fac: walker $k$-triv vs $k$-dist}
	[Walker \cite{walker2023distality}] For $T$ stable, $k$-distality coincides with $(k-1)$-triviality of forking, for all $k \geq 2$.
\end{fact}

We consider a generalization to arbitrary theories:
\begin{defn}\label{def: tot k-triv forking NIP}
We say that a theory $T$	 has \emph{totally $k$-trivial forking} if given any $A$ and tuples $a_0, \ldots, a_{k+1}, b$ so that $(a_i : i \in u) \ind_{A} b$ for all $u \subseteq k+1, |u| = k$, then $a_0 \ldots a_{k+1} \ind_{A} b$. We say ``totally trivial forking'' to refer to ``totally $1$-trivial forking''.
\end{defn}

\begin{remark}
	The direction in which we require independence here is important for our later considerations to work when forking is not symmetric. Also, in DLO, if $a_0 < b < a_1$, then $b \ind a_0, b \ind a_1, a_0 \ind a_1$, but $b  \nind a_0 a_1$.
\end{remark}

\begin{problem}
	Do we get an equivalent definition if we restrict $A$ to models in the base? At least in extensible (NIP) theories?
\end{problem}

%\begin{prop}
%	If $T$ is binary, then $T$ has totally trivial forking (in the sense of Definition \ref{def: tot k-triv forking NIP}).
%\end{prop}
%\begin{proof}
%
%First assume that for each $i < k$, $(a_{i,j} : j \in I)$ is an $A$-indiscernible sequence of tuples. Then, using binarity of $T$, the sequence $(a_{0,j}, \ldots, a_{k-1,j})_{j \in I}$ is also indiscernible over $A$.
%
%This implies that forking equals dividing for formulas over arbitrary sets. Indeed, assume $\varphi(x,b) \vdash \bigvee_{i < k} \psi_i(x,a_{i})$ and each $\psi_i(x,a_{i})$ divides over $A$. 
%\end{proof}

\begin{lemma}\label{lem: tot triv implies indisc triv}
	If $T$ is an extensible resilient theory with totally $k$-trivial forking, then $T$ is endlessly indiscernibly $k$-trivial.
\end{lemma}
\begin{proof}

Assume we are given an endless sequence $\mathcal{I}_0$ and tuples $(a_t : t < k + 1)$ so that $\mathcal{I}_0 = (b_i : i \in I_0)$ is indiscernible over $(a_t : t \in s)$ for every $s \subseteq \{0,1,\ldots, k\}$ with $|s| = k$. By compactness, we can find sequences $\mathcal{I}_{-1}, \mathcal{I}_1$ indexed by $\mathbb{Q}$ so that $\mathcal{I}_{-1} + \mathcal{I}_{0} + \mathcal{I}_{1}$ is also indiscernible over $\bar{a}_s := (a_t : t \in s)$ for every $s \subseteq \{0,1,\ldots, k\}$ with $|s| = k$. 

Fix such $s$. Given any finite $J_0 \subseteq I_0$, choose an arbitrary sequence $\left(J_\ell : \ell \in \omega \right) $ with $J_\ell \subseteq I_0, |J_\ell| = J_0$ and $J_0 < J_1 < \ldots $, and let $\bar{b}_{\ell} := \left( b_i : i \in J_{\ell} \right)$. Then $\left(\bar{b}_{\ell} : \ell \in \omega \right)$ is indiscernible over $\mathcal{I}_{-1}, \mathcal{I}_1$, and for every $\ell \in \omega$ by indiscernibility (and using that $\mathcal{I}_{-1}, \mathcal{I}_{1}$ have no endpoints) we have $ \left( \bar{b}_{\ell'} : \ell' \in \omega, \ell' \neq \ell \right) \ind^{u}_{\mathcal{I}_{-1}, \mathcal{I}_1} \bar{b}_{\ell}$.
Let $p \left(x, \bar{b}_{0} \right) := \tp(\bar{a}_s / \bar{b}_0 \mathcal{I}_{-1} \mathcal{I}_{1})$. As $\left( \bar{b}_{\ell} : \ell \in \omega \right)$ is indiscernible over $\bar{a}_s \mathcal{I}_{-1} \mathcal{I}_{1}$, we have in particular that $\left\{ p \left(x, \bar{b}_{\ell} \right) : \ell \in \omega \right\}$ is consistent.
As $T$ is resilient, by Fact \ref{fac: resilient div witness} this implies that $p \left(x, \bar{b}_{0} \right)$ does not divide over $\mathcal{I}_{-1} \mathcal{I}_{1}$. As $T$ is extensible, we get $\bar{a}_s \ind_{\mathcal{I}_{-1} \mathcal{I}_{1}} \bar{b}_0$ by Fact \ref{fac: forking = div NTP2}. And as $\bar{b}_0$ lists an arbitrary finite subset of $\mathcal{I}_0$, we get $\bar{a}_s \ind_{\mathcal{I}_{-1} \mathcal{I}_{1}} \mathcal{I}_0$. As this holds for every $s \subseteq \{0,1,\ldots, k\}, |s| = k$, it follows by total $k$-triviality that $a_0 \ldots a_{k} \ind_{\mathcal{I}_{-1} \mathcal{I}_{1} } \mathcal{I}_0$.  As $\mathcal{I}_0$ is indiscernible over $\mathcal{I}_{-1} \mathcal{I}_1$, by Fact \ref{fac: forking in NIP}(3)  we conclude that $\mathcal{I}_0$ is indiscernible over $a_0 \ldots a_{k}$.
\end{proof}

\begin{remark}
Does the converse of Lemma \ref{lem: tot triv implies indisc triv} hold, at least in NIP theories? 
\end{remark}

\begin{prop}\label{prop: total triv in stable}
	If $T$ is stable, then the following are equivalent for all $k \geq 1$: 
	\begin{enumerate}
	\item $T$ is strongly $(k+1)$-distal,
	\item  $T$ is (endlessly) indiscernibly $k$-trivial,
	\item  $T$ has totally $k$-trivial forking.
	\end{enumerate}

\end{prop}
\begin{proof}

(2) $\Rightarrow$ (1). Endless indiscernible $k$-triviality implies strong $(k+1)$-distality in any theory (straightforward from the definitions).

	(1) $\Rightarrow$ (2). We give a proof for $k=1$ (which adapts to the general $k$ in a straightforward manner). Assume $T$ is strongly $2$-distal, but not endlessly indiscernibly $1$-trivial, and let $\mathcal{I}_0$ be an endless sequence which is indiscernible over $a_0$ and over $a_1$, but not over $a_0, a_1$. Let $J_0, J_1 \subseteq I_0$ with $\ell := |J_0|=|J_1|<\omega$ and $\varphi \in \mathcal{L}(\emptyset)$ be such that $\models \varphi(b_{J_0},a_0,a_1) \land \neg \varphi(b_{J_1},a_0,a_1)$.

	 We can find a sequence $\mathcal{I}_1$ indexed by $\mathbb{Q}$ so that $\mathcal{I}_0 + \mathcal{I}_1$ is indiscernible over $a_0$, $\mathcal{I}_0 + \mathcal{I}_1$ is indiscernible over $a_1$, and $\mathcal{I}_1$ is indiscernible over $a_0 a_1$  (for each finite $J \subseteq I_0$, finite sets of formulas $\Delta_0 \in \mathcal{L}(a_0), \Delta_1 \in \mathcal{L}(a_1)$, $\Delta \in \mathcal{L}(a_0a_1)$ and $n \in \omega$, by Ramsey can find $i_0 < \ldots < i_{n-1} \in I_0$ with $J < i_0$ so that $(b_{i_j} : j < n)$ is $\Delta$-indiscernible, and by the assumption on $\mathcal{I}_0$, $(b_i : i \in J) + (b_{i_j} : j < n)$ is both $\Delta_0$-indiscernible and $\Delta_1$-indiscernible --- we can then conclude by compactness). 
	
	We either have $\varphi(b_J, a_0, a_1)$ for every $J \subseteq I_1, |J|=\ell$, in which case we let $J^* := J_1$, or $\neg \varphi(b_J, a_0, a_1)$ for every $J \subseteq I_1, |J|=\ell$, in which case we let $J^* := J_0$.
	
	Let $I'_0 = I'_1 := \mathbb{Q}, I' := I'_0 + I'_1$ and choose arbitrary $(J'_i : i \in I')$ with $J'_i \subseteq I_1, |J'_i| = \ell$ and $J'_i < J'_j$ for all $i<j \in I'$.
	By stability of $T$ the sequence $\mathcal{I}_0 + \mathcal{I}_1$ is totally indiscernible over each of $a_0$ and $a_1$, hence taking $\mathcal{I}'_t := \left( b_{J'_i} :  i \in I'_t\right)$, the sequence of $\ell$-tuples $\mathcal{I}'_0 + b_{J^*} + \mathcal{I}'_1$ is indiscernible over each of $a_0$ and $a_1$, and $\mathcal{I}'_0  + \mathcal{I}'_1$ is $a_0a_1$-indiscernible. Hence $\mathcal{I}'_0 + b_{J^*} + \mathcal{I}'_1$ is $a_0a_1$-indiscernible by strong $2$-distality. But this is a contradiction, as by the choice of $J^*$ we have $\models \varphi(b_{J^*}, a_0, a_1) \not \leftrightarrow \varphi(b_{J_i}, a_0, a_1)$ for some/any $i \in I'$.
 	
	(3) $\Rightarrow$ (2). By Lemma \ref{lem: tot triv implies indisc triv} (and Remark \ref{rem: endless vs indisc triv}). We also give a quicker proof using symmetry of forking for $T$ stable. Assume that an infinite endless sequence  $\mathcal{I}_0$ is indiscernible over each of $a_0, a_1$.   
By compactness, as in the proof of (1)$\Rightarrow$(2), we can find an endless sequence $\mathcal{I}_1$ so that $\mathcal{I}_0 + \mathcal{I}_1$ is indiscernible over each of $a_1$ and $a_2$. 
Hence $\mathcal{I}_0 \ind^u_{\mathcal{I}_1} a_t$, and so $\mathcal{I}_0 \ind_{\mathcal{I}_1} a_t$, for $t \in \{0,1\}$.  But then by total triviality $\mathcal{I}_0 \ind_{\mathcal{I}_1} a_0 a_1$, hence $ a_0 a_1 \ind_{\mathcal{I}_1} \mathcal{I}_0 $ by symmetry. By Fact \ref{fac: forking in NIP}(1),(2), as $\mathcal{I}_0$ is indiscernible over $\mathcal{I}_1$, it follows that $\mathcal{I}_0$ is indiscernible over $a_0a_1$.

(2) $\Rightarrow$ (3). Again, we give a proof for $k=1$, it generalizes to arbitrary $k$ in a straightforward manner. Assume $T$ is indiscernibly trivial. Assume $b_0 \ind_A a_0, b_0 \ind_A a_1$, then also $b_0 \ind_{\acl^{\eq}(A)} a_0, b_0 \ind_{\acl^{\eq}(A)} a_1$ (by Fact \ref{fac: forking in general T}).
Let $p(x)$ be a global type extending $\tp(b_0/\acl^{\eq}(A))$ and non-forking over $\acl^{\eq}(A)$, hence invariant over $\acl^{\eq}(A)$ by Fact \ref{fac: forking in NIP}. For $1 \leq i < \omega$, let $b_i \models p|_{\acl^{\eq}(A) b_{<i} a_0 a_1}$. Then $\mathcal{I} := (b_i : 0 \leq i < \omega) \models p^{\otimes \omega}|_{\acl^{\eq}(A)}$ is indiscernible over $A$. And for each $t \in \{0,1\}$ we have $\mathcal{I} \ind_{A} a_t$. Indeed, by induction on $i \in \omega$, we show $b_{\leq i} \ind_{A} a_t$. For $i = 0$, $b_0 \ind_A a_t$ by assumption. For $i \in \omega$, we have $b_{i+1} \ind_{A} b_{\leq i} a_t $ by the choice of $b_{i+1}$, hence $b_{i+1} \ind_{A b_{\leq i}}  a_t$ by base monotonicity, and $b_{\leq i} \ind_{A} a_t$ by the inductive assumption, hence by left transitivity $b_{\leq i+1} \ind_{A} a_t$. Now by symmetry we have $a_t \ind_{A} \mathcal{I}$ and $\mathcal{I}$ is $A$-indiscernible, hence, by Fact \ref{fac: forking in NIP}(2) and (3), $\mathcal{I}$ is $a_tA$-indiscernible, for both $t \in \{0,1\}$. By indiscernible triviality (and Remark \ref{rem: indisc triv def remarks}(1)), $\mathcal{I}$ is $a_0 a_1 A$-indiscernible. By Kim's lemma (as in the proof of Lemma \ref{lem: tot triv implies indisc triv})  this implies $a_0 a_1 \ind_{A} \mathcal{I}$, in particular $a_0 a_1 \ind_{A} b_0$.
\end{proof}

\begin{remark}\label{rem: gen stable tot triv}
	The proof in particular shows that if $T$ is strongly $2$-distal, $p \in S_x(\cU)$ is generically stable over $A$, $b \models p|_A$ and $a_1,a_2$ are arbitrary with $b \ind_{A} a_i$ (which is equivalent to $a_i \ind_{A} b$ in this case), then $b \ind_{A} a_1 a_2$ --- so we get total triviality of forking for realizations of generically stable types.
\end{remark}

\begin{problem}

Does the implications (2)$\Rightarrow$(3) in Proposition \ref{prop: total triv in stable} generalize to (extensible) NIP theories (with respect to Definition \ref{def: tot k-triv forking NIP})? 
%***
%In simple theories (maybe with enough amalgamation) might work if we allow to move indiscernible sequences in the definition of indiscernible triviality?
%***
\end{problem}

Next we show that the implication (1)$\Rightarrow$(2) can fail badly for unstable NIP theories. The following is  \cite[Proposition 4.3]{braunfeld2021characterizations} for $k = 1$, and we show that it generalizes to arbitrary $k$:
\begin{prop}\label{prop: fin dprk group not indisc k-triv}
	Assume $T$ has finite $\Dp$-rank and satisfies endless indiscernible $k$-triviality for some $1 \leq k < \omega$. Then $T$ does not define an infinite group.
\end{prop}
\begin{proof}
Recall that if $\Dp(b/A) \leq m$ and $\mathcal{I}_0, \ldots, \mathcal{I}_{n-1}$ for $m < n \in \omega$ are infinite mutually indiscernible over $A$ sequences of tuples, then there exists some $u \subseteq n$ with $|n \setminus u| \geq n - m$ so that $\left( \mathcal{I}_i : i \in u \right)$ are mutually indiscernible over $A b$ \cite[Proposition 4.4]{kaplan2013additivity}. In particular this implies that $\Dp$-rank is subadditive: for any $a,b$ and $A$, $\Dp(a,b/A) \leq \Dp(a/A) + \Dp(b/Aa)$.

	Assume now $G = G(x)$ is an infinite $A$-definable group, and let $m := \Dp(G(x))$. We show that $T$ is not endlessly indiscernibly $2$-trivial (the general $k$ is similar). Let $n := m + (mk + 1)$. We can choose endless mutually indiscernible sequences $\mathcal{I}_0, \ldots, \mathcal{I}_{n-1}$ of elements of $G$ (by Ramsey and compactness, e.g.~choose a single endless indiscernible sequence using that $G$ is infinite, and cut it into pieces). Let $a_i$ be an arbitrary element of $\mathcal{I}_i$ and $b := a_0 \cdot \ldots \cdot a_{n-1} \in G$. As $\Dp(b/A) \leq m$, by the above and permuting the order of the sequences if necessary, we may assume that $(\mathcal{I}_i : i < mk+1)$ are mutually indiscernible over $bA$. By subadditivity of $\Dp$-rank we have $\Dp(a_{mk+1}, \ldots, a_{mk + m}/A) \leq mk$, hence also $\Dp(a_{mk+1}, \ldots, a_{mk + m}/bA) \leq mk$. Applying the previous paragraph again, we thus have (without loss of generality) that $\mathcal{I}_0$ is indiscernible over $A, b, a_{mk+1}, \ldots, a_{mk + m}$. Combining, we have that $\mathcal{I}_0$ is indiscernible over:
	\begin{itemize}
		\item $b; a_{mk+1}, \ldots, a_{mk + m}; A$,
		\item $b; a_1, \ldots, a_{mk}; A$ (as $(\mathcal{I}_i : i < mk+1)$ are mutually indiscernible over $bA$),
		\item $a_1, \ldots, a_{mk}; a_{mk+1}, \ldots, a_{mk + m}; A$ (as $\mathcal{I}_0, \ldots, \mathcal{I}_{n-1}$ are mutually indiscernible over $A$).
	\end{itemize}
But $\mathcal{I}_0$ is not indiscernible over $b; a_1, \ldots, a_{mk}; a_{mk+1}, \ldots, a_{mk + m}; A$ as $a_0$ is definable over this tuple.
\end{proof}

\begin{remark}
	 For example, let $T = \RCF$. Then $T$ is a dp-minimal, extensible theory which is strongly $1$-distal (so in particular strongly $2$-distal). However, it fails  endless indiscernible $k$-triviality for all $k \in \omega$ by 	Proposition \ref{prop: fin dprk group not indisc k-triv}.
\end{remark}

We summarize the results in the literature about the distality hierarchy:
\begin{fact}
	\begin{enumerate}
		\item  $T$ is $1$-distal if and only if it is strongly $1$-distal (right to left is clear; left to right: by Proposition \ref{prop: n-dist implies n-dep}, if $T$ is $1$-distal then it is NIP, hence  it is also strongly $1$-distal by the equivalence of the ``internal'' and ``external'' characterizations of distality from \cite{simon2013distal}). 
		\item No stable theory is $1$-distal \cite{simon2013distal}. There exist superstable  rank $1$ strongly $2$-distal theories that are not $1$-distal (e.g.~the theory of equality); or not $k$-distal for any $k$ (e.g.~$T = \ACF$, see \cite[Section 5]{walker2023distality}).
		\item If $T$ is superstable and $k$-distal for some $k \geq 2$, then it is already $2$-distal (\cite[Proposition 8.11]{walker2023distality}, combining Fact \ref{fac: walker $k$-triv vs $k$-dist} and  Fact \ref{fac: Poizat triv}).
		\item For every $k \in \omega$, there exist (not NIP) theories $T$ which are strongly $k$-distal, but not $k$-distal \cite[Section 5]{walker2023distality}.
	\end{enumerate}
\end{fact}

We recall some results of Poizat concerning $k$-triviality and total $k$-triviality of forking in stable theories:
\begin{fact}\label{fac: Poizat triv} The following hold for stable $T$ and any $k \in \omega$.
	\begin{enumerate}
		\item \cite[Proposition 3]{goode1991some} A superstable $k$-trivial theory is trivial.
		\item \cite[Proposition 5]{goode1991some} A trivial superstable theory with finite $U$-rank is totally trivial. More generally, a trivial superstable theory with $U$-rank strictly bounded by $\omega^{n+1}$ is $2^{n}$-totally trivial.
		\item \cite[Example after Proposition 5]{goode1991some} For each $n < \omega$ there exists a superstable trivial theory $T_n$ (of $U$-rank $\omega^{n}$) which is exactly $2^{n}$-totally trivial (i.e~$2^n$-totally trivial, but not $(2^{n}-1)$-totally trivial). There exists a superstable theory $T_{\omega}$ (of rank $\omega^{\omega}$) which is trivial, but not $n$-totally trivial for any $n$.
		\item (see \cite[Proposition 8]{goode1991some} and the ``Added in proof'' section) For a  $1$-based theory, all notions of triviality are equivalent (trivial, $k$-trivial, totally trivial, $k$-totally trivial).
	\end{enumerate}
\end{fact}

%\begin{problem}
%	Does this generalize to extensible NIP? In an NIP theory, if $a_0 \ind b, a_1 \ind b$, then can find a sequences containing $b$  and indiscernible over each of $a_0$ and $a_1$? Can in stable. *** Assume $b \ind a_0, b \ind a_1$. If $\tp(b/A)$ has a global strictly invariant extension, generate a MS over $a_0a_1A$... If strict independence has some poor mans transitivity, maybe can get $I \ind^{\st}_A a_t$, so $a \ind^i_A I$, so $I$ is $a$-indiscernible ***
%\end{problem}

Combining Fact \ref{fac: Poizat triv} with Proposition \ref{prop: total triv in stable} (and Fact \ref{fac: walker $k$-triv vs $k$-dist}), we can thus answer a question of Walker \cite[Section 9]{walker2023distality}:
\begin{cor}\label{cor: sep for strong $n$-dist}
	 For every $n\geq 1$, there exists a superstable $2$-distal $T$ which is strongly $2^{(n+1)}$-distal but not $2^{n}$-strongly distal. There exists a $2$-distal superstable theory which is not strongly $n$-distal for any $n < \omega$. For a $1$-based stable theory and any $k \geq 2$, $k$-distality implies strong $2$-distality.
\end{cor}

\begin{prop}\label{prop: 2-dist implies 2+-dist}
	If $T$ is stable, then for any $k \geq 2$,  $T$ is $(k,0)$-distal if and only if $T$ is $(k,1)$-distal.
\end{prop}
\begin{proof}
	We show the case $k=2$ (general $k$ is similar). Assume that $T$ is stable and $2$-distal (hence $1$-trivial by Fact \ref{fac: walker $k$-triv vs $k$-dist}), and we are given  an indiscernible sequence $\mathcal{I}_0 + b_0 + \mathcal{I}_1 + b_1 + \mathcal{I}_2$ so that each of $\mathcal{I}_0 + b_0 + \mathcal{I}_1 +\mathcal{I}_2, \mathcal{I}_0 + \mathcal{I}_1 + b_1 + \mathcal{I}_2$  is indiscernible over $a$ (where $\mathcal{I}_t = (b_i : i \in I_t)$ are endless). It follows by indiscernibility that $\mathcal{I}_0 b_0 \ind^u_{\mathcal{I}_2} a$ and $\mathcal{I}_1 b_1 \ind^{u}_{\mathcal{I}_2} a$, and $\mathcal{I}_1 b_1 \ind^{u}_{\mathcal{I}_2} \mathcal{I}_0 b_0 $. Hence, by $1$-triviality and symmetry, $a \ind_{\mathcal{I}_2} \mathcal{I}_0 b_0 \mathcal{I}_1 b_1$. By assumption $\mathcal{I}_0 + b_0 + \mathcal{I}_1 + b_1$ is indiscernible over $\mathcal{I}_2$, hence, by Fact \ref{fac: forking in NIP}, $\mathcal{I}_0 + b_0 + \mathcal{I}_1 + b_1$ is indiscernible over $a \mathcal{I}_2$, which implies that $\mathcal{I}_0 + b_0 + \mathcal{I}_1 + b_1 + \mathcal{I}_2$ is indiscernible over $a$. Indeed, for any $\bar{i}_t$ tuple from $I_t$ and formula $\varphi \in L(a)$, we have $\models \varphi(b_{\bar{i}_0}, b_0,  b_{\bar{i}_1}, b_1, b_{\bar{i}_2} )\Leftrightarrow  \ \models \varphi(b_{\bar{i}_0}, b_{i_0},  b_{\bar{i}'_1}, b_{i_1}, b_{\bar{i}_2})$ for some/any $\bar{i}_0 < i_0 < \bar{i}'_1 < i_1$ in $I_0$ (as $\mathcal{I}_0 + b_0 + \mathcal{I}_1 + b_1$ is indiscernible over $a \mathcal{I}_2$), $\Leftrightarrow \  \models \varphi(b_{\bar{i}_0}, b_{i_0},  b_{\bar{i}'_1}, b_{i_1}, b_{\bar{i}'_2})$ for some/any $ i_1 < \bar{i}'_{2}$ in $I_0$ (as $\mathcal{I}_0 + \mathcal{I}_2$ is indiscernible over $a$). The other direction is by Proposition \ref{prop: refinements of n-distality} (and Remark \ref{rem: k,l dist implic}).
\end{proof}

%\begin{remark}
%	In fact, NIP is sufficient. 
%	
%	Assume we have an $\emptyset$-indiscernible sequence $\mathcal{I}_0 + b_0 + \mathcal{I}_1 + b_1 + \mathcal{I}_2$ not indiscernible over $a$, but so that each of $\mathcal{I}_0 + b_0 + \mathcal{I}_1 +\mathcal{I}_2, \mathcal{I}_0 + \mathcal{I}_1 + b_1 + \mathcal{I}_2$  is indiscernible over $a$. 
%	Then there exist $\varphi \in L(\emptyset)$ and, for each $t \in 3$, finite tuples $\bar{i}_t, \bar{i}'_t$  of the same length from $I_t$ so that 
%	$$\models \varphi(a; b_{\bar{i}_0}, b_0, b_{\bar{i}_1}, b_1, b_{\bar{i}_3}) \land \neg \varphi(a; b_{\bar{i}'_0}, b_0, b_{\bar{i}'_1}, b_1, b_{\bar{i}'_3}).$$
%	
%	Let $\psi \in L(\emptyset)$ be as given by UDTFS for $\varphi$. 
%	
%	Then 
%	$$ \models \psi(; b_{\bar{i}_0}, b_0, b_{\bar{i}_1}, b_1, b_{\bar{i}_3}) \land \neg \psi()$$ 
%\end{remark}

\begin{problem}
	Does Proposition \ref{prop: 2-dist implies 2+-dist} hold only assuming that $T$ is NIP? Even assuming stability, for $k \geq 3$, does $(k,\ell)$-distality form a strict hierarchy when varying $1 \leq \ell < k$?
\end{problem}

\section{Invariant generically stable measures in $n$-distal groups}\label{sec: Invariant generically stable measures in n-distal groups}

\begin{fact}\cite{simon2013distal}\label{fac: distal implies gen stab is smooth}
	If $T$ is $1$-distal then every generically stable measure $\mu \in \mathfrak{M}_x(\cU)$ is smooth.
\end{fact}

This fails badly in stable $2$-distal theories:
\begin{expl}
	The theory of an infinite set with no additional structure is $2$-ary by quantifier elimination, hence strongly $2$-distal. The unique non-algebraic type is generically stable, but not smooth (viewed as a measure).
\end{expl}

Recall that if $T$ is NIP and $G$ is an  $\emptyset$-type-definable group, then taking $G^{00} \leq G$ to be the intersection of all type-definable (over an arbitrary small subset of $\cU$) subgroups $H \leq G$ of bounded index, we have that $G^{00}$ is type-definable over $\emptyset$ \cite{shelah2008minimal}. Moreover, equipped with the logic topology, $G/G^{00}$ is a compact Hausdorff topological group, hence it is equipped with (left-invariant Borel probability) Haar measure \cite{pillay2004type}. We let $\pi_0 : G \to G/G^{00}$ be the canonical projection homomorphism. The following is an unpublished result of Hrushovski, Macpherson and Pillay (see \cite[Theorem 8.37]{simon2015guide}):
\begin{fact}\label{fac: comp dom iff smooth measure}
	 ($T$ is NIP) A definable group  $G$ admits a smooth (left-)$G$-invariant measure $\mu \in \mathfrak{M}_{G}(\cU)$ if and only if $G$ satisfies \emph{compact domination}, i.e.~for every $\cU$-definable set $D$, the set $\{a \in G/G^{00} : \pi^{-1}_0(a) \cap D \neq \emptyset \textrm{ and } \pi^{-1}_0(a) \cap \neg D \neq \emptyset \}$ has $h$-measure $0$.
\end{fact}

\begin{cor}
	If $T$ is distal and $G$ is an fsg group, then $G$ is compactly dominated.
\end{cor}
\noindent This has a combinatorial consequence, an arithmetic version of the distal regularity lemma \cite[Section 6]{conant2021structure}. We conjecture that these results hold already in $k$-distal theories for an arbitrary $k$:

\begin{conjec}\label{conj: k-dist G-inv meas}
Assume $2 \leq k \in \omega$, $T$ is $k$-distal NIP and $G$ a definable group. If  $\mu \in \mathfrak{M}_{G}(\cU)$ is generically stable and $G$-invariant, then $\mu$ is smooth. In particular, if $G$ is fsg, then it is compactly dominated. More generally,  we can ask the same question for \emph{fim} groups in (not necessarily NIP) $k$-distal theories (see \cite{chernikov2024definable}).
\end{conjec}

Using determinacy for types (Fact \ref{fac: det for types}), we can show that Conjecture \ref{conj: k-dist G-inv meas}  holds in the case of types:

\begin{prop}\label{prop: stable groups are not $n$-distal}
\begin{enumerate}
	\item If in $T$ there is a type-definable group $G$ with a non-algebraic (left-) $G$-invariant type $p \in S_G(\mathbb{M})$ so that $p^{\otimes k}$ is generically stable for all $1 \leq k \in \omega$ (so e.g.~$p$ is generically stable and $T$ is $\NTP_2$, see \cite[Section 4]{conant2025generic}), then $T$ is not $k$-distal for any $1 \leq k \in \omega$.
	\item In particular, if $T$ is stable and $G$ is an infinite type-definable group, then $T$ is not $k$-distal for any $1 \leq k\in \omega$.
\end{enumerate}

\end{prop}
\begin{proof}

(1) Assume $M \prec \cU$ is a small model so that $G$ is type-definable over $M$ and $p^{\otimes k}$ is generically stable over $M$ for all $k$. Let $1 \leq k \in \omega$ be arbitrary. Consider the $M$-definable partial function $f: (x_1, \ldots, x_k) \mapsto (x_1, \ldots, x_k, x_1 \cdot \ldots \cdot x_k)$, and the push-forward type $q(x_1, \ldots, x_k, y) := f_{\ast}\left(p^{\otimes k}_{x_1, \ldots, x_k} \right)$. Then $q$ is also generically stable over $M$ (see e.g.~\cite[Section 3.5]{chernikov2024definable}). Of course $q \restriction_{x_1, \ldots, x_k} = p^{\otimes k}_{x_1, \ldots, x_k}$, and we claim that for any $i \in [k]$, $q \restriction_{x_{\neq i}, y} = p^{\otimes k}(x_{\neq i}, y)$. Indeed, fix $i \in [k]$, let $M \preceq N \prec \cU$ be arbitrary, and let $(a_1, \ldots, a_k) \models p^{\otimes k}|_{N}$. By generic stability $p$ commutes with itself, hence also $(a_{\neq i}, a_i) \models p^{\otimes k}|_{N}$, so $a_i \models p|_{N, a_{\neq i}}$. In particular $a_i \models p|_{N, a_{\neq i}, a_1 \cdot \ldots \cdot a_{i-1}}$.
By left $G$-invariance of $p$  we then have $ (a_1 \cdot \ldots \cdot a_{i-1}) \cdot a_{i} \models p|_{N, a_{\neq i}}$, in particular  $ a_1 \cdot \ldots \cdot a_{i-1} \cdot a_{i} \models p|_{N, a_{\neq i}, a_{i+1} \cdot \ldots \cdot a_k}$. As $p$ is also right-$G$-invariant (see \cite[Lemma 1]{pillay2011generic}), we get  $ (a_1 \cdot \ldots \cdot a_{i-1} \cdot a_{i}) \cdot (a_{i+1} \cdot \ldots \cdot a_k) \models p|_{N, a_{\neq i}}$. Hence $(a_{\neq i}, a_1 \cdot \ldots \cdot a_k) \models p^{\otimes k}|_{N}$, as wanted.

However, if $p$ is not realized in $\cU$, we have $q \neq p^{\otimes (k+1)}$ --- contradicting $k$-determinacy of $p^{\otimes (k+1)}$ in $T$. Indeed, $x_1 \cdot \ldots \cdot x_k = y \in q$ by definition of $q$. But if $x_1 \cdot \ldots \cdot x_k = y \in p^{\otimes (k+1)}$, for $(a_1, \ldots, a_{k}) \models p^{\otimes (k+1)}|_{M}$ we must have $a_1 \cdot \ldots \cdot a_k = y \in p_y$ and $a_1 \cdot \ldots \cdot a_k \in G(\cU)$, hence $p$ must be realized in $\cU$.

(2) 	Let $T$ be stable and $G$ an infinite type-definable group. Let $G^0 = G^{00}$ be the intersection of all definable subgroups of finite index, it is a type-definable subgroup of $G$ and is still infinite (by compactness), and there is principal generic type $p \in S_{G^0}(\cU)$ of $G$ which is $G^0$-invariant. As $G$ is infinite, $p$ has to be non-algebraic (every $\varphi(x) \in p$ is generic, that is finitely many $G$-translates of $\varphi(\cU)$ cover $G$, hence $\varphi(\cU)$ is infinite).  As $p$ is generically stable (and $p^{\otimes k}$ is generically stable for all $k<\omega$, by stability of $T$, this contradicts (1).
\end{proof}

\begin{remark}
	We note that if $\mu \in \mathfrak{M}_{G}(\cU)$ is a $G$-invariant generically stable measure, a similar argument with pushforwards of measures (see \cite[Proposition 3.37, Lemma 3.44]{chernikov2024definable}) shows that for 
	$\nu(x_1, \ldots, x_{k}, x_{k+1}) := f_{\ast}\left(\mu^{\otimes k}_{x_1, \ldots, x_{k}} \right)$, we have $\nu|_{(x_i : i \in u)} = \mu^{\otimes k}(x_i : i \in u)$ for all $u \subseteq k+1, |u|=k$. And $\nu(x_1 \cdot \ldots \cdot x_k = x_{k+1}) = 1$, while $\mu^{\otimes (k+1)}(x_1 \cdot \ldots \cdot x_k = x_{k+1}) = 1$. However, to contradict $k$-determinacy for measures, we would need to know more: that $\nu|_{\mathcal{L}^{k}_{x_1, \ldots, x_{k+1}}} = \mu^{\otimes (k+1)}|_{\mathcal{L}^{k}_{x_1, \ldots, x_{k+1}}}$.
\end{remark}

We now prove Conjecture \ref{conj: k-dist G-inv meas} for measures in \emph{strongly} $n$-distal NIP theories.
\begin{theorem}\label{thm: G-inv gen stab mu is smooth}
 Assume $1 \leq n \in \omega$ and $T$ is NIP and strongly $n$-distal. Assume $G$ is a definable group and $\mu \in \mathfrak{M}_{G}(\cU)$ is generically stable and $G$-invariant. Then $\mu$ is smooth.
\end{theorem}
\begin{proof}
Let $\M$ be a small model so that $G$ is definable over $\M$ and $\mu$ is generically stable over $\M$.  As $T$ is NIP and $G$ is fsg, we know that $\mu$ is the unique left $G$-invariant measure in $\mathfrak{M}_{G}(\cU)$, and it is also the unique right $G$-invariant one (see \cite[Theorem 4.3]{hrushovski2013generically}).

As $T$ is NIP, using Lemma \ref{existence of smooth extensions of products}, let $\mu' \in \mathfrak{M}(\cU)$ be a global measure extending $\mu|_{\M}$ and smooth (i.e.~$1$-smooth) over a small model $\Ncal \succeq \M$; in particular $\mu' \in \mathfrak{M}_{G}(\cU)$. We have $\mu' \neq \mu$ (as otherwise we would be done), so let $\varphi(x) \in \Lcal(\cU)$ be so that $\mu(\varphi(x)) \neq \mu'(\varphi(x))$.

Without loss of generality $n \geq 2$, as otherwise we are done by Fact \ref{fac: distal implies gen stab is smooth}. Consider the $\M$-definable map $g: (x_0, \ldots, x_{n-1}) \mapsto (x_0, \ldots, x_{n-2}, x_0  \cdot \ldots \cdot x_{n-1})$ and let $\omega(x_0, \ldots, x_{n-1}) := g_{\ast} \left( \mu_{x_0} \otimes \ldots \otimes \mu_{x_{n-2}} \otimes \mu'_{x_{n-1}}\right) \in \mathfrak{M}_{x_0, \ldots, x_{n-1}}(\cU)$ be the pushforward measure (note that all of these measures are invariant and generically stable over $\Ncal$, so the $\otimes$-product is well-defined and does not depend on the order by Fact \ref{fact: gen stable measures}; we refer to \cite[Section 3.5]{chernikov2024definable} for a general discussion of definable pushforwards). That is, for every $\psi(x_0, \ldots, x_{n-1}) \in \Lcal(\cU)$ we have $\omega(\psi(x_0, \ldots, x_{n-1})) = \mu_{y_0} \otimes \ldots \otimes \mu_{y_{n-2}} \otimes \mu'_{y_{n-1}}(\psi(y_0, \ldots, y_{n-2}, y_0 \cdot \ldots \cdot y_{n-1}))$.
	
	\medskip
	
	 \noindent  \textbf{Claim 1.} For every $j < n$ we have that $\omega|_{(x_i : i \neq j)} = \bigotimes_{i \neq j}\mu_i(x_i)$. 

   \noindent \textbf{Proof of Claim 1.} This is obvious for $j = n-1$, so assume $0 \leq j \leq n-2$ and let $\psi(x_0, \ldots, x_{j-1}, x_{j+1}, \ldots, x_{n-1}) \in \Lcal(\cU)$ be arbitrary. Take a small model $\M'$ containing $\Ncal$ and the parameters of $\psi$. Then
	\begin{gather*}
		\omega(\psi(x_0, \ldots, x_{j-1}, x_{j+1}, \ldots, x_{n-1})) \\ 
		= \mu_{y_0} \otimes \ldots \otimes \mu_{y_{n-2}} \otimes \mu'_{y_{n-1}}(\psi(y_0, \ldots, y_{j-1}, y_{j+1}, \ldots, y_{n-2}, y_0 \cdot \ldots \cdot y_{n-1})) =\\
		 \int_{q \in S_{y_{n-1}}(\M')} \mu_{y_0} \otimes \ldots \otimes \mu_{y_{n-2}}(\psi(y_0, \ldots, y_{j-1}, y_{j+1}, \ldots, y_{n-2}, y_0 \cdot \ldots \cdot  y_{n-2} \cdot a_{n-1})) d \mu'(q)
	\end{gather*}
	for $a_{n-1} \models q$ in $\cU$. For a fixed $a_{n-1}$, taking $\M'' \supseteq \M' a_{n-1}$, $I := \{0, \ldots, n-2\} \setminus \{j\}$ and using that the measures commute we then have
	\begin{gather*}
		\mu_{y_0} \otimes \ldots \otimes \mu_{y_{n-2}}(\psi(y_0, \ldots, y_{j-1}, y_{j+1}, \ldots, y_{n-2}, y_0 \cdot \ldots \cdot  y_{n-2} \cdot a_{n-1})) = \\
		\mu_{y_j} \otimes \left( \bigotimes_{i \in I} \mu_{y_{i}}  \right) (\psi(y_0, \ldots, y_{j-1}, y_{j+1}, \ldots, y_{n-2}, y_0 \cdot \ldots \cdot  y_{n-2} \cdot a_{n-1})) =\\
		\int_{r \in S_{(y_i : i \in I)}(\M'')} \mu_{y_j} ( \psi(a_0, \ldots, a_{j-1}, a_{j+1}, \ldots, a_{n-2},\\
		 a_0 \cdot \ldots \cdot a_{j-1} \cdot y_j \cdot a_{j+1} \cdot \ldots \cdot   a_{n-2} \cdot a_{n-1}) ) d \left( \bigotimes_{i \in I} \mu_{y_{i}}  \right)(r) 
	\end{gather*}
for $(a_i : i \in I) \models r$	in $\cU$. As $\mu$ is both left and right $G$-invariant, this is equal to 
\begin{gather*}
	 = \int_{r \in S_{(y_i : i \in I)}(\M'')} \mu_{y_j} ( \psi(a_0, \ldots, a_{j-1}, a_{j+1}, \ldots, a_{n-2}, y_j ) ) d \left( \bigotimes_{i \in I} \mu_{y_{i}}  \right)(r) = \\
	\left( \left( \bigotimes_{i  \in I } \mu_{y_i} \right) \otimes \mu_{y_j}  \right)( \psi(y_0, \ldots, y_{j-1}, y_{j+1}, \ldots, y_{n-2}, y_j ) ).
\end{gather*}
Plugging this into the first integral and renaming the variables we get 
\begin{gather*}
	\omega(\psi(x_0, \ldots, x_{j-1}, x_{j+1}, \ldots, x_{n-1})) =\\
	\left( \left( \bigotimes_{i  \in I } \mu_{y_i} \right) \otimes \mu_{y_j}  \right) ( \psi(y_0, \ldots, y_{j-1}, y_{j+1}, \ldots, y_{n-2}, y_j )\\
	= \mu_{x_0} \otimes \ldots \otimes \mu_{x_{j-1}} \otimes \mu_{x_{j+1}} \otimes \ldots \otimes \mu_{x_{n-1}}(\psi(x_0, \ldots, x_{j-1}, x_{j+1}, \ldots, x_{n-1})).
\end{gather*}
 \hfill$\blacksquare_{\textit{Claim }1}$
	
	\medskip
	
 \noindent  \textbf{Claim 2.}  $\omega|_{\M} = \left(\bigotimes_{i=1}^{n-1}\mu_i(x_i) \right)|_{\M}$. 

   \noindent \textbf{Proof of Claim 2.}
	Let $\psi(x_0, \ldots, x_{n-1}) \in \Lcal(\M)$ be arbitrary. As $\mu'|_{\M} = \mu|_{\M}$,  we also have the equality $\widetilde{\mu'|_{\M}} = \widetilde{\mu|_{\M}}$ of their extensions to regular Borel measures on $S_{x}(\M)$ by uniqueness. Hence, using that $\mu$ is generically stable over $\M$ and left $G$-invariant,    
	\begin{gather*}
		\omega(\psi(x_0, \ldots, x_{n-1})) = \mu_{y_0} \otimes \ldots \otimes \mu_{y_{n-2}} \otimes \mu'_{y_{n-1}}(\psi(y_0, \ldots, y_{n-2}, y_0 \cdot \ldots \cdot y_{n-1})) =\\
		\int_{q \in S_{y_{n-1}}(\M)}(\mu_{y_0} \otimes \ldots \otimes \mu_{y_{n-2}})(\psi(y_0, \ldots, y_{n-2}, y_0 \cdot \ldots \cdot  y_{n-2} \cdot a_{n-1})) d\widetilde{\mu'|_{\M}}(q) =\\
		\int_{q \in S_{y_{n-1}}(\M)}(\mu_{y_0} \otimes \ldots \otimes \mu_{y_{n-2}})(\psi(y_0, \ldots, y_{n-2}, y_0 \cdot \ldots \cdot  y_{n-2} \cdot a_{n-1})) d\widetilde{\mu|_{\M}}(q)= \\
		\mu_{y_0} \otimes \ldots \otimes \mu_{y_{n-2}} \otimes \mu_{y_{n-1}}(\psi(y_0, \ldots, y_{n-2}, y_0 \cdot \ldots \cdot y_{n-2} \cdot y_{n-1})) = \\
		\int_{r \in S_{x_0, \ldots, x_{n-2}}(\M)} \mu_{y_{n-1}}(\psi(a_0, \ldots, a_{n-2}, a_0 \cdot \ldots \cdot a_{n-2} \cdot y_{n-1})) d \mu_{y_0} \otimes \ldots \otimes \mu_{y_{n-2}}(r) = \\
		\int_{r \in S_{x_0, \ldots, x_{n-2}}(\M)} \mu_{y_{n-1}}(\psi(a_0, \ldots, a_{n-2},  y_{n-1})) d \mu_{y_0} \otimes \ldots \otimes \mu_{y_{n-2}}(r) = \\
		\mu_{y_0} \otimes \ldots \otimes \mu_{y_{n-1}}(\psi(x_0, \ldots, x_{n-1})),
	\end{gather*}
where $a_{n-1} \models q$ and $(a_0, \ldots, a_{n-2}) \models r$ in $\cU$.
	 \hfill$\blacksquare_{\textit{Claim }2}$
	 
	 \medskip 
	 
	Now consider the formula $\theta(x_0, \ldots, x_{n-1}) := \varphi( x^{-1}_{n-2} \cdot  \ldots \cdot  x^{-1}_1 \cdot x^{-1}_0 \cdot x_{n-1}) \in \Lcal(\cU)$. On the one hand we have 
	\begin{gather*}
		\omega(\theta(x_0, \ldots, x_{n-1}))) = \omega(\varphi( x^{-1}_{n-2} \cdot  \ldots \cdot  x^{-1}_1 \cdot x^{-1}_0 \cdot x_{n-1})) =\\
		\mu_{y_0} \otimes \ldots \otimes \mu_{y_{n-2}} \otimes \mu'_{y_{n-1}}(\varphi(y^{-1}_{n-2} \cdot  \ldots \cdot  y^{-1}_1 \cdot y^{-1}_0 \cdot (y_0 \cdot y_1 \cdot \ldots \cdot y_{n-1} ))) =\\
		\mu_{y_0} \otimes \ldots \otimes \mu_{y_{n-2}} \otimes \mu'_{y_{n-1}}(\varphi(y_{n-1} )) = \mu'_{y_{n-1}}(\varphi(y_{n-1} )).
	\end{gather*}
	\noindent On the other hand, letting $\M'$ be a small model containing $\M$ and the parameters of $\theta$,
		\begin{gather*}
		\mu_{x_0} \otimes \ldots \otimes \mu_{x_{n-1}}(\theta(x_0, \ldots, x_{n-1}))) = \mu_{x_0} \otimes \ldots \otimes \mu_{x_{n-1}}(\varphi( x^{-1}_{n-2} \cdot  \ldots \cdot  x^{-1}_1 \cdot x^{-1}_0 \cdot x_{n-1}))\\
		= \int_{q \in S_{x_0, \ldots, x_{n-2}}(\M') } \mu_{x_{n-1}} \left( \varphi( a^{-1}_{n-2} \cdot  \ldots \cdot  a^{-1}_1 \cdot a^{-1}_0 \cdot x_{n-1}) \right) d (\mu_{x_0} \otimes \ldots \otimes \mu_{x_{n-2}})(q)
	\end{gather*}
	for $(a_0, \ldots, a_{n-2}) \models q$ in $\cU$, which by left $G$-invariance of $\mu$ is equal to
	\begin{gather*}
			= \int_{q \in S_{x_0, \ldots, x_{n-2}}(\M') } \mu_{x_{n-1}} \left( \varphi( x_{n-1}) \right) d (\mu_{x_0} \otimes \ldots \otimes \mu_{x_{n-2}})(q) = \mu_{x_{n-1}} \left( \varphi( x_{n-1}) \right).
	\end{gather*}
	\noindent By the choice of $\varphi(x)$ we thus conclude that $\omega \neq \mu_{x_0} \otimes \ldots \otimes \mu_{x_{n-1}}$. Combined with Claims 1 and 2, this shows that $\mu_{x_0} \otimes \ldots \otimes \mu_{x_{n-1}}$ is not $n$-smooth over $\M$, contradicting strong $n$-distality by Proposition \ref{2-determinacy with a type}.
\end{proof}
\begin{cor}
Hence compact domination holds for definable fsg groups in strongly $n$-distal  NIP theories by Fact \ref{fac: comp dom iff smooth measure}, and so the  arithmetic version of the distal regularity lemma holds for definable groups exactly as in \cite[Section 6]{conant2021structure}	
\end{cor}

It is conjectured in \cite[Conjecture 1.1]{chernikov2021n} (and also  \cite[Problem 4.10]{chernikov2019mekler}) that every $n$-dependent field (in the ring language) is already $1$-dependent. This and related questions are studied further in \cite{chernikov2024n}. 
Here we conjecture an analog for $n$-distality (one can also consider variants for strong $n$-distality, or for valued fields, possibly assuming NIP first):
\begin{conjec}
	Every  $n$-distal field (in the ring language) is already $1$-distal. In fact, we can even ask: are there (strongly, NIP) $(n+1)$-distal, not $n$-distal groups for all $n$ (in the pure group language)?
\end{conjec}

Distal (valued) fields are studied in \cite{aschenbrenner2022distality}, in particular distal Henselian valued fields are classified modulo Shelah's conjecture.  First examples of strictly $n$-dependent pure groups are given in \cite{chernikov2019mekler} using Mekler's construction (but note that the groups produced using it always interpret a (stably embedded) infinite stable group, hence cannot be $n$-distal for any $n$ by Proposition \ref{prop: stable groups are not $n$-distal}). 
It is observed in \cite{hieronymi2017distal} that dense pairs of $o$-minimal structures need not be $1$-distal. Are they $2$-distal?

\section{Infinite strongly $k$-distal field have characteristic $0$}\label{sec: Infinite strongly k-distal field have characteristic 0}

In this section we generalize \cite[Corollary 6.3]{chernikov2018regularity}, which shows that no theory satisfying strong Erd\H{o}s-Hajnal  (= $1$-strong Erd\H{o}s-Hajnal) can define an infinite field of positive characteristic, to $k$-strong Erd\H{o}s-Hajnal for arbitrary $k$ (see Corollary \ref{cor: n-str EH} and Definition \ref{def: k-sEH for finite}).

\subsection{Strong discrepancy and dense fibers on cylinder intersections}

For a function $f:Z\to B$ on a finite set $Z$ and a subset $S\subseteq Z$, Babai--Hayes--Kimmel \cite{babai1998cost} and subsequent work in multiparty communication complexity consider the following ``strong discrepancy'' measure:

\begin{defn}\cite[Definition 2.1]{babai1998cost}\label{def:strong-disc}
Let $f:Z\to B$ be a function between finite sets and let $S\subseteq Z$. The \emph{strong discrepancy} of $f$ on $S$ is
\begin{equation}
\Gamma(f,S)
\;:=\;
\max_{y\in B}\ \frac{1}{|Z|}\,\Bigl|\ |f^{-1}(y)\cap S|-\frac{|S|}{|B|}\ \Bigr| = \max_{y\in B}\ \frac{|S|}{|Z|}\,\Bigl|\ \Pr_{z\sim S}\bigl[f(z)=y\bigr]-\frac{1}{|B|}\ \Bigr|.
\end{equation}
where $z\sim S$ denotes a uniformly random element of $S$.
\end{defn}

Low strong discrepancy forces $f$ to take all possible values on dense sets (see the discussion in \cite[Section~4.3]{beame2010separating}):
\begin{lemma}\label{lem:disc-implies-nonempty}
Let $f:Z\to B$ and $S\subseteq Z$. Then for every $y\in B$,
\begin{equation}\label{eq:fiber-lower}
|f^{-1}(y)\cap S|
\ \ge\
\frac{|S|}{|B|}-\Gamma(f,S)\,|Z|.
\end{equation}
In particular, if $\alpha >0$ is such that $|S|\ge \alpha |Z|$ and $\Gamma(f,S)<\alpha/|B|$, then $f^{-1}(y)\cap S\neq\varnothing$ for every $y\in B$.
\end{lemma}

\begin{proof}
Fix $y\in B$. By Definition~\ref{def:strong-disc} we have
$
\Bigl|\ |f^{-1}(y)\cap S|-\frac{|S|}{|B|}\ \Bigr|
\le \Gamma(f,S)\,|Z|$. 
Rearranging gives~\eqref{eq:fiber-lower}. If moreover $|S|\ge \alpha|Z|$ and $\Gamma(f,S)<\alpha/|B|$, then the right-hand side of~\eqref{eq:fiber-lower} is strictly positive, so $|f^{-1}(y)\cap S|>0$ for every $y \in B$.
\end{proof}

\subsection{Discrepancy of generalized inner products over finite fields}

\begin{defn}\cite[Definition 2.1]{babai1998cost}
	For a prime power $q$, positive integers $s,k$, the generalized inner product $\GIP_{q,s,k}: \left( \mathbb{F}^{s}_{q} \right)^k \to \mathbb{F}_q$ is defined for $x_1, \ldots, x_k \in \mathbb{F}^s_{q}$ via $\GIP_{q,s,k}(x_1, \ldots, x_k) = \sum_{i=1}^{s} x_{1,i} \cdot x_{2,i} \cdot  \ldots \cdot x_{k,i}$.
\end{defn}

\begin{fact}\cite{babai1998cost}\label{fac: GIP str discrep bound}
Given a prime power $q$ and positive integers $s,k$, let $X_1 = \ldots = X_k := \mathbb{F}_q^s$. Then for any cylinder intersection set $C \subseteq X := X_1 \times \ldots X_k$,
\begin{gather*}
	\Gamma(\GIP_{q,s,k}, C) \leq (1-1/q)\left(1 - (1-1/q)^{k-1} \right)^{s 2^{1-k}}.
\end{gather*}
	
\end{fact}
\begin{proof}
	This is contained in the proof of \cite[Corollary 4.12]{babai1998cost} (and pointed out in the case $q = 2^m$ in \cite[Fact 4.11]{beame2010separating}).
	
	Namely, let $\widetilde{X} := X_2 \times \ldots \times X_k$. The two claims in the proof of \cite[Corollary 4.12]{babai1998cost} combined give
	 $$\Pr_{u \in \widetilde{X}}[ \forall x \in X_1, \ \GIP_{q,s,k}(x,u) = 0]  = \left(1 - (1-1/q)^{k-1} \right)^s.$$
	 
	 Then \cite[Theorem 4.6]{babai1998cost} implies that for any cylinder intersection set $C \subseteq X$, 
	 $$\Gamma^{\textrm{weak}}(\GIP_{q,s,k}, C) \leq \left( \Pr_{u \in \widetilde{X}}[ \forall x \in X_1, \ \GIP_{q,s,k}(x,u) = 0] \right)^{2^{1-k}} = \left(1 - (1-\frac{1}{q})^{k-1} \right)^{s 2^{1-k}},$$
	 where $\Gamma^{\textrm{weak}}$ is defined in \cite[Definition 2.5]{babai1998cost}; 
	 hence by \cite[Lemma 2.9]{babai1998cost},
	 $$\Gamma(\GIP_{q,s,k}, C) \leq (1-1/q) \Gamma^{\textrm{weak}}(\GIP_{q,s,k}, C) \leq (1-1/q)\left(1 - (1-1/q)^{k-1} \right)^{s 2^{1-k}}. $$
\end{proof}

\begin{cor}\label{cor: k-sEH fails in fin fields}
Fix $k\ge 3$ and let $s:=2^{k}$. Then, for any $\alpha\in(0,1]$, if $
q>\frac{(k-1)^2}{\alpha}$ is a prime power and $C\subseteq (\F_q^{\,s})^k$ is a cylinder intersection set with $\abs{C}\ge \alpha\,q^{ks}$, then 
$\GIP_{q,s,k}$ assumes every value of $\F_q$ on $C$.
In particular, $\GIP_{q,s,k}$ has a zero on $C$ and also a nonzero on $C$.
\end{cor}

\begin{proof}
Let $Z:=(\F_q^{\,s})^k$, so $\abs{Z}=q^{ks}$ and $\abs{C}\ge \alpha\abs{Z}$.
By Fact \ref{fac: GIP str discrep bound} and the choice $s=2^k$,
\[
\Gamma(\GIP_{q,s,k},C)
\le
\left(1-\frac{1}{q}\right)\left(1-\left(1-\frac{1}{q}\right)^{k-1}\right)^{2}.
\]
Using Bernoulli's inequality $(1-x)^{k-1}\ge 1-(k-1)x$ for $x\in[0,1]$ with $x=1/q$ gives
$
1-\left(1-\frac{1}{q}\right)^{k-1}\ \le\ \frac{k-1}{q}
$, hence 
\[
\Gamma(\GIP_{q,s,k},C)
\le
\left(1-\frac{1}{q}\right)\left(\frac{k-1}{q}\right)^2
\le
\frac{(k-1)^2}{q^2}.
\]
If $q>(k-1)^2/\alpha$, then $\Gamma(\GIP_{q,s,k},C)<\alpha/q$.
Lemma~\ref{lem:disc-implies-nonempty} then implies that for every $y\in\F_q$,
$\GIP_{q,s,k}^{-1}(y)\cap C\neq \varnothing$.
In particular, $C$ contains a point $x$ where $\GIP_{q,s,k}(x)=0$ and a point $x'$ where $\GIP_{q,s,k}(x')=1$.
\end{proof}

\subsection{No theory satisfying the $k$-strong Erd\H{o}s-Hajnal property can define an infinite field of positive characteristic}

\begin{theorem}\label{thm: no inf fields n-dist}
	No theory satisfying the $k$-strong Erd\H{o}s-Hajnal property (for uniform finitely supported measures) can define an infinite field of positive characteristic. In particular (by Corollary \ref{cor: n-str EH}), any infinite field definable in a strongly $k$-distal NIP theory has characteristic $0$.
\end{theorem}

\begin{proof}

Assume $T$ is a theory and $(K,+,\cdot)$ is an infinite definable field of characteristic $p > 0$, $M \models T$.

Assume first that $T$ is not $k$-dependent, witnessed by a formula $\varphi(x_1, \ldots, x_{k+1}) \in L$. Then the $(k+1)$-ary relation defined by $\varphi$ on $M^{x_1} \times \ldots \times M^{x_{k+1}}$ does not satisfy $k$-sEH. Indeed, fix any $\alpha > 0$. For any finite $X_1, \ldots, X_{k+1}$ and $E \subseteq X := X_1 \times \ldots \times X_{k+1}$ there exist some $Y_i \subseteq M^{y_i}$ and bijections $f_i : X_i \to Y_i$ so that $(a_1, \ldots, a_{k+1}) \in E \Leftrightarrow M \models \varphi(f(a_1), \ldots, f(a_{k+1}))$ for all $a = (a_1, \ldots, a_{k+1}) \in X$. In particular, taking any $q > k^{2}/\alpha$ and $X_i := \mathbb{F}_q^{2^{k+1}}$, by Corollary \ref{cor: k-sEH fails in fin fields} and translating via the bijections $f_i$, for any cylinder intersection set $C \subseteq Y := Y_1 \times \ldots \times Y_{k+1}$ with $|C| \geq \alpha |Y|$, we have $M \models \varphi(b)$ and $M \models \neg \varphi(b')$ for some $b,b' \in C$.

So we may assume $T$ is $k$-dependent. Generalizing \cite[Corollary 4.5]{kaplan2011artin} for $k=1$, we then have: 
\begin{claim}
	$\mathbb{F}_p^{\alg}$ is a subfield of $K$.
\end{claim}
\begin{proof}
	Let $F := K \cap \mathbb{F}_p^{\alg}$, the relative algebraic closure of $\mathbb{F}_p$ in $K$. By \cite[Theorem 6.3]{hempel2016n} (generalizing \cite{kaplan2011artin} for $k=1$), the field $K$ is Artin-Schreier closed, hence so is $F$. Hence $F$ is infinite, perfect and PAC (pseudo-algebraically closed). But by \cite[Theorem 7.3]{hempel2016n} (generalizing \cite{duret2006corps} for $k=1$), any fields with a relatively algebraically closed PAC subfield which is not separably closed is not $k$-dependent. Hence $F$ is algebraically closed, i.e.~$F = \mathbb{F}_p^{\alg}$.
	 \end{proof}

 We let $s := 2^{k+1}$ and let the partitioned $L$-formula $\varphi(x_1, \ldots, x_{k+1})$, where  $x_i = (x_{i,1}, \ldots, x_{i,s})$ and $x_{i,j}$ a tuple of variables corresponding to elements of $K$, be 
 \begin{gather*}
 	\varphi(x_1, \ldots, x_{k+1}) := \left( \sum_{i = 1}^{s} x_{1,i} \cdot \ldots \cdot x_{k+1,i} = 0 \right).
 \end{gather*}
 
 We show that $\varphi$ does not satisfy $k$-sEH. Let $\alpha >0$ be arbitrary. Let $m$ be sufficiently large so that for $q := p^m$, $q > k^2/\alpha$. By the claim, $K$ contains $\mathbb{F}_q$ as a subfield. We let $X_1 = \ldots = X_{k+1} := \mathbb{F}_q^{2^{k+1}}$ (more precisely, the set of tuples of elements of $K$ corresponding to it). Then, by Corollary \ref{cor: k-sEH fails in fin fields}, for any cylinder intersection set $C \subseteq X := X_1 \times \ldots \times X_{k+1}$ with $|C| \geq \alpha |X|$, there exist some $a,a' \in C$ so that $M \models \varphi(a) \land \neg \varphi(a')$. 
\end{proof}

\begin{cor}
	The theory $\ACF_p$ of algebraically closed fields of characteristic $p>0$ is stable, but does not admit a strongly $n$-distal NIP expansion for any $n$.
\end{cor}

\bibliographystyle{plain}
\bibliography{ref}

\begin{thebibliography}{10}

\bibitem{zbMATH08064119}
A.~Abd~Aldaim, G.~Conant, and C.~Terry.
\newblock Higher arity stability and the functional order property.
\newblock {\em Sel. Math., New Ser.}, 31(3):79, 2025.
\newblock Id/No 59.

\bibitem{alon2007efficient}
Noga Alon, Eldar Fischer, and Ilan Newman.
\newblock Efficient testing of bipartite graphs for forbidden induced
  subgraphs.
\newblock {\em SIAM Journal on Computing}, 37(3):959--976, 2007.

\bibitem{alon2005crossing}
Noga Alon, J{\'a}nos Pach, Rom Pinchasi, Rado{\v{s}} Radoi{\v{c}}i{\'c}, and
  Micha Sharir.
\newblock Crossing patterns of semi-algebraic sets.
\newblock {\em Journal of Combinatorial Theory, Series A}, 111(2):310--326,
  2005.

\bibitem{anderson2023fuzzy}
Aaron Anderson.
\newblock Fuzzy {VC} combinatorics and distality in continuous logic.
\newblock {\em Preprint, arXiv:2310.04393}, 2023.

\bibitem{aschenbrenner2022distality}
Matthias Aschenbrenner, Artem Chernikov, Allen Gehret, and Martin Ziegler.
\newblock Distality in valued fields and related structures.
\newblock {\em Transactions of the American Mathematical Society},
  375(7):4641--4710, 2022.

\bibitem{babai1998cost}
L{\'a}szl{\'o} Babai, Thomas~P Hayes, and Peter~G Kimmel.
\newblock The cost of the missing bit: Communication complexity with help.
\newblock In {\em Proceedings of the thirtieth annual ACM symposium on Theory
  of computing}, pages 673--682, 1998.

\bibitem{basu2010combinatorial}
Saugata Basu.
\newblock Combinatorial complexity in o-minimal geometry.
\newblock {\em Proceedings of the London Mathematical Society},
  100(2):405--428, 2010.

\bibitem{bays2023incidence}
Martin Bays and Jean-Fran{\c{c}}ois Martin.
\newblock Incidence bounds in positive characteristic via valuations and
  distality.
\newblock {\em Annales Henri Lebesgue}, 6:627--641, 2023.

\bibitem{beame2010separating}
Paul Beame, Matei David, Toniann Pitassi, and Philipp Woelfel.
\newblock Separating deterministic from randomized multiparty communication
  complexity.
\newblock {\em Theory of Computing}, 6(1):201--225, 2010.

\bibitem{belegradek1999extended}
Oleg~V Belegradek, Alexei~P Stolboushkin, and Michael~A Taitslin.
\newblock Extended order-generic queries.
\newblock {\em Annals of Pure and Applied Logic}, 97(1-3):85--125, 1999.

\bibitem{ben2009continuous}
Ita{\"\i} Ben~Yaacov.
\newblock Continuous and random {V}apnik-{C}hervonenkis classes.
\newblock {\em Israel Journal of Mathematics}, 173(1):309--333, 2009.

\bibitem{benedikt2003definable}
Michael Benedikt, Leonid Libkin, Thomas Schwentick, and Luc Segoufin.
\newblock Definable relations and first-order query languages over strings.
\newblock {\em Journal of the ACM (JACM)}, 50(5):694--751, 2003.

\bibitem{boxall2018definable}
Gareth Boxall and Charlotte Kestner.
\newblock The definable (p, q)-theorem for distal theories.
\newblock {\em The Journal of Symbolic Logic}, 83(1):123--127, 2018.

\bibitem{boxall2023theories}
Gareth Boxall and Charlotte Kestner.
\newblock Theories with distal {S}helah expansions.
\newblock {\em The Journal of Symbolic Logic}, 88(4):1323--1333, 2023.

\bibitem{braunfeld2021characterizations}
Samuel Braunfeld and Michael Laskowski.
\newblock Characterizations of monadic {NIP}.
\newblock {\em Transactions of the American Mathematical Society, Series B},
  8(30):948--970, 2021.

\bibitem{braunfeld2024corrigenda}
Samuel Braunfeld and Michael Laskowski.
\newblock Corrigenda to ``{C}haracterizations of monadic {NIP}''.
\newblock {\em Transactions of the American Mathematical Society, Series B},
  11(34):1226--1232, 2024.

\bibitem{CheOber}
Artem Chernikov.
\newblock Towards higher classification theory.
\newblock In {\em Model Theory: Combinatorics, Groups, Valued Fields and
  Neostability. {Abstracts} from the workshop held {January} 8--14, 2023},
  volume~20 of {\em Oberwolfach Workshop Reports}, pages 129--134.
  Mathematisches Forschungsinstitut Oberwolfach, 2023.

\bibitem{chernikov2025externally}
Artem Chernikov.
\newblock Externally definable fsg groups in {NIP} theories.
\newblock {\em Preprint, arXiv:2506.23265}, 2025.

\bibitem{chernikov2020cutting}
Artem Chernikov, David Galvin, and Sergei Starchenko.
\newblock Cutting lemma and {Z}arankiewicz's problem in distal structures.
\newblock {\em Selecta Mathematica, New Series}, 26(2):8, 2020.

\bibitem{chernikov2022definable}
Artem Chernikov and Kyle Gannon.
\newblock Definable convolution and idempotent {K}eisler measures.
\newblock {\em Israel Journal of Mathematics}, 248(1):271--314, 2022.

\bibitem{chernikov2024definable}
Artem Chernikov, Kyle Gannon, and Krzysztof Krupi\'nski.
\newblock Definable convolution and idempotent {K}eisler measures {III}.
  {G}eneric stability, generic transitivity, and revised {N}ewelski's
  conjecture.
\newblock {\em Preprint, arXiv:2406.00912}, 2024.

\bibitem{chernikov2019mekler}
Artem Chernikov and Nadja Hempel.
\newblock Mekler's construction and generalized stability.
\newblock {\em Israel Journal of Mathematics}, 230(2):745--769, 2019.

\bibitem{chernikov2021n}
Artem Chernikov and Nadja Hempel.
\newblock On $n$-dependent groups and fields {II}.
\newblock {\em Forum of Mathematics, Sigma}, 9:e38, 2021.

\bibitem{chernikov2024n}
Artem Chernikov and Nadja Hempel.
\newblock On $n$-dependent groups and fields {III}. {M}ultilinear forms and
  invariant connected components.
\newblock {\em Preprint, arXiv:2412.19921}, 2024.

\bibitem{chernikov2012forking}
Artem Chernikov and Itay Kaplan.
\newblock Forking and dividing in {NTP2} theories.
\newblock {\em The Journal of Symbolic Logic}, 77(1):1--20, 2012.

\bibitem{chernikov2025semi}
Artem Chernikov and Alex Mennen.
\newblock Semi-equational theories.
\newblock {\em The Journal of Symbolic Logic}, 90(1):391--422, 2025.

\bibitem{chernikov2014n}
Artem Chernikov, Daniel Palacin, and Kota Takeuchi.
\newblock On n-dependence.
\newblock {\em Notre Dame Journal of Formal Logic, to appear}, 2019.

\bibitem{chernikov2024model}
Artem Chernikov, Ya'acov Peterzil, and Sergei Starchenko.
\newblock Model-theoretic {E}lekes--{S}zab{\'o} for stable and $o$-minimal
  hypergraphs.
\newblock {\em Duke Mathematical Journal}, 173(3):419--512, 2024.

\bibitem{chernikov2015externally}
Artem Chernikov and Pierre Simon.
\newblock Externally definable sets and dependent pairs {II}.
\newblock {\em Transactions of the American Mathematical Society},
  367(7):5217--5235, 2015.

\bibitem{chernikov2018regularity}
Artem Chernikov and Sergei Starchenko.
\newblock Regularity lemma for distal structures.
\newblock {\em Journal of the European Mathematical Society},
  20(10):2437--2466, 2018.

\bibitem{chernikov2021definable}
Artem Chernikov and Sergei Starchenko.
\newblock Definable regularity lemmas for {NIP} hypergraphs.
\newblock {\em The Quarterly Journal of Mathematics}, 72(4):1401--1433, 2021.

\bibitem{chernikov2021model}
Artem Chernikov and Sergei Starchenko.
\newblock Model-theoretic {E}lekes--{S}zab{\'o} in the strongly minimal case.
\newblock {\em Journal of Mathematical Logic}, 21(02):2150004, 2021.

\bibitem{chernikov2021ramsey}
Artem Chernikov, Sergei Starchenko, and Margaret~EM Thomas.
\newblock Ramsey growth in some {NIP} structures.
\newblock {\em Journal of the Institute of Mathematics of Jussieu},
  20(1):1--29, 2021.

\bibitem{chernikov2020hypergraph}
Artem Chernikov and Henry Towsner.
\newblock Hypergraph regularity and higher arity {VC}-dimension.
\newblock {\em Preprint, arXiv:2010.00726}, 2020.

\bibitem{chernikov2024perfect}
Artem Chernikov and Henry Towsner.
\newblock Perfect stable regularity lemma and slice-wise stable hypergraphs.
\newblock {\em Preprint, arXiv:2402.07870}, 2024.

\bibitem{arXiv:2508.05839}
Artem Chernikov and Henry Towsner.
\newblock Averages of hypergraphs and higher arity stability.
\newblock Preprint, {arXiv}:2508.05839, 2025.

\bibitem{chernikov2025higher}
Artem Chernikov and Henry Towsner.
\newblock Higher-arity {PAC} learning, {VC} dimension and packing lemma.
\newblock {\em Preprint, arXiv:2510.02420}, 2025.

\bibitem{conant2025generic}
Gabriel Conant, Kyle Gannon, and James~E Hanson.
\newblock Generic stability, randomizations and {NIP} formulas.
\newblock {\em Journal of Mathematical Logic}, page 2550016, 2025.

\bibitem{conant2021structure}
Gabriel Conant, Anand Pillay, and Caroline Terry.
\newblock Structure and regularity for subsets of groups with finite
  {VC}-dimension.
\newblock {\em Journal of the European Mathematical Society}, 24(2):583--621,
  2021.

\bibitem{CorMal2025}
Leonardo~N. Coregliano and Maryanthe Malliaris.
\newblock Sample completion, structured correlation, and {N}etflix problems.
\newblock {\em Preprint, arXiv:2509.20404}, 2025.

\bibitem{duret2006corps}
Jean-Louis Duret.
\newblock Les corps faiblement alg{\'e}briquement clos non s{\'e}parablement
  clos ont la propri{\'e}t{\'e} d'ind{\'e}pendance.
\newblock In {\em Model Theory of Algebra and Arithmetic: Proceedings of the
  Conference on Applications of Logic to Algebra and Arithmetic Held at
  Karpacz, Poland, September 1--7, 1979}, pages 136--162. Springer, 2006.

\bibitem{fox2012overlap}
Jacob Fox, Mikhail Gromov, Vincent Lafforgue, Assaf Naor, and J{\'a}nos Pach.
\newblock Overlap properties of geometric expanders.
\newblock {\em Journal f{\"u}r die Reine und Angewandte Mathematik}, 2012(671),
  2012.

\bibitem{fox2008erdHos}
Jacob Fox and J{\'a}nos Pach.
\newblock Erd{\H{o}}s-{H}ajnal-type results on intersection patterns of
  geometric objects.
\newblock In {\em Horizons of combinatorics}, pages 79--103. Springer, 2008.

\bibitem{fox2016polynomial}
Jacob Fox, J{\'a}nos Pach, and Andrew Suk.
\newblock A polynomial regularity lemma for semialgebraic hypergraphs and its
  applications in geometry and property testing.
\newblock {\em SIAM Journal on Computing}, 45(6):2199--2223, 2016.

\bibitem{fox2019erdHos}
Jacob Fox, J{\'a}nos Pach, and Andrew Suk.
\newblock Erd{\H{o}}s--{H}ajnal conjecture for graphs with bounded
  {VC}-dimension.
\newblock {\em Discrete \& Computational Geometry}, 61(4):809--829, 2019.

\bibitem{gishboliner2025regularity}
Lior Gishboliner, Asaf Shapira, and Yuval Wigderson.
\newblock Regularity for hypergraphs with bounded {VC}$_2$ dimension.
\newblock {\em Preprint, arXiv:2508.09969}, 2025.

\bibitem{goode1991some}
John~B Goode.
\newblock Some trivial considerations.
\newblock {\em The Journal of symbolic logic}, 56(2):624--631, 1991.

\bibitem{arXiv:2506.19147}
James~E. Hanson.
\newblock Indiscernible extraction at small large cardinals from a higher-arity
  stability notion.
\newblock Preprint, {arXiv}:2506.19147, 2025.

\bibitem{haussler1995sphere}
David Haussler.
\newblock Sphere packing numbers for subsets of the boolean $n$-cube with
  bounded {V}apnik-{C}hervonenkis dimension.
\newblock {\em Journal of Combinatorial Theory, Series A}, 69(2):217--232,
  1995.

\bibitem{hempel2016n}
Nadja Hempel.
\newblock On n-dependent groups and fields.
\newblock {\em Mathematical Logic Quarterly}, 62(3):215--224, 2016.

\bibitem{hieronymi2017distal}
Philipp Hieronymi and Travis Nell.
\newblock Distal and non-distal pairs.
\newblock {\em The Journal of Symbolic Logic}, 82(1):375--383, 2017.

\bibitem{hossain2022extension}
Akash Hossain.
\newblock Extension bases in {H}enselian valued fields.
\newblock {\em Preprint, arXiv:2210.01567}, 2022.

\bibitem{hrushovski2008groups}
Ehud Hrushovski, Ya'acov Peterzil, and Anand Pillay.
\newblock Groups, measures, and the {NIP}.
\newblock {\em Journal of the American Mathematical Society}, 21(2):563--596,
  2008.

\bibitem{hrushovski2011nip}
Ehud Hrushovski and Anand Pillay.
\newblock On {NIP} and invariant measures.
\newblock {\em Journal of the European Mathematical Society}, 13(4):1005--1061,
  2011.

\bibitem{hrushovski2013generically}
Ehud Hrushovski, Anand Pillay, and Pierre Simon.
\newblock Generically stable and smooth measures in {NIP} theories.
\newblock {\em Transactions of the American Mathematical Society},
  365(5):2341--2366, 2013.

\bibitem{kaplan2013additivity}
Itay Kaplan, Alf Onshuus, and Alexander Usvyatsov.
\newblock Additivity of the dp-rank.
\newblock {\em Transactions of the American Mathematical Society},
  365(11):5783--5804, 2013.

\bibitem{kaplan2011artin}
Itay Kaplan, Thomas Scanlon, and Frank~O Wagner.
\newblock Artin-{S}chreier extensions in {NIP} and simple fields.
\newblock {\em Israel Journal of Mathematics}, 185(1):141--153, 2011.

\bibitem{kaplan2017exact}
Itay Kaplan, Saharon Shelah, and Pierre Simon.
\newblock Exact saturation in simple and {NIP} theories.
\newblock {\em Journal of Mathematical Logic}, 17(01):1750001, 2017.

\bibitem{kaplan2014strict}
Itay Kaplan and Alexander Usvyatsov.
\newblock Strict independence.
\newblock {\em Journal of Mathematical Logic}, 14(02):1450008, 2014.

\bibitem{keisler1987measures}
H~Jerome Keisler.
\newblock Measures and forking.
\newblock {\em Annals of Pure and Applied Logic}, 34(2):119--169, 1987.

\bibitem{Kobayashi}
Munehiro Kobayashi.
\newblock A generalization of the {PAC} learning in product probability spaces.
\newblock {\em RIMS Kokyuroku (Proceedings of the workshop Model theoretic
  aspects of the notion of independence and dimension),
  http://hdl.handle.net/2433/223742}, 1938:33--37, 2015.

\bibitem{KotaPAC}
Takayuki Kuriyama and Kota Takeuchi.
\newblock On the ${PAC}_n$ learning.
\newblock {\em RIMS Kokyuroku (Proceedings of the workshop Model theoretic
  aspects of the notion of independence and dimension),
  http://hdl.handle.net/2433/223742}, 1938:54--58, 2015.

\bibitem{los1949extensions}
Jerzy {\L}o{\'s} and Edward Marczewski.
\newblock Extensions of measure.
\newblock {\em Fundamenta Mathematicae}, 36(1):267--276, 1949.

\bibitem{lovasz2010regularity}
L{\'a}szl{\'o} Lov{\'a}sz and Bal{\'a}zs Szegedy.
\newblock Regularity partitions and the topology of graphons.
\newblock In {\em An Irregular Mind: Szemer{\'e}di is 70}, pages 415--446.
  Springer, 2010.

\bibitem{okura2026distal}
Koki Okura.
\newblock Distal expansions of the integers and the $p$-adic fields.
\newblock {\em Preprint, arXiv:2603.19786}, 2026.

\bibitem{parigot1982theories}
Michel Parigot.
\newblock Th{\'e}ories d'arbres.
\newblock {\em The Journal of Symbolic Logic}, 47(4):841--853, 1982.

\bibitem{pillay1996geometric}
Anand Pillay.
\newblock {\em Geometric stability theory}.
\newblock Oxford University Press, 1996.

\bibitem{pillay2004type}
Anand Pillay.
\newblock Type-definability, compact {L}ie groups, and $o$-minimality.
\newblock {\em Journal of Mathematical Logic}, 4(02):147--162, 2004.

\bibitem{pillay2011generic}
Anand Pillay and Predrag Tanovic.
\newblock Generic stability, regularity, and quasiminimality.
\newblock {\em Models, logics, and higher-dimensional categories}, 53:189--211,
  2011.

\bibitem{sheats2025linear}
Hannah Sheats and Caroline Terry.
\newblock On the linear complexity of subsets of $\mathbb{F}_p^n$ of bounded
  {VC}$_2$-dimension.
\newblock {\em Preprint, arXiv:2512.02001}, 2025.

\bibitem{shelah2008minimal}
Saharon Shelah.
\newblock Minimal bounded index subgroup for dependent theories.
\newblock {\em Proceedings of the American Mathematical Society},
  136(3):1087--1091, 2008.

\bibitem{MR3273451}
Saharon Shelah.
\newblock Strongly dependent theories.
\newblock {\em Israel J. Math.}, 204(1):1--83, 2014.

\bibitem{MR3666349}
Saharon Shelah.
\newblock Definable groups for dependent and 2-dependent theories.
\newblock {\em Sarajevo J. Math.}, 13(25)(1):3--25, 2017.

\bibitem{simon2012finding}
Pierre Simon.
\newblock Finding generically stable measures.
\newblock {\em The Journal of Symbolic Logic}, 77(1):263--278, 2012.

\bibitem{simon2013distal}
Pierre Simon.
\newblock Distal and non-distal {NIP} theories.
\newblock {\em Annals of Pure and Applied Logic}, 164(3):294--318, 2013.

\bibitem{simon2015guide}
Pierre Simon.
\newblock {\em A guide to NIP theories}.
\newblock Cambridge University Press, 2015.

\bibitem{simon2016note}
Pierre Simon.
\newblock A note on ``{R}egularity lemma for distal structures''.
\newblock {\em Proceedings of the American Mathematical Society},
  144(8):3573--3578, 2016.

\bibitem{takeuchi}
Kota Takeuchi.
\newblock On 2-order property.
\newblock {\em Slides from a talk given at the Asian Logic Conference 2017,
  Daejeon, Korea}, 2017.

\bibitem{terry2025quadratic}
C~Terry and J~Wolf.
\newblock On the quadratic complexity of subsets of $\mathbb{F}_p^n$ of bounded
  {VC}$_2$-dimension.
\newblock {\em Preprint, arXiv:2510.12767}, 2025.

\bibitem{terry2025structure}
C~Terry and J~Wolf.
\newblock The structure of subsets of $\mathbb{F}_p^n$ of bounded
  {VC}$_2$-dimension.
\newblock {\em Preprint, arXiv:2510.12867}, 2025.

\bibitem{terry2023improved}
Caroline Terry.
\newblock An improved bound for regular decompositions of 3-uniform hypergraphs
  of bounded {VC}$_2$-dimension.
\newblock {\em Model Theory}, 2(2):325--356, 2023.

\bibitem{terry2024growth}
Caroline Terry.
\newblock Growth of regular partitions 4: strong regularity and the pairs
  partition.
\newblock {\em Preprint, arXiv:2404.02030}, 2024.

\bibitem{terry2021higher}
Caroline Terry and Julia Wolf.
\newblock Higher-order generalizations of stability and arithmetic regularity.
\newblock {\em Preprint, arXiv:2111.01739}, 2021.

\bibitem{terry2021irregular}
Caroline Terry and Julia Wolf.
\newblock Irregular triads in 3-uniform hypergraphs.
\newblock {\em Memoirs of the American Mathematical Society, accepted
  (arXiv:2111.01737)}, 2021.

\bibitem{TongThesis}
Ho~Wang~Mervyn Tong.
\newblock {\em Distality to and from combinatorics}.
\newblock PhD thesis, University of Leeds, 2025.

\bibitem{tong2025distal}
Mervyn Tong.
\newblock Distal expansions of {P}resburger arithmetic by a sparse predicate.
\newblock {\em The Journal of Symbolic Logic}, pages 1--33, 2025.

\bibitem{tong2026zarankiewicz}
Mervyn Tong.
\newblock Zarankiewicz bounds from distal regularity lemma.
\newblock {\em Bulletin of the London Mathematical Society}, 58(3):e70310,
  2026.

\bibitem{walker2023distality}
Roland Walker.
\newblock Distality rank.
\newblock {\em The Journal of Symbolic Logic}, 88(2):704--737, 2023.

\bibitem{WalkerThesis}
Roland Walker.
\newblock {\em Distality Rank and Tree Dimension}.
\newblock PhD thesis, University of Illinois at Chicago, 2023.

\bibitem{WestheadThesis}
Francis~Joseph Westhead.
\newblock Towards a regularity lemma for higher arity distal structures.
\newblock Master's thesis, University of Maryland, College Park, 2025.

\bibitem{yaacov2008model}
I~Ben Yaacov, Alexander Berenstein, C~Ward Henson, and Alexander Usvyatsov.
\newblock Model theory for metric structures.
\newblock {\em London Mathematical Society Lecture Note Series}, 350:315, 2008.

\bibitem{yaacov2014independence}
Ita{\"\i}~Ben Yaacov and Artem Chernikov.
\newblock An independence theorem for {NTP2} theories.
\newblock {\em The Journal of Symbolic Logic}, 79(1):135--153, 2014.

\end{thebibliography}

\end{document}